\documentclass{TPmod}
\DeclareMathAlphabet{\mathpzc}{OT1}{pzc}{m}{it}
\usepackage{url}
\usepackage{setspace}
\usepackage{scrextend}
\usepackage{mathrsfs}
\usepackage{tocloft}
\usepackage{empheq}

\makeatletter
\DeclareRobustCommand{\hvec}[1]{{\mathpalette\hvec@{#1}}}
\newcommand{\hvec@}[2]{
  \vbox{\offinterlineskip
    \ialign{
      \hfil##\hfil\cr
      $\m@th#1{}_{\rightharpoonup}$\kern-\scriptspace\cr
      $\m@th#1#2$\cr
    }
  }
}
\makeatother

\usepackage{graphicx,calc}
\newlength\myheight
\newlength\mydepth
\settototalheight\myheight{Xygp}
\usepackage{amsthm}
\usepackage{amssymb}
\usepackage[capitalize]{cleveref}
\usepackage{mathtools}
\usepackage{xy}
\usepackage{tikz}
\usepackage{tikz-cd}

\usepackage{cancel}

\usepackage[normalem]{ulem}
\usetikzlibrary{shapes.geometric,arrows}
\usetikzlibrary{calc, arrows,decorations.pathmorphing,backgrounds,fit,positioning,shapes.symbols,chains,snakes}
\usetikzlibrary{decorations.pathreplacing,decorations.markings}
\usetikzlibrary{arrows.meta}

\newcommand{\Ind}{\mathrm{Ind}}

\usepackage{xcolor}

\DeclareMathOperator{\Symp}{Symp}

\def\bC{\mathbb{C}}
\def\bD{\mathbb{D}}
\def\bE{\mathbb{E}}
\def\bF{\mathbb{F}}

\def\bH{\mathbb{H}}

\def\bP{\mathbb{P}}

\def\bR{\mathbb{R}}
\def\bS{\mathbb{S}}
\def\bT{\mathbb{T}}
\def\bU{\mathbb{U}}
\def\bV{\mathbb{V}}

\def\bZ{\mathbb{Z}}

\newcommand{\eps}{\varepsilon}

\newcommand{\cA}{\mathcal{A}}
\newcommand{\cB}{\mathcal{B}}
\newcommand{\cC}{\mathcal{C}}

\newcommand{\cE}{\mathcal{E}}
\newcommand{\cF}{\mathcal{F}}
\newcommand{\cG}{\mathcal{G}}
\newcommand{\cH}{\mathcal{H}}
\newcommand{\cI}{\mathcal{I}}
\newcommand{\cJ}{\mathcal{J}}
\newcommand{\cK}{\mathcal{K}}
\newcommand{\cL}{\mathcal{L}}
\newcommand{\cM}{\mathcal{M}}

\newcommand{\cO}{\mathcal{O}}
\newcommand{\cP}{\mathcal{P}}
\newcommand{\cQ}{\mathcal{Q}}
\newcommand{\cR}{\mathcal{R}}
\newcommand{\cS}{\mathcal{S}}
\newcommand{\cT}{\mathcal{T}}

\newcommand{\cV}{\mathcal{V}}
\newcommand{\cW}{\mathcal{W}}
\newcommand{\cX}{\mathcal{X}}
\newcommand{\cY}{\mathcal{Y}}
\newcommand{\cZ}{\mathcal{Z}}

\def\bZ{\mathbb{Z}}

\renewcommand{\cong}{\simeq}

\begin{document}

\title{Spectral Floer theory and tangential structures}

\author[]{Noah Porcelli, Ivan Smith}
\thanks{The Clynelish distillery.}
\date{September 2024. Unauthorised readers will be pestered by pufflings.}

\address{Noah Porcelli,  Imperial College Department of Mathematics, Huxley Building, 180 Queen's Gate, London SW7 2RH, U.K.}
\address{Ivan Smith, Centre for Mathematical Sciences, University of Cambridge, Wilberforce Road, Cambridge CB3 0WB, U.K.}

\begin{abstract} {\sc Abstract:} 
In \cite{PS}, for a stably framed Liouville manifold $X$ we defined a Donaldson-Fukaya category $\scrF(X;\bS)$  over the sphere spectrum, and developed an obstruction theory for lifting quasi-isomorphisms from $\scrF(X;\bZ)$ to $\scrF(X;\bS)$. Here, we define a spectral Donaldson-Fukaya category for any  `graded tangential pair' $\Theta \to \Phi$ of spaces living over $BO \to BU$, whose objects are Lagrangians $L\to X$ for which the classifying maps of their tangent bundles lift to $\Theta \to \Phi$. The previous case corresponded to $\Theta = \Phi = \{\mathrm{pt}\}$. We extend our obstruction theory to this setting. The flexibility to `tune' the choice of $\Theta$ and $\Phi$ increases the range of cases in which one can kill the obstructions, with applications to bordism classes of Lagrangian embeddings in the corresponding bordism theory $\Omega^{(\Theta,\Phi),\circ}_*$. We include a self-contained discussion of when (exact) spectral Floer theory over a ring spectrum $R$ should exist, which may be of independent interest.
\end{abstract}

\maketitle
\thispagestyle{empty}
\setlength{\cftbeforesubsecskip}{-2pt}
\tableofcontents

\section{Introduction}

\subsection{Setting}
Let $X$ be a Liouville manifold with $2c_1(X) = 0$. In \cite{PS},  assuming furthermore that $X$ admits a stable framing, we adapted work of Large \cite{Large} to the setting of flow modules to construct a `Donaldson-Fukaya category over the sphere spectrum' $\scrF(X;\bS)$. The  objects of this category were compact exact Lagrangian submanifolds whose stable Gauss maps $L \to U/O$ were nullhomotopic compatibly with the ambient stable framing. We developed an obstruction theory for lifting quasi-isomorphisms from $\scrF(X;\bZ)$ to $\scrF(X;\bS)$, and gave applications in settings where the obstructions vanished for degree reasons, or could be killed by passing to `Baas-Sullivan truncations', akin to working over quotients $\bS/I$ of the sphere spectrum by an ideal.

This paper has three goals:
\begin{enumerate}
    \item we give a high-level and sometimes conjectural description of when spectral Floer theory over a ring spectrum $R$ should exist for a given symplectic manifold $X$, in terms of \emph{tangential pairs}, including a number of `universal' cases. This is carried out in Section \ref{sec: flav}, and is intended to be readable independent of the rest of the paper;
    \item for a graded tangential pair $\Theta \to \Phi$ over $BO \to BU$, we build a spectral Donaldson-Fukaya category $\scrF(X;(\Theta, \Phi))$, whose objects are exact Lagrangians $L \to X$ for which the classifying maps of the tangent bundle lift compatibly to the given tangential pair;
    \item we extend the obstruction theory from \cite{PS} to this more general setting, giving a sufficient condition for a quasi-isomorphism in $\scrF(X;\bZ)$ to lift to $\scrF(X;(\Theta,\Phi))$, and explain the consequences for bordism classes in that case. 
\end{enumerate}

Special cases of the construction give a category $\scrF(X;\bS)$ over the sphere spectrum when $X$ admits a stable polarisation, i.e. when
\[
[TX] \in \mathrm{image}(KO(X) \stackrel{c}{\longrightarrow} KU(X))
\]
 lies in the image of the natural complexification map; and a category $\scrF(X;MU)$ for any Liouville $X$ with $2c_1(X)=0$.

 We work throughout with exact Lagrangians in a Liouville manifold which are either compact or geometrically bounded, in which case we impose `small wrapping' at infinity.

\subsection{Results}

In Section \ref{sec: Lag Floer conj} we consider the notion of a \emph{tangential pair}. This consists of a pair of spaces $(\Theta, \Phi)$ fitting into a commutative diagram:
\begin{equation}\label{eq: strong intro}
            \xymatrix{
                \Theta 
                \ar[r]
                \ar[d]
                &
                \Phi
                \ar[d]
                \\
                BO 
                \ar[r]
                &
                BU
            }
        \end{equation}
This provides a convenient and indeed universal language in which to phrase various types of condition on the stable tangent bundle of $X$, and on the Lagrangian Gauss map of $L$. 

An \emph{oriented} tangential pair is one where $\Theta$ lives over $B\operatorname{Pin}$ instead of $BO$, corresponding to tangential pairs where Floer theory exists away from characteristic two (in particular in characteristic zero). A \emph{graded} tangential one is one where $\Phi$ lives over $BS_\pm U$, the fibre of the map $BU \to K(\bZ, 2)$ classifying $2c_1$, corresponding to where $\bZ$-graded Floer theory exists.

We let $F$ be the homotopy fibre of the map $\Theta \to \Phi$. There is a natural map 
\[
\Omega F \to BO\times \bZ
\]
and a bordism theory\footnote{Following very established conventions, $\Omega$ denotes both based loops and a bordism theory; we hope no confusion will arise.} $\Omega^{(\Theta,\Phi)}_*$ associated to the resulting Thom spectrum of the pullback of the universal bundle. There is a related bordism theory $\Omega^{(\Theta,\Phi),\circ}_*$, which we call `capsule bordism', corresponding to cobordism of manifolds with stable tangent bundle pulled back from the \emph{universal index bundle} over the homotopy limit of:
\begin{equation}\label{eq: above}
    \xymatrix{
        &
        \Phi \ar[d]
        \\
        \cL \Theta \ar[r]
        &
        \cL \Phi
    }
\end{equation} 
where $\cL$ denotes the free loop space. This homotopy limit maps to the homotopy limit of (\ref{eq: above}) in the universal case $(\Theta,\Phi)=(BO,BU)$ which is $BO \times \Omega U/O$. Real Bott periodicity implies that $\Omega U/O \simeq BO \times \bZ$ and so it carries a canonical virtual vector bundle; the universal index bundle over $BO \times \Omega U/O$ is then the sum of the virtual vector bundles on each factor. 
\begin{rmk}
    We give a more explicit heuristic description. A point in this space consists roughly of a pair $(E, F)$, where $E \to D^2$ is a complex vector bundle and $F \to \partial D^2$ is a totally real subbundle, along with a $\Phi$-structure and $\Theta$-structure on $E$ and $F$ respectively. The data of $E$ and $F$ (along with some contractible choices, such as an $E$-valued 1-form, cf. Definition \ref{defn:CR data}) determines a \emph{Cauchy-Riemann operator}, a Fredholm linear map; the universal index bundle is roughly its kernel minus cokernel. Restricting to the fibre $F_1$ of $F$ over $1 \in \partial D^2$ determines a point in $BO$. The virtual bundle $E - (F_1 \otimes \bC)$ over $D^2$ is canonically trivial, and the virtual bundle $F-F_1$ over $\partial D^2$ determines a loop $\gamma$ in $\Omega U/O$. The index bundle splits into two terms, one corresponding to $F_1$, and one to the index of the Cauchy-Riemann operator corresponding to $\gamma$, cf. \cite[Section 3]{P}.
\end{rmk}

If the classifying map $X \to BU$ for $TX$ lifts to $\Phi$, we will say that $X$ admits a $\Phi$-structure. 
In that case, a $\Theta$-brane is an exact Lagrangian $L$ for which the classifying map $TL \to BO$ lift to $\Theta$ compatibly with the $\Phi$-structure on $X$.

\begin{thm}\label{thm:main}
    Let $X$ be a Liouville manifold of real dimension $2d$ with $2c_1(X)=0$ and let $(\Theta,\Phi)$ be an oriented graded tangential pair. Assume that $X$ admits a $\Phi$-structure $f: X \to \Phi$. There is a spectral Donaldson-Fukaya category $\scrF(X;(\Theta,\Phi))$ whose objects are geometrically bounded exact Lagrangian $\Theta$-branes.  If $L$ and $K$ are quasi-isomorphic $\Theta$-branes, then $[L] = [K] \in \Omega_d^{(\Theta,\Phi),\circ}(X, f)$. 
\end{thm}

The precise description of the capsule bordism theory on a pair $(X,f)$ is described in Section \ref{sec: rel flow mod}, along with conditions under which ordinary bordism $\Omega^{(\Theta,\Phi)}_*$ splits off from it. (Section \ref{Sec:proper bordism} includes a brief discussion of proper bordism, for the case when the Lagrangians are non-compact.)  Whenever $\Phi = \{pt\}$ we will typically drop it from the notation, and just write $F = \Theta \to U/O$.  

\begin{rmk}
    For a commutative ring spectrum $R$, when $X$ is stably polarised an $R$-linear Fukaya category has been constructed in \cite{ADP}. The objects are Lagrangians with $\Theta$-structure (for appropriate $\Theta$ depending on $R$), but we caution that these categories are not isomorphic: we expect them to be related by base-change, similarly to (\ref{eq: base-change}) below.
\end{rmk}

\begin{rmk}
    We expect the morphism groups in $\scrF(X;(\Theta,\Phi))$ to be modules over the bordism group $\Omega_*^{\Phi}$, cf. Remark \ref{rmk:module action}, but have not spelled out the details.
\end{rmk}

\begin{rmk}
    The appearance of the theory $\Omega^{(\Theta,\Phi),\circ}$ reflects the fact that the bordism theory may be arising from a non-commutative ring spectrum. The fundamental class of a Lagrangian is defined by an open-closed map, which is naturally a trace, and so valued in a version of topological Hochschild homology, which is related to the free loop space by \cite{BCS}. (This issue does not arise in the setting of \cite{PS}, since $THH(\bS) \simeq \bS$.) 
\end{rmk}

\begin{example} 
If $\Theta = \{\ast\}$ then $\Omega^{\Theta} = \Omega^{fr}$ is framed bordism, with underlying spectrum $\bS$.
\end{example}

\begin{example} 
Let $(\widetilde{U/O})^{or}$ denote the universal cover of the homotopy fibre  of the map $U/O \to K(\bZ/2,2)$ which is the composition of the map $U/O \to BO$ classifying the underlying vector bundle of a totally real subbundle of $\bC^n$, followed by the map $BO \to K(\bZ/2,2)$ classifying the second Stiefel-Whitney class.  If $\Theta = (\widetilde{U/O})^{or}$ then $\Omega^{\Theta} = \Omega^{SO}$ is oriented bordism, with underlying spectrum $MSO$.
\end{example}

\begin{example} 
If $\Theta$ is the pullback of the universal covering $\widetilde{U/O} \to U/O$ under $U\to U/O$,  then $\Omega^{\Theta} = \Omega^{U}$ is complex cobordism, with underlying spectrum $MU$.
\end{example}

\begin{example}
    If $R$ is any $\bE_2$ ring spectrum, there is a map $BO \to BGL_1(\bS) \to BGL_1(R)$, where $GL_1(\cdot)$ denotes the group of units of a ring spectrum, the first map is a version of the $J$ homomorphism and the second is induced by the unit. This map is naturally an $\bE_1$-map (using the `real' version of work of B.~Harris \cite{Harris} to see that the Bott periodicity map is $\bE_1$), so one can deloop to get a map $U/O \to B^2GL_1(R)$, whose homotopy fibre provides a natural source of spaces over $U/O$. Taking $R=\bS$ itself gives rise to the bordism theory of manifolds with trivialisations of the stable spherical fibration associated to the tangent bundle. 
\end{example} 

To make this useful, one would like to be able to effectively detect quasi-isomorphism in $\scrF(X;(\Theta,\Phi))$. We extend the obstruction theory from \cite{PS} to the more general set-up, with the following consequence:

\begin{thm}\label{thm: fuk ob}
In the setting of Theorem \ref{thm:main}, suppose ($X$ has a $\Phi$-structure and) $L$ and $K$ are $\Theta$-branes which are quasi-isomorphic over $\bZ$. There is a series of obstructions $[\mathcal{O}_i] \in H^i(L;\Omega_{i-1}^{(\Theta,\Phi)})$ with the property that if all such vanish, then one can lift the quasi-isomorphism to $\scrF(X;(\Theta,\Phi))$.
\end{thm}

\subsection{Applications}

\noindent \textbf{Complex cobordism.} Consider for instance the case where (i) $(\Theta,\Phi)$ is the pullback of the pair  ($B\operatorname{Pin}, B\operatorname{Pin} \times BU)$ under the map $BS_\pm U \to BU$, (ii) the map $\Theta \to \Phi$ is $(\mathrm{identity}, \mathrm{complexification})$, and (iii) $\Phi \to BS_\pm U$ forgets the first factor. Then $\Omega^{(\Theta,\Phi)} = MU$ is complex cobordism, so $\Omega_*^{(\Theta,\Phi)}$ is concentrated in even degrees.  We obtain:

\begin{cor}\label{cor:1}
    Let $X$ be Liouville with $2c_1(X)=0$ and suppose  $L,K$ are $\Theta$-branes in $X$. More explicitly, this means there is a real vector bundle $E \to X$ of rank $k$ such that there are stable isomorphisms $E|_L \simeq TL$ and $E|_K \simeq TK$.
    
    Suppose $H^*(L;\bZ)$ is concentrated in even degrees.  If $K$ is quasi-isomorphic to $L$ over $\bZ$, then $[L]=[K] \in MU_{d-k}(\mathrm{Thom}(-E \to X))$; more concretely, there is a cobordism from $L$ to $K$ over $X$ equipped with a stable complex structure relative to $E$.
\end{cor}

In the setting of Corollary \ref{cor:1}, suppose $L$ and $K$ are exact oriented spin Lagrangian cobordant. Then they are quasi-isomorphic over $\bZ$, by \cite{Biran-Cornea}, so in fact they represent the same class in $MU_{d-k}(\mathrm{Thom}(-E \to X))$, even though the given Lagrangian cobordism is not assumed to have any tangential structure beyond being oriented and spin. 

\begin{example}
In general the natural map $MU_*(X) \otimes_{MU_*(pt)} \bZ \to H_d(X;\bZ)$ is neither injective nor surjective. For instance, 
    Totaro \cite{Totaro} observed that the image of $MU_*(X) \to H_*(X;\bZ)$ lies in the subset of classes annihilated by all odd-degree Bockstein homomorphisms, and consequently gave a number of examples of torsion homology classes on algebraic varieties which do not lift to $MU$.
\end{example}

\noindent \textbf{Immersed Lagrangian cobordism.}
As recalled above, exact Lagrangian submanifolds which are exact Lagrangian cobordant define quasi-isomorphic objects of the Fukaya category \cite{Biran-Cornea}.  Indeed, that result holds if the Lagrangians are related by an unobstructed but possibly immersed Lagrangian cobordism. It is natural to wonder to what extent the converse holds, i.e. when one can construct a geometric immersed cobordism from Fukaya-categorical hypotheses on quasi-isomorphism. 

Consider the infinite complex Lagrangian Grassmannian $Sp/U$. This admits a natural map to $U/O$ by viewing a complex Lagrangian subspace of $\bH^N$ as a real subspace. Since $Sp/U$ is simply connected, this canonically lifts to $\widetilde{U/O}$.

\begin{cor} \label{cor:immersed general}
    Suppose $X$ is stably framed and $L,K$ admit $\Theta$-brane structures for $\Theta = Sp/U \to \widetilde{U/O}$. If $L$ and $K$ are quasi-isomorphic in $\scrF(X;\Theta)$ then they are immersed oriented Lagrangian cobordant.
\end{cor}

The coefficients groups relevant to oriented immersed cobordism are rather complicated \cite{Audin}, so it is hard to infer vanishing of obstructions for degree reasons.  The coefficient groups for the bordism theory relevant to the case of not-necessarily-orientable cobordisms are much simpler. This yields:

\begin{cor}\label{cor:2}
    Suppose $X$ is stably framed and $L$ and $K$ are Lagrangian homotopy spheres of dimension $d \in \{2,4,6,7,8,12\}$.
    
    If $L$ and $K$ are quasi-isomorphic in $\scrF(X;\bZ)$ then they are unoriented Lagrangian cobordant. 
\end{cor}

Finding a bounding cochain on such a cobordism, to make it unobstructed, seems likely to need essentially different techniques.

In future work, we incorporate spectral local systems into our version of the Fukaya category, with further applications to the topology of Lagrangian submanifolds.

A quite different application of Floer homotopy theory on polarised manifolds to lower bounds on numbers of Lagrangian intersection points has recently been given by \cite{Blakey, Blakey-Bonciocat}.
\medskip

\noindent \textbf{\emph{Acknowledgements.}} N.P. is supported by EPSRC standard grant EP/W015889/1.  I.S. is partially supported by UKRI Frontier Research grant EP/X030660/1 (in lieu of an ERC Advanced grant). We are grateful to Mohammed Abouzaid, Ciprian Bonciocat, Alice Hedenlund, Jeff Hicks, Thomas Kragh, Tim Large and Alex Oancea for helpful discussions. We are grateful to the anonymous referee for helpful comments and corrections.

\section{Flavours of spectral Fukaya category}\label{sec: flav}

    In this section, we discuss existence of Fukaya categories over various ring spectra. We motivate and conjecture when these should exist, and explain how they relate to other constructions or conjectures in the literature. We work with the minimum possible commutativity assumptions on the ring spectra we work over: as well as providing generality, this is the natural setting for the ``universal'' cases discussed in Example \ref{ex: univ lag} and elsewhere. Furthermore working without extra commutativity data is convenient in practice in Section \ref{sec: Thet flow}.

    Notable omissions from this discussion include (but are not limited to) $A_\infty$ structures, the closed-open map, bubbling/curvature and local systems.

    Section \ref{sec: flav} is largely conjectural; we work under the following (imprecise) assumption:
    \begin{asmp}\label{asmp}
        Different models for Bott periodicity (including those of Section \ref{sec: abstr dis} and \cite{Harris}) agree and can be made to respect all multiplicative structure; in particular, the pairs of ring spectra (both called $(A,R)$) constructed in Sections \ref{sec: Thom} and \ref{sec: flav} are equivalent.
    \end{asmp}
    In later sections we will also sometimes refer to this, but it will \emph{not} be used in our topological applications.
    
    Throughout this paper, two spaces we encounter repeatedly are $BO \times \bZ$ and $BU \times \bZ$, the classifying spaces for real and complex virtual vector bundles. These are both infinite loop spaces in a natural way, essentially inherited from the symmetric monoidal structure on the category of vector spaces; this is also the infinite loop space structure one obtains from Bott periodicity. However, this infinite loop space structure does \emph{not} respect the product structure: it is not the one obtained from $BO$ (or $BU$) and $\bZ$ and taking the product. To avoid ambiguity, we declare:
    \begin{conv}\label{convention}
        Throughout this paper, $BO \times \bZ$ and $BU \times \bZ$ are equipped with the infinite loop structure coming from Bott periodicity (i.e. not the product one).
    \end{conv}
    
    \begin{rmk}
        $BO \times \bZ$ has the product $\bE_1$ but not $\bE_2$-structure. $BU \times \bZ$ has the product $\bE_2$-structure.
    \end{rmk}

\subsection{Flow categories and stable homotopy theory}\label{sec: 2.1}

    Recall from \cite{CJS} that a \emph{flow category} is a topologically enriched category with morphism spaces compact manifolds with corners, such that their boundary is covered by the images of the composition maps (see Section \ref{sec: back} for a precise definition).

    These can be assembled into a category $\operatorname{Flow}$ with objects flow categories. One can build a similar version by considering manifolds equipped with (compatible) stable framings or stable complex structures, as is carried out in \cite{AB2}. 

    Generalising this, given any map of loop spaces $f: \Omega F \to BO \times \bZ$ for a space $F$, one can construct a category of flow categories using manifolds whose stable tangent bundle admits a lift to $\Omega F$; we call such a lift an \emph{$\Omega F$-structure}. The case of framings is when $F$ is a point and the case of complex structures is when $\Omega F \simeq BU$. 
    
    Taking the Thom spectrum of the universal bundle pulled back along $f$ produces a spectrum $R := \operatorname{Thom}(\Omega F \to BO \times \bZ)$; since $f$ was a map of loop spaces, $R$ admits the structure of a unital ring spectrum \cite{Lewis-May-Steinberger-McClure}.

    \begin{conj}\label{conj: flow thom}
        There is a stable $\infty$-category $\operatorname{Flow^F}$ of flow categories built out of manifolds with $\Omega F$-structure. 
        
        This category is equivalent to the $\infty$-category of (left) $R$-modules.

        In the case where $R$ is associative but not commutative, there are analagous versions of $\operatorname{Flow}^F$ corresponding to the categories of right modules and bimodules over $R$.
    \end{conj}

    In the framed case, $R= \bS$ is the sphere spectrum, and this was proved in \cite[Section 8]{AB2}. In the general case, we construct (conjectural) flow-categorical models for left/right/bimodules over $R$ in Section \ref{sec: Thet flow}.

\subsection{Symplectic cohomology}\label{sec: conj SH}
    We begin by discussing when symplectic cohomology should be defined over what ring spectrum, before moving on to Lagrangians. Let $X$ be a Liouville domain. Its stable tangent bundle is classified by a map $TX: X \to BU$. 

    Let $H: X \times S^1 \to \bR$ be a generic Hamiltonian, linear at infinity (with slope not equal to the length of any Reeb chord). As shown in \cite{Large}, there is a flow category $\cM^H$ whose objects are Hamiltonian chords and morphisms are Floer trajectories $\cM^H_{xy}$. After making some choices, each $\cM^H_{xy}$ admits a map to the double loop space $\Omega^2 X$, such that the composition
    \begin{equation}\label{eq: class Ham}
        \cM^H_{xy} \to \Omega^2 X \xrightarrow{\Omega^2 TX} \Omega^2 BU \xrightarrow[\simeq]{\textrm{Bott periodicity}} BU \times \bZ
    \end{equation}
    is a classifying map for the stable tangent bundle for $\cM^H_{xy}$; in particular, $\cM^H_{xy}$ admits a natural stable complex structure (existence of stable complex structures was proved in \cite[Appendix B]{AB2}).

    Let $\Phi \to BU$ be some based space over $BU$. Assume $X$ is equipped with a \emph{$\Phi$-structure}, meaning the classifying map for $TX$ is given a lift to $\Phi$. Then by (\ref{eq: class Ham}), the classifying maps for each $T\cM^H_{xy}$ admit lifts to $\Omega^2 \Phi$.

    The upshot of this discussion is that the $\Phi$-structure on $X$ induces a $\Omega^2 \Phi$-structure on all moduli spaces $\cM^H_{xy}$, and so we expect the flow category $\cM^H$ to admit an induced $\Omega^2 \Phi$-structure in the sense of Section \ref{sec: 2.1}.
    
    \begin{rmk}
        Since $\Omega^2 \Phi$ is a double loop space, the ring spectrum $A := \operatorname{Thom}(\Omega^2 \Phi \to BU \times \bZ)$ admits the structure of an $\bE_2$ ring spectrum \cite{Ando-Blumberg-Gepner}.
    \end{rmk}

    Combined with Conjecture \ref{conj: flow thom}, we conclude:

    \begin{conj}\label{conj: SH}
        If the classifying map $X \to BU$ factors through $\Phi \to BU$, then Floer theory of a symplectomorphism $\phi$ preserving the $\Phi$-structure (including Hamiltonian Floer theory and symplectic cohomology) can be defined as a module over the ring spectrum $A := \operatorname{Thom}(\Omega^2 \Phi \to BU \times \bZ)$. We write $SH(\phi; A)$ or $SH(X; A)$ for this $A$-module.
    \end{conj}

    \begin{example}
        If $TX$ is stably trivial, this corresponds to a $\Phi$-structure where $\Phi=\operatorname{pt}$. Then $A = \bS$ is equivalent to the sphere spectrum, and so we would obtain symplectic cohomology over the sphere spectrum.
    \end{example}

    \begin{example}[Universal general case] \label{ex: univ sh}
        The condition in Conjecture \ref{conj: SH} always holds when $\Phi=BU$ and $\Phi \to BU$ is the identity. In this case, $A$ is the \emph{periodic complex cobordism} ring spectrum $MUP$. This $A$ is the universal coefficient ring that symplectic cohomology should be defined over, independent of $X$ and without additional assumptions on $TX$.
    \end{example}

    \begin{example}[Gradings]
        Let $\Phi$ be the homotopy fibre of the map $BU \to K(\bZ;2)$ classifying two times the universal first Chern class. 
        
        A $\Phi$-structure on $X$ induces a \emph{grading} on $X$ in the sense of \cite{Seidel:graded}. In this case, $\Omega^2\Phi \simeq BU$ and $A=MU$ is the \emph{complex cobordism} ring spectrum.
    \end{example}

    \begin{example}[Polarisations]\label{ex: polar}
        A \emph{stable polarisation} of $X$ is a $\Phi$-structure where $\Phi=BO \to BU$ is the complexification map. 

        A choice of stable polarisation is essentially the same data as a section of the stable Lagrangian Grassmannian over $X$. In that case, if $L\subset X$ is a Lagrangian submanifold, then its stable Gauss map is naturally valued in $U/O$. 
        Stable polarisations exist when $X = T^*Q$ is a cotangent bundle, or is a positively arborealisble Weinstein domain \cite{Alvarez-Gavela-Eliashberg-Nadler}.

        The composition $\Omega^2\Phi \simeq \Omega O \to BU \times \bZ \to BO \times \bZ$ is nullhomotopic (as a map of infinite loop spaces, recalling Convention \ref{convention}), so the stable vector bundle classified by it is trivial; therefore the ring spectrum $A = \operatorname{Thom}(\Omega O \to BO \times \bZ)$ is a suspension spectrum: $A \simeq \Sigma^\infty_+ \Omega O$.

        Thus Conjecture \ref{conj: SH} says that symplectic cohomology for polarised $X$ should be defined over the ring spectrum $\Sigma^\infty_+ \Omega O$: this is $SH(X; \Sigma^\infty_+\Omega O)$.

        For stably polarised $X$, symplectic cohomology exists over the sphere spectrum $SH^*(X; \bS)$ (see \cite{KraghViterbo}; this can also be obtained from minor modifications to \cite{Large}). There is a map of ring spectra $\Sigma^\infty_+ \Omega O \to \bS$; we expect that the former determines the latter via change-of-rings:
        \begin{equation}\label{eq: chan ring}
            SH(X; \Sigma^\infty_+ \Omega O ) \otimes_{\Sigma^\infty_+ \Omega O } \bS \simeq SH(X; \bS)
        \end{equation}
        but \emph{not} the other way around. In particular, even if Floer theory over the sphere spectrum exists, it does not determine all other forms of Floer theory.
        \end{example}

        \begin{rmk}
             If $X$ is stably polarised, then $TX \oplus E_\bC \cong \bC^N$ for some real vector bundle $E \to \bC$, and the symplectic manifold $\operatorname{Tot}(E_\bC)$ is stably framed. A disc bundle neighbourhood of the zero-section in $\operatorname{Tot}(E_{\bC})$ admits a Liouville structure \cite{Avdek}, and so its symplectic cohomology is defined over $\bS$. Passing from $X$ to such a disc bundle in $\operatorname{Tot}(E_\bC)$ might be viewed as a geometric model for change-of-rings as in (\ref{eq: chan ring}).
        \end{rmk}

    \begin{example}[Universal specific case]\label{ex: univ spec SH}
        Fixing $X$, we may also set $\Phi=X$ and the map $\Phi \to BU$ the classifying map for $TX$. Then $X$ tautologically admits a $\Phi$-structure. In this case, $A = \operatorname{Thom(\Omega^2 X \to BU \times \bZ)}$. Conjecture \ref{conj: SH} says that symplectic cohomology should be defined over $A$. This $A$ is the universal coefficient ring that symplectic cohomology of $X$ should be defined over.
    \end{example}

    \begin{rmk}
        Throughout the rest of the paper, we work with Lagrangian rather than Hamiltonian Floer theory. However, much of the abstract set-up can be adapted to the closed-string setting, by replacing discs with boundary punctures considered in Section \ref{sec: abstr dis} by punctured rational curves, cf. Section \ref{sec: int pun}.
    \end{rmk}

\subsection{Lagrangian Floer theory}\label{sec: Lag Floer conj}
    In the previous section we discussed when symplectic cohomology should be defined over an $\bE_2$-ring spectrum $A$. In this section, we discuss when Lagrangian Floer homology can be defined over an $A$-algebra $R$.
    \begin{defn}
        A \emph{tangential pair} consists of a pair of based spaces $\Theta, \Phi$, along with maps that fit into a commutative diagram:
        \begin{equation}\label{eq: strong}
            \xymatrix{
                \Theta 
                \ar[r]
                \ar[d]
                &
                \Phi
                \ar[d]
                \\
                BO 
                \ar[r]
                &
                BU
            }
        \end{equation}
    \end{defn}

    Fix a tangential pair $(\Theta, \Phi)$. Associated to this tangential pair are two ring spectra. Firstly, we let $A = \operatorname{Thom}(\Omega^2 \Phi \to BU \times \bZ)$ be the $\bE_2$ ring spectrum considered in Section \ref{sec: conj SH}.

    Let $F$ be the homotopy fibre of $\Theta \to \Phi$. From (\ref{eq: strong}), $F$ maps naturally to the homotopy fibre of $BO \to BU$, which is $U/O$. Extending the two fibration sequences further to the left, we obtain a commutative diagram:
    \begin{equation}\label{eq: strong thom}
        \xymatrix{
            \Omega^2 \Phi
            \ar[r]
            \ar[d]
            &
            \Omega F
            \ar[d]
            \\
            BU \times \bZ \ar[r]
            &
            BO \times \bZ
        }
    \end{equation}
    where the bottom horizontal arrow is the map classifying realification of complex vector bundles. 

    Let $R$ be the Thom spectrum of the stable vector bundle over $\Omega F$ classified by the map (\ref{eq: strong thom}): $\Omega F \to BO \times \bZ$.

    Then $R$ is an associative ring spectrum, and by functoriality of the Thom spectrum construction is an algebra over the $\mathbb{E}_2$ ring spectrum $A$.

    Similar Bott periodicity considerations to those in Section \ref{sec: conj SH} motivate the following:

    \begin{conj}
        Assume $X$ admits a $\Phi$-structure. Then there exists a Fukaya category $\scrF(X; A,R)$, whose objects are exact Lagrangians $L$ equipped with a \emph{compatible $\Theta$-structure}, meaning the following diagram commutes:
        \begin{equation}
            \xymatrix{
                L 
                \ar[r]
                \ar[d]
                &
                \Theta
                \ar[r]
                \ar[d]
                &
                BO
                \ar[d]
                \\
                X
                \ar[r]
                &
                \Phi
                \ar[r]
                &
                BU
            }
        \end{equation}
        where the composition along the rows are classifying maps for $TL$ and $TX$ respectively. 
    \end{conj}
    As usual in the Fukaya category, we must either assume $L$ is compact, or allow it to be cylindrical at infinity and make a choice of ``wrapping scheme'' at infinity; we choose to ignore this in this section since it is orthogonal to the study of tangential structures.

    \begin{rmk}\label{rmk:reverse bundle}
        The Pontrjagin-Thom construction identifies the homotopy groups of the Thom ring spectrum $R$ with the bordism groups of manifolds with stable \emph{normal} bundle classified by the pullback of the universal bundle $\xi_{\mathrm{taut}}$ under $\iota: \Omega F \to BO\times \bZ$. By precomposing with the map taking a loop to its inverse, this is isomorphic to the opposite of the Thom ring spectrum of the pullback of $-\xi_{\mathrm{taut}}$ under the same map $\iota$. Since a ring and its opposite have the same homotopy groups, this means that the homotopy groups of $R$ also describe bordism of manifolds with stable \emph{tangent} bundle classified by the pullback of the universal bundle. (See \cite[Section 3]{Galatius_etal}  and \cite{Nguyen} for a detailed  discussion of the corresponding relation between stable tangential and stable normal structures for the spectra $MO$ and $MU$, and in particular the indexing of the unstable versions of these bordism theories.) Passing between stable tangential and stable normal data accounts for the appearance of the minus sign in Corollary \ref{cor:1}.

        Put differently, if $\Ind_{\Omega}$ is the index bundle over $\Omega F$, and $\iota: \Omega F \to \Omega F$ reverses loops, then $\Ind_{\Omega} + \iota^*\Ind_{\Omega}$ is stably trivial, so $\iota^*\Ind_{\Omega}$ is stably equivalent to $-\Ind_{\Omega}$. But by reversing loops, the Thom spectra of $\iota^*\Ind_{\Omega}$ and $\Ind_{\Omega}$ also agree.
    \end{rmk}

    \begin{example}\label{ex: thet cont}
        If $\Theta$ is contractible, $\Omega \Phi \simeq F$ and so $R \simeq A$. Therefore the objects in the $A$-linear Fukaya category should be exact Lagrangians equipped with a framing, which is compatible with a trivialisation of the restriction of the $\Phi$-structure on $X$ to $L$.
    \end{example}
    There is a spectrum $MOP \simeq \vee_{k \in \bZ} \Sigma^k MO$, whose associated homology theory $MOP_i(X)$ consists of bordism classes of closed manifolds over $X$ of any dimension. This has a similarly-defined complex-oriented cousing $MUP \simeq \vee_{l \in \bZ} \Sigma^{2l} MU$. These can be constructed as the Thom spectra of the universal virtual vector bundles over $BO \times \bZ$ and $BU \times \bZ$ respectively. These can be constructed as $\bE_\infty$-ring spectra \cite{Ando-Blumberg-Gepner} (cf. Convention \ref{convention}).
    
    \begin{example}[Universal general case]
        Let $\Theta = BO$ and $\Phi = BU$. We have that $A = MUP$ (cf. Example \ref{ex: univ sh}) and $R$ is the Thom spectrum of a bundle over $\Omega U/O \simeq BO \times \bZ$; this is $R = MOP$, the periodic unoriented bordism ring spectrum. Any Lagrangian admits a canonical such $\Theta$-structure.
    \end{example}
    \begin{example}[Universal specific case] \label{ex: univ lag}
        Let $L \subseteq X$ be a closed exact Lagrangian. Leting $\Theta = L$ and $\Phi = X$, $L$ and $X$ then tautologically admit compatible $\Theta$ and $\Phi$ structures respectively.

        Conjecture \ref{sec: Lag Floer conj} says that $L$ is therefore an object over a Fukaya category defined over $R$, where $R$ is the Thom spectrum of a virtual vector bundle over $\Omega \operatorname{hofib}(L \to X)$.
    \end{example}
    \begin{example}[Contractible Lagrangian case]
        Let $\Theta, \Phi$ be a tangential pair, and assume $X$ admits a $\Phi$-structure. Then any contractible Lagrangian $L$ admits a $\Theta$-structure.
    \end{example}
    One can compute the set of $\Theta$-structures on a given Lagrangian:
    \begin{prop}\label{prop: class brane str}
        Let $(\Theta, \Phi)$ be a tangential pair and assume $X$ is equipped with a $\Phi$-structure. Let $L \subseteq X$ be a closed exact Lagrangian admitting a $\Theta$-structure.

        Let $P$ be the homotopy pullback of the following diagram:
        \begin{equation}\label{eq:bran}
            \xymatrix{
                &
                \Phi \ar[d]
                \\
                BO 
                \ar[r]
                &
                BU
            }
        \end{equation}
        $L$ and $\Theta$ both admit natural maps to $P$; let $Q$ be the homotopy fibre of $\Theta \to P$. 

        Then the set of homotopy classes of $\Theta$-structure bijects with the set of lifts of the map $L \to P$ to $\Theta$. If (\ref{eq: strong}) is a diagram of loop spaces, this is either empty or a torsor over the set of homotopy classes of map $[L,Q]$.
    \end{prop}
    
    In the next section we lay out our expectations for what the structure of this category is.
    
\subsection{Structure of spectral Fukaya category}\label{sec: conj struc}

    Let $\Theta, \Phi$ be a tangential pair. Recall our standing notation:

\begin{align}\label{eq: FAR}
\nonumber    F &= \operatorname{hofib}(\Theta \to \Phi) \\
A &= \mathrm{Thom}(\Omega^2\Phi \to BU \times \bZ) \\ 
\nonumber R &= \mathrm{Thom}(\Omega F \to BO \times \bZ)
\end{align}

    Assume $X$ has a $\Phi$-structure, and $L, K$ and $P$ are exact Lagrangians in $X$ with compatible $\Theta$-structures.

    In Section \ref{sec: Floer tang}, we explain how to construct a $(\Theta,\Phi)$-oriented Floer flow category $\cM^{LK}$ (see Section \ref{sec: Thet flow} for a definition) to each such pair of Lagrangians. We work under additional technical assumptions (see Section \ref{sec: tan}) that amount to saying $A$ and $R$ are connective and satisfy $\pi_0 \cong \bZ$.
    
    These flow categories are the objects of a category $\operatorname{Flow}^{\Theta,\Phi}$ which we conjecture is equivalent to the category of $A$-linear $R-R$ bimodules (i.e. $R \otimes_A R^{op}$-modules) (see Conjecture \ref{conj: PT/CJS bimod}, motivated by Proposition \ref{prop: flow pt pair case} which essentially verifies this for flow categories with one object (under Assumption \ref{asmp})). We write $M^{LK}$ for the $R-R$ bimodule corresponding to the flow category $\cM^{LK}$.

    Composition in this Fukaya category is a bilinear map of flow categories $\cM^{LKP}: \cM^{LK} \times \cM^{KP} \to \cM^{LP}$ built from moduli spaces of holomorphic triangles; under Conjecture \ref{conj: PT/CJS bimod} this corresponds to a map of $R-R$ bimodules $M^{LK} \otimes_R M^{KP} \to M^{LP}$.

    The set of morphisms from $L$ to $K$ in the Donaldson-Fukaya category $\scrF_*(L,K;A,R)$ is then obtained by taking the set of bordism classes of right flow modules over $\cM^{LK}$. Conjecture \ref{conj: flow mod PT} (motivated by Proposition \ref{prop: bord pt bimod}) says that this set of morphisms is obtained from $M$ and the diagonal bimodule $R_\Delta$ via
    \begin{equation}\label{eq: fuk flow}
        \scrF_*(L,K;A,R) = \pi_*\left(M^{LK} \otimes_{R-R} R_{\Delta}\right)
    \end{equation}
    $M^{LK} \otimes_{R-R} R_{\Delta}$ is an $A$-module, so we find that the Fukaya category is enriched over the monoidal category of $A$-modules (cf. \cite{Mandell:Smash} for the fact that the (homotopy) category of $A$-modules is monoidal). However, if $L \cap K$ is a single point, the corresponding module is a copy of $R$, viewed as an $A$-module, instead of a copy of $A$.
    \begin{rmk}\label{rmk:module action}
        (\ref{eq: fuk flow}) implies that the Floer groups $\scrF_*(L,K;A,R)$ are modules over $\pi_*(A)$; though we do not carry it out, a flow-categorical model for this can be constructed using methods in Section \ref{sec: int pun}.
    \end{rmk}
    \begin{rmk}
        Extending Remark \ref{rmk:module action}, (\ref{eq: fuk flow}) further implies that the Floer groups $\scrF_*(L,K;A,R)$ are modules over the ring of $R-R$ bilinear endomorphisms of $R_\Delta$: this is $\pi_*THC_A(R)$, where $THC_A$ denotes topological hochschild \emph{co}homology (relative to $A$).
    \end{rmk}
    
    A reason that we take morphisms to be bordism classes of right flow modules over $\cM^{LK}$, rather than bordism classes of maps from a fixed flow category with one object to $\cM^{LK}$ (which would correspond to $\pi_*M^{LK}$), is that the unit morphism in the Fukaya category, constructed using holomorphic curves in Section \ref{sec: unit OC}, lies in the former rather than the latter.

\subsection{Open-closed map}

     Assume for now we restrict to working with closed Lagrangian submanifolds. In this case, when working in the graded case over $\bZ$, the open-closed map is a map $\cO\cC: HH_*(\scrF(X;\bZ)) \to H_{*+n}(X)$. Composing with the Dennis trace, this gives a map $K_0(\scrF(X; \bZ)) \to H_n(X)$, sending the class of a Lagrangian brane to its homology class \cite[Section 3]{PS}.

     \begin{rmk}
         Recall that the \emph{topological Hochschild homology} of an associative ring spectrum $R$ is the spectrum defined by $THH(R) := R_{\Delta} \otimes_{R \otimes R^{op}} R_{\Delta}$ where $R_\Delta$ is the diagonal bimodule, and $\otimes$ denotes (derived) tensor product in the category of spectra.

         If $R$ is instead an $A$-algebra for some $\bE_2$ ring spectrum $A$, \emph{relative} topological Hochschild homology is defined by taking derived tensor product over $A$: $THH_A(R) := R_{\Delta} \otimes_{R \otimes_A R^{op}} R_{\Delta}$.
     \end{rmk}

     In the case where $\Phi$ is contractible, so $A \simeq \bS$, we expect to find a version of the open-closed map landing in homology of $X$ with topological hochschild homology coefficients:
     \begin{equation*}
         THH(\scrF(X; A,R)) \to \Sigma^{-n}X \wedge THH(R)
     \end{equation*}
     In Section \ref{sec: unit OC} we construct a map from the set of objects of $\scrF(X; A,R)$ to $\pi_0$ of the right hand side; see Section \ref{sec: rel flow mod}.

     In the general case, the target of the open-closed map is more complicated; motivated by the computations of Section \ref{sec: rel flow mod}, we expect it to be a map 
     \begin{equation*}
         THH_A(\scrF(X;A,R)) \to \Sigma^{-n} THH_B(R)
     \end{equation*}
    where $B$ is the Thom spectrum of the virtual vector bundle over $\Omega^2 X$ classified by the composition
    \begin{equation*}
        \Omega^2 X \xrightarrow{\Omega^2 TX} \Omega^2 BU \simeq BU \times \bZ
    \end{equation*}

    We discuss the target of the open-closed map in more detail in Section \ref{sec: rel flow mod}, in particular Corollary \ref{cor: aaa}.

\subsection{Relation to sheaf quantisation}
    Let $C$ be an $\bE_\infty$ ring spectrum. Nadler-Shende \cite{Nadler-Shende} construct a $C$-linear category $\mu sh(X;C)$ associated to a Liouville domain $X$, when the \emph{$C$-valued Maslov obstruction} of $X$, is nullhomotopic. This obstruction is defined to be the composition:
    \begin{equation}\label{eq: maslov comp}
        X \xrightarrow{TX} BU \to B(U/O) \xrightarrow{B^2 J_C} B^2 \operatorname{Pic}(C)
    \end{equation}
    where $\operatorname{Pic}(C)$ is the Picard space of $C$, a nonconnected delooping of the group of units $GL_1(C)$, and $B^2 J_C$ is a two-fold delooping of the $C$-valued $J$-homomorphism, $J_C: BO \to Pic(C)$.

    A nullhomotopy of the composition (\ref{eq: maslov comp}) induces a lift of $TX$ through $G \to BU$, where $G$ is the homotopy fibre of the composition $BU \to B^2 \operatorname{Pic}(C)$.
    
    Applying Conjecture \ref{sec: Lag Floer conj} with $\Theta$ contractible predicts that if the $C$-valued Maslov obstruction is nullhomotopic, there is a spectral Fukaya category defined over $A := \operatorname{Thom}(\Omega^2 G \to BU \times \bZ)$ (note if $\Theta$ is contractible then $A \simeq R$). The objects in this category are described in Example \ref{ex: thet cont}.

    There is a similar story for more general tangential pairs, with the sheaf quantisation side explored in \cite{Nadler-Shende} and \cite{Jin-Treumann}. 
    \begin{rmk}
        Recall a \emph{$C$-orientation} on an $n$-manifold $M$ consists of a $C$-valued Thom class, $\tau_M \in H^n(M; C)$. Note that if $C=\bS$ is the sphere spectrum, the Pontrjagin-Thom theorem implies that $\pi_* C$ bijects with the framed bordism groups, but these are \emph{not} isomorphic to the $C$-oriented bordism groups.

        However, any $C$-oriented closed $n$-manifold determines a class in $\pi_nC$, called its \emph{$C$-oriented fundamental class}, cf. \cite[Chapter V]{Rud98}.
    \end{rmk}

    \begin{rmk}
        The ring spectrum $A=\operatorname{Thom}(\Omega^2 G \to BU \times \bZ)$ is not in general equivalent to the original ring $C$, but by results of \cite{Antolin-Camarena-Barthel} there is a map of rings $f:A \to C$. The homotopy groups of $A$ are the bordism groups of manifolds equipped with both a stable complex structure and a $C$-orientation, and $f_*$ sends such a manifold to its $C$-fundamental class.
    \end{rmk}
    
    We expect that in this case the $A$-linear Fukaya category $\scrF(X; A)$ recovers $\mu sh(X; C)$ via change-of-rings along $f$:
    \begin{equation}\label{eq: base-change}
        \mu sh(X; C) \simeq \scrF(X; A) \otimes_A C
    \end{equation}
    but not the converse: in general, there is no splitting $C \to A$, and even in the case where there is a splitting the converse still may not hold. From this point of view, Fukaya categories seem to pick up more information. 

    We sketch an example where this phenomenon occurs. Let $C = \bS$. In this case there is a map of rings $\iota: C \to A$ splitting $f$, but we argue that the Fukaya category over $A$ cannot be obtained from the sheaf quantisation category over $C$ by base-change along $\iota$. 

    In this example, we assume both Assumption \ref{asmp} and that the $A$-linear spectral Fukaya category $\scrF(X; A)$ has morphisms as described in Section \ref{sec: conj struc}.

    \begin{ex}
        
        Let $C=\bS$, and $\Phi$ be the homotopy fibre of the map $BU \to B^2 \operatorname{Pic}(\bS)$, as above. Let $\Theta$ be contractible; this implies that $A \simeq R$ (where $A,R$ are as defined in (\ref{eq: FAR})).

        We can compute the homotopy pullback $P$ of (\ref{eq:bran}) appearing in Proposition \ref{prop: class brane str} explicitly. Since (\ref{eq:bran}) is a diagram of infinite loop spaces, $P$ is the homotopy fibre of the composite $BO \to BU \to B^2\operatorname{Pic}(\bS)$. But this factors through $B(U/O)$ and the composition $BO \to BU \to B(U/O)$ is nullhomotopic, so $BO \to B^2\operatorname{Pic}(\bS)$ also is. It follows that the fibration sequence for $P$ splits, so $P \simeq BO \times B\operatorname{Pic}(\bS)$.

        Now $A$ is the Thom spectrum whose homotopy groups compute the bordism groups of manifolds equipped with an $\bS$-orientation as well as a stable complex structure, and the map $f_*: \pi_* A \to \pi_* \bS$ sends such a manifold to its $\bS$-fundamental class. 

        Let $X= T^*S^d$; this admits a $\Phi$-structure coming from the standard stable framing on $S^d$. Let $L$ be the zero-section equipped with the $\Theta$-structure coming from the same stable framing.

        Then using Proposition \ref{prop: class brane str} and the isomorphism $\pi_j \operatorname{Pic}(\bS) \cong \pi_{j-1} \bS$ (for $j \geq 2$), we find that the set of $\Theta$-structures on the zero-section is a torsor over $\pi_d O \oplus \pi_{d-1} \bS$. Let $(\alpha, \beta) \in \pi_d O \oplus \pi_{d-1}\bS$, and let $K$ be the zero-section, with the $\Theta$-structure obtained from the one on $L$ by applying $(\alpha, \beta)$. 

        By picking an appropriate Morse function on $S^d$ with exactly 2 critical points $x$ and $y$, we may choose perturbation data such that the flow category $\cM := \cM^{LK}$ associated to the Lagrangians $L$ and $K$ has exactly 2 objects $x$ and $y$ in degrees $d$ and $0$ respectively, with the only nonempty moduli space $\cM_{xy}$ being diffeomorphic to $S^{d-1}$.

        This moduli space $\cM_{xy}$ admits a stable complex structure and $\bS$-orientation. Let $[\cM_{xy}] \in \pi_{d-1}A$ be its bordism class. Its image under the map $f_*:\pi_*A \to \pi_*\bS$ is exactly $\beta$.

        Now assume $d$ is 3 mod 4. Then $\pi_d O \cong \bZ$; we may choose $\alpha$ to be nonzero. Appealing to Assumption \ref{asmp}, the stable tangent bundle of $\cM_{xy} \simeq S^{d-1}$, with its complex structure, is classified by the image of $\alpha$ under the composition
        \begin{equation*}
            \pi_d O \cong \pi_{d+1} BO \to \pi_{d+1} BU \cong \pi_{d-1} BU \cong [\cM_{xy}, BU]
        \end{equation*}
        This map is injective, so for nonzero $\alpha$, we find that $\cM_{xy}$ has nontrivial stable tangent bundle (as a stably complex bundle). In particular, $\cM_{xy}$ has nonzero Chern character and so does not admit a complex nullbordism. This implies that its class in $\pi_{d-1}A$ is nonzero, and similarly is not in the image of the map $\iota_*: \pi_{d-1}\bS \to \pi_{d-1}A$.

        Write $CF(L,K;A)$ for the $A$-module corresponding to the $(\Theta,\Phi)$-oriented flow category $\cM$ under Conjecture \ref{conj: PT/CJS bimod}. Assuming Conjecture \ref{conj: PT/CJS bimod} behaves as expected with respect to taking cones, we find that $HF(L, K; A) \simeq \operatorname{Cone}([\cM_{xy}]: \Sigma^{d-1} A \to A)$ as $A$-modules. The equivalence class of this $A$-module determines $[\cM_{xy}]$ up to sign. 
        
        If $HF(L,K;A) \simeq E \otimes_\bS A$ for some (bounded below) spectrum $E$, $E$ would also have to be of the form $E \simeq \operatorname{Cone}(\gamma: \Sigma^{d-1}\bS \to \bS)$ for some $\gamma \in \pi_{d-1} \bS$, so we would have that $[\cM_{xy}] = \pm \iota_* \gamma$, which we have argued cannot hold.

        In particular, the Fukaya category $\scrF(X; A)$ is not equivalent to $\mu sh(X; \bS) \otimes_\bS A$, since $HF(L,K;A)$ is not obtained from any spectrum by base-change along $\bS \to A$.
    \end{ex}

\section{Spaces of abstract discs} \label{sec: abstr dis}
    In this section we construct \emph{spaces of abstract discs with tangential structure}: this provides an ad-hoc model for Bott periodicity that is close enough to almost-complex geometry that we can use them in Floer theory, whilst also being flexible enough to incorporate into the abstract set-up of flow categories. 

    Roughly, we construct virtual vector bundles over (the universal cover of) $U/O$ or $U/O(N)$, along with a suitable multiplicative structure: this allows us to avoid fixing any specific model for $BO$. 
\subsection{Tangential structures}\label{sec: tan}

    For any $1 \leq N < \infty$, the rank $N$ Lagrangian Grassmannian $U/O(N)$ is path-connected, and has fundamental group $\bZ$. We consider its universal cover $\widetilde{U/O}(N) \to U/O(N)$, which we equip with some basepoint living over the natural basepoint $\bR^N \subseteq \bC^N$ in $U/O(N)$. Maps into, and points in, $\widetilde{U/O}(N)$ will always be denoted $\tilde \cdot$, and we will drop the $\tilde{\,}$ to denote its projection down to $U/O(N)$.

    There are natural \emph{stabilisation} maps $U/O(N) \to U/O(N+1)$, which send a totally real subspace $V \subseteq \bC^N$ to $\bR \oplus V \subseteq \bC \oplus \bC^{N}$. We define $U/O$ to be the colimit as $N \to \infty$; since the stabilisation maps are cofibrant this is also the homotopy colimit. For convenience, we will often consider this to be the rank $N$ case, for $N=\infty$: $U/O = U/O(\infty)$.
    \begin{defn}\label{def: tan}
        For $1 \leq N \leq \infty$, a \emph{tangential structure of rank $N$} is a based space over $\widetilde{U/O}(N)$. 
        An \emph{oriented tangential structure of rank $N$} is a 1-connected based space over $\widetilde{U/O}^{or}(N)$, where $\widetilde{U/O}^{or}(N)$ is defined to be the homotopy fibre of the map $\widetilde{U/O}(N) \to K(\bZ/2, 2)$ classifying the second Stiefel-Whitney class.
    \end{defn}
    Note that all our tangential structures are \emph{graded}, meaning they live over $\widetilde{U/O}(N)$ rather than over $U/O(N)$.
    \begin{ex}
        We write $fr$ for the tangential structure given by taking the inclusion of the basepoint $\mathrm{pt} \hookrightarrow \widetilde{U/O}(N)$. 
    \end{ex}

\subsection{Abstract discs and boundary data}\label{subsec:abstract discs}
    Let $\Theta$ be a tangential structure of some finite rank $N$.
    \begin{defn}
        Let $i \geq 0$ be a non-negative integer and $j\in \{0,1\}$. A \emph{domain with $i$ inputs and $j$ outputs}, $D_{i,j}$, is a compact Riemann surface $\overline D_{i,j}$ biholomorphic to a disc, with $(i+j)$ ordered boundary marked points and, if $j=0$, a marked point $p$ in the interior along with a nonzero tangent vector $v$ in $T_pD_{i,j}$. If $i>0$, we require $v$ to point towards the boundary component between the first and final boundary punctures. $i$ of the boundary marked points are labelled \emph{incoming} and the others are labelled \emph{outgoing}. 
        We call the incoming marked points and, if $j=1$, the outgoing one, \emph{punctures}. If $j=0$, we call the vector $v$ the \emph{asymptotic marker}.
        
        We order the incoming marked points (non-cyclically) in an anticlockwise manner, such that the output (if present) is anticlockwise of the final input and clockwise of the initial input (if there are any inputs). 
          \end{defn}
          
          We write $D_{i,j}$ for the complement of the punctures; this is an open Riemann surface with boundary. We may sometimes write $D_{i,0/1}$ when  simultaneously discussing the cases in which there is or is not an output puncture.

  \begin{rmk}
      Discs with one output arise when considering Floer differentials, holomorphic triangles, etc. Discs with no output arise when considering the open-closed map.
  \end{rmk}

\begin{figure}[ht]\label{Domain with 3 inputs}
\begin{center}
\begin{tikzpicture}[scale=0.5] 
\draw (0,0) circle (3);
\draw[gray] (3,0) circle (0.1);
\draw[fill,gray] (-3,0) circle (0.1);
\draw[fill,gray] (-3/2, 2.6) circle (0.1);
\draw[fill,gray] (-3/2, -2.6) circle (0.1);

\draw[->] (3.2,0) -- (3.7,0);
\draw[->] (-3.7,0) -- (-3.2,0);
\draw[->] (-1.95, 3.38) -- (-1.65, 2.86);
\draw[->] (-1.95, -3.38) -- (-1.65, -2.86);
\end{tikzpicture}
\caption{Domain with 3 inputs and 1 output}

\end{center}
\end{figure}
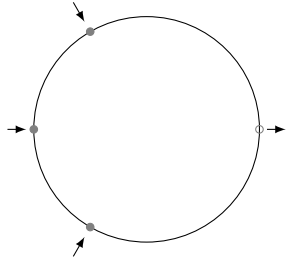

    For fixed $i,j$, the space of domains with the data of these marked points is contractible; when $j=0$ this requires the condition on the direction of the asymptotic marker.

    \begin{defn}\label{def: boun dat}
        We define \emph{boundary data of rank $N$} for a domain $D_{i,0/1}$ to be a pair $(\eps, \tilde\Lambda)$ such that
        \begin{itemize}
            \item $\eps$ is a collection of disjoint \emph{strip-like ends} for each boundary puncture: this consists of holomorphic boundary-respecting embeddings $\eps_x: (-\infty, 0] \times [0,1] \hookrightarrow D_{i,0/1}$ onto a neighbourhood of each incoming puncture $x$ and a similar embedding $\eps_y: [0, +\infty) \times [0,1] \hookrightarrow D_{i,0/1}$ onto a neighbourhood of the outgoing puncture $y$, if there is one.\par 
           
            \item $\tilde \Lambda$ is a set of  \emph{boundary conditions} $\tilde \Lambda: \partial D_{i,0/1} \to \widetilde{U/O}(N)$ which are constant near the punctures (e.g. over each strip-like end). 
                    \end{itemize}

            We call the image of all the $\eps$ the \emph{strip-like ends}. We will sometimes use the phrase `shrunken strip-like ends' to mean a neighbourhood of infinity in the strip-like ends (precise such neighbourhoods are prescribed in Definition \ref{defn:CR data}). 
            
            We write the projection of $\tilde \Lambda$  to $U/O(N)$ as $\Lambda$, which we view as a totally real subbundle of the trivial complex vector bundle $\bC^N$ over $\partial D_{i,0/1}$. 
            We assume $\tilde \Lambda$ is constant over each strip-like end. We may sometimes write $\Lambda_b$ or $\tilde \Lambda_b$ for the restrictions of these maps to a component $b$ of the boundary $\partial D_{i,1}$.
            
\begin{figure}[ht]\label{Strip-like ends}
\begin{center}
\begin{tikzpicture}[scale=0.5] 
\draw (0,0) circle (3);
\draw[gray] (3,0) circle (0.1);
\draw[fill,gray] (-3,0) circle (0.1);
\draw[fill,gray] (-3/2, 2.6) circle (0.1);
\draw[fill,gray] (-3/2, -2.6) circle (0.1);

\draw[semithick, dotted] (-0.5,2.8) arc (25:-145:1);
\draw[semithick, red] (-1.02,2.72) arc (30:-145:0.5);

\draw[semithick, dotted] (-0.5,-2.8) arc (-25:145:1);
\draw[semithick, red] (-1.02,-2.72) arc (-30:145:0.5);

\draw[semithick, red] (-2.9,0.5) arc (85:-85:0.5);
\draw[semithick,dotted] (-2.7,0.95) arc (80:-80:1);

\draw[semithick,red] (2.9,0.5) arc (95:265:0.5);
\draw[semithick,dotted] (2.7,0.95) arc (100:260:1);
\draw[->] (3.2,0) -- (3.7,0);
\draw[->] (-3.7,0) -- (-3.2,0);
\draw[->] (-1.95, 3.38) -- (-1.65, 2.86);
\draw[->] (-1.95, -3.38) -- (-1.65, -2.86);

\end{tikzpicture}
\caption{Strip-like ends (dotted boundary) and shrunken strip-like ends (red boundary)}

\end{center}
\end{figure}
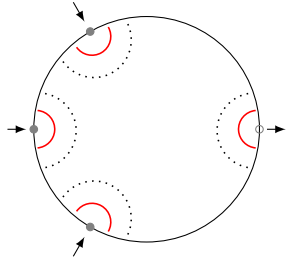

    For each puncture $x$, we write $\tilde u_x \in \widetilde{U/O}(N)$ for the restriction of $\tilde \Lambda$ to the $\{+1\}$-boundary of the strip-like end in a neighbourhood of the puncture where it is constant; similarly we write $\tilde v_x \in \widetilde{U/O}(N)$ for the restriction of $\tilde \Lambda$ to the corresponding $\{0\}$-boundary. $u_x, v_x$ refer to the corresponding elements of $U/O(N)$. 
    
    We assume that for each puncture $x$, $u_x$ and $v_x$ are transverse. We call the $\tilde u$ and $\tilde v$ the \emph{boundary conditions at the punctures}.

    \end{defn}
  
    \begin{defn}\label{def: hom lift}
        Let $\Theta$ be a tangential structure of rank $N$. A \emph{$\Theta$-orientation} on a domain with boundary data of rank $N$, $(D_{i,j}, \eps, \tilde\Lambda)$, consists of a pair $(H, \Gamma)$, where:
        \begin{itemize}
            \item For each boundary component $b$ of $\partial D_{i,j}$, $\Gamma_b$ is a totally real subbundle of the trivial complex vector bundle $\bC^N$ over $b \times [0,1]$ (equivalently, a map $b \times [0,1] \to \widetilde{U/O}(N)$), which agrees with $\Lambda$ over $b \times \{0\}$, and is constant in the $b$-direction (i.e. only depends on the $[0,1]$-direction) over each intersection of $b$ with the shrunken strip-like ends. 
            \item For each boundary component of $\partial D_{i,j}$, $H_b: b \times \{1\} \to \Theta$ is a lift of $\Lambda|_{b \times \{1\}}$, meaning the following diagram commutes:
            \begin{equation}
                \xymatrix{
                    b \times \{1\} \ar[r]_{H_b} \ar[d] &
                    \Theta \ar[d] \\
                    b \times [0,1] \ar[r]_{\Lambda_b} &
                    \widetilde{U/O}(N)
                }
            \end{equation}
            again such that $H$ is constant over each intersection of $b$ with the shrunken strip-like ends. 
        \end{itemize}
        We call $(H, \Gamma)$ a \emph{homotopy lift} of $\tilde \Lambda_b: b \to \widetilde{U/O}(N)$ to $\Theta$ for each $b$.\par 
        \end{defn}

\begin{figure}[ht]\label{Theta lift}
\begin{center}
\begin{tikzpicture}[scale=0.8] 
\draw[thick] (0,0) circle (3);
\draw[gray] (3,0) circle (0.1);
\draw[fill,gray] (-3/2, 2.6) circle (0.1);
\draw[fill,gray] (-3/2, -2.6) circle (0.1);

\draw[semithick,red] (3,0) -- (3.85,0.65);
\draw[semithick,red] (3,0) -- (3.85,-0.65);
\draw[semithick,red] (-1.5,2.6) -- ({3.9*cos(115)}, {3.9*sin(115)});
\draw[semithick,red] (-1.5,2.6) -- ({3.9*cos(125)}, {3.9*sin(125)});
\draw[semithick,red] (-1.5,-2.6) -- ({3.9*cos(-115)}, {3.9*sin(-115)});
\draw[semithick,red] (-1.5,-2.6) -- ({3.9*cos(-125)}, {3.9*sin(-125)});

\draw[semithick,gray, domain = 125:235] plot ({3.9*cos(\x)}, {3.9*sin(\x)});
\draw[semithick,gray, domain = 10:115] plot ({3.9*cos(\x)}, {3.9*sin(\x)});
\draw[semithick,gray, domain = -10:-115] plot ({3.9*cos(\x)}, {3.9*sin(\x)});

\draw (0,0) node {$D_{2,1}$};
\draw (0.25,2.5) node {$\Lambda$};
\draw (0.25,-2.5) node {$\Lambda$};
\draw (-2.5,0) node {$\Lambda$};

\node (a) at (-2.3,0) {};
\node (b) at (-2.6,1.5) {};
\draw[dashed,semithick,->] (a) to [out=60,in=-20, looseness=2] (b);

\draw  (1.5,3) node {$\Gamma$};
\draw  (1.5,-3) node {$\Gamma$};
\draw  (-3.5,0) node {$\Gamma$};

\draw[dashed,semithick,->] ({4.1*cos(45)}, {4.1*sin(45)}) -- ({5.1*cos(45)}, {5.1*sin(45)});
\draw ({4.3*cos(45)+0.5}, {4.3*sin(45)-0.25}) node {$H$};
\draw ({5.6*cos(45)}, {5.6*sin(45)}) node {$\Theta$};

\draw[dashed,semithick,->] ({4.1*cos(-45)}, {4.1*sin(-45)}) -- ({5.1*cos(-45)}, {5.1*sin(-45)});
\draw ({4.3*cos(-45)+0.45}, {4.3*sin(-45)+0.15}) node {$H$};
\draw ({5.6*cos(-45)}, {5.6*sin(-45)}) node {$\Theta$};

\draw[dashed,semithick,->] ({4.1*cos(180)}, {4.1*sin(180)}) -- ({5.1*cos(180)}, {5.1*sin(180)});
\draw ({4.3*cos(180)}, {4.3*sin(180)+0.45}) node {$H$};
\draw ({5.6*cos(180)}, {5.6*sin(180)}) node {$\Theta$};

\draw[thick,red, domain=110:130] plot ({3*cos(\x)},{3*sin(\x});
\draw[red] ({2.6*cos(105)},{2.6*sin(105)}) node {$u_{x_1}$};
\draw[red] ({2.6*cos(125)},{2.6*sin(125)}) node {$v_{x_1}$};

\draw[thick,red, domain=-110:-130] plot ({3*cos(\x)},{3*sin(\x});
\draw[red] ({2.6*cos(-110)},{2.6*sin(-110)}) node {$v_{x_2}$};
\draw[red] ({2.6*cos(-130)},{2.6*sin(-130)}) node {$u_{x_2}$};

\draw[thick,red, domain=-10:10] plot ({3*cos(\x)},{3*sin(\x});
\draw[red] ({2.6*cos(10)},{2.6*sin(10)}) node {$u_{y}$};
\draw[red] ({2.6*cos(-10)},{2.6*sin(-10)}) node {$v_{y}$};

\node (c) at (3.4,0.35) {};
\node (d) at (4.5,1) {};
\node (e) at (3.4,-0.35) {};
\node (f) at (4.5,-1) {};
\draw[dashed,semithick,<-] (c) to [out=-30,in=90, looseness=2] (d);
\draw[dashed,semithick,<-] (e) to [out=30,in=-90, looseness=2] (f);
\draw[red] (4.5,1.75) node {$\gamma_{u_y}$};
\draw[red] (4.5,-1.75) node {$\gamma_{v_y}$};

\node (g) at ({3.3*cos(115)-0.15},{3.3*sin(115)}) {};
\node (h) at ({4.5*cos(120)-0.25},{4.5*sin(120)}) {};
\draw[dashed,semithick,<-] (g) to [out=140,in=0, looseness=2] (h);
\draw[red] ({4.5*cos(120)-0.5},{4.5*sin(120)}) node {$\gamma_{u_{x_1}}$};

\end{tikzpicture}
\caption{Boundary conditions for the $\Theta$-lift}
\end{center}
\end{figure}
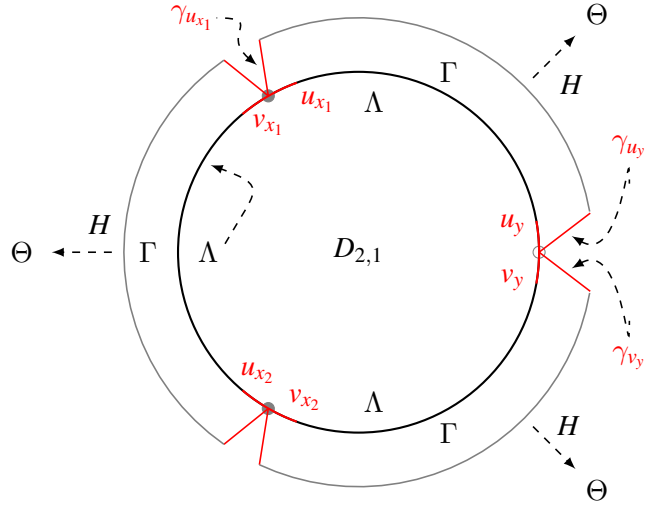

        \begin{rmk}\label{rmk: pnct}
            For each puncture $x$ and component $b$ touching $x$, restricting $\Gamma_b$ and $H_b$ to $\{z\} \times [0,1]$ and $\{z\} \times \{1\}$ for some $z$ in the intersection of $b$ with the shrunken strip-like end near $x$ gives us paths $\gamma: [0,1] \to \widetilde{U/O}(N)$ such that $\gamma(0) = u_x$ or $v_x$, and $h$ a lift of $\gamma(1)$ to $\Theta$. Together, these define such paths and lifts $\gamma_{u_x}, \gamma_{v_x}$ and $h_{u_x}, h_{v_x}$ for each $u_x$ and each $v_x$. Note that by assumption, these does not depend on the choice of $z$.
        \end{rmk}

    \begin{defn}
        A set of \emph{puncture data of rank $N$} (with $i$ inputs and $1$ respectively 0 outputs) is a tuple 
        \[
        (\tilde u, \tilde v) = ((\tilde u_{x_1}, \tilde v_{x_1}), \ldots, (\tilde u_{x_i}, \tilde v_{x_i}); (\tilde u_y, \tilde v_y))
        \]
        respectively
        \[
        (\tilde u, \tilde v) = ((\tilde u_{x_1}, \tilde v_{x_1}), \ldots, (\tilde u_{x_i}, \tilde v_{x_i}))
        \]
        of pairs in $\widetilde{U/O}(N)$ whose projections to $U/O(N)$ are transverse totally real subspaces of $\bC^N$ (note that despite the suggestive name and notation, there is no Riemann surface as part of this data). 
        
        A set of \emph{$\Theta$-oriented} puncture data $(\tilde u, \tilde v)$ also consists of paths $\gamma_{u_x}, \gamma_{v_x}: [0,1] \to \widetilde{U/O}(N)$, sending 0 to $u_x$ and $v_x$ respectively for each $x$, and lifts $h_{u_x}, h_{v_x}$ of each $\gamma_{u_x}(1), \gamma_{v_x}(1)$ to $\Theta$. 
    \end{defn}

        By Remark \ref{rmk: pnct}, any ($\Theta$-oriented) domain $\bD$ determines a set of ($\Theta$-oriented) puncture data $\bE(\bD)$.
        
        We think of puncture data as (locally constant) boundary data for some domain $D_{i,j}$, but only defined in some small neighbourhood of the punctures. We write $\bE = (\tilde u, \tilde v, \gamma, h)$ for this tuple, $\bE_x = (\tilde u_x,\tilde v_x, \gamma_{u_x}, \gamma_{v_x}, h_{u_x}, h_{v_x})$ its restriction to the data near a puncture $x$, and $\bE_{u_x}$ and $\bE_{v_x}$ for its restriction to each side of the puncture. If we do not want to label the inputs and outputs, we will sometimes use the chosen ordering on the inputs, and write this as $\bE = (\bE_1, \ldots, \bE_i; \bE_{out})$ (or omit the final term if there is no output).

    \subsection{Cauchy-Riemann data on abstract discs}\label{subsec: CR data}
    
    \begin{defn}\label{defn:CR data}
        \emph{Cauchy-Riemann data} on a domain with boundary data $(D_{i,j}, \eps, \tilde\Lambda)$ consists of a tuple $(L, J, Y, g)$, where
        \begin{itemize}
         \item The \emph{lengths} $L$ consist of a nonnegative number $L_x \geq 0$ for each puncture $x$.
            \item $J$ is a complex structure on the trivial bundle $\bC^N \to D_{i,j}$, compatible with the standard symplectic structure on $\bC^N$, and agreeing with the standard complex structure over $\partial D_{i,j}$ and over all the shrunken strip-like ends.
            \item $Y \in \Omega^{0,1}\left(D_{i,j}, \bC^N\right)$ is a $(0,1)$-form on the bundle $\bC^N \to D_{i,j}$ (with respect to the complex structure $J$), agreeing with $dt$ (where $t$ is the coordinate on $[0,1]$) on all the shrunken strip-like ends.
            \item $g$ is a Riemannian metric on $D_{i,j}$, agreeing with the standard metric on all the strip-like ends.
        \end{itemize}
    \end{defn}
    
     Here we use the term `shrunken strip-like ends' for the image of all the $\eps((\pm \infty, \pm L] \times [0,1])$.

     We now fix once and for all a Sobolev weight $\kappa > 2$. 
     
    \begin{defn}\label{def: CR data}
        Let $(D_{i,j}; L, \eps, \tilde\Lambda; J, Y, g)$ be a domain with boundary and Cauchy-Riemann data. We write $W^{2,\kappa}(D_{i,j},\bC,\Lambda)$ for the space of $W^{2,\kappa}$-sections of the trivial bundle $\bC^N \to D_{i,j}$, restricting to sections of $\Lambda$ over $\partial D_{i,j}$.

        Then the associated \emph{Cauchy-Riemann operator} is the linear map:

        $$D^{CR} := \bar \partial_J + Y: W^{2,\kappa}\left(D_{i,j}, \bC^N, \Lambda\right) \to \Omega^{0,1}_{W^{2,\kappa-1}}\left(D_{i,j}, \bC^N\right)$$
        \end{defn}
        
        We may sometimes add subscripts to the notation $D^{CR}$ to indicate the choice of data it depends upon.\par 
        It is standard (cf. \cite{McDuff-Salamon}) that $D^{CR}$ is always a Fredholm operator.
        
        \begin{defn} \label{defn:perturbation}
        A \emph{perturbation datum} for the above data consists of a pair $(V, f)$, where
        \begin{itemize}
            \item $V$ is a finite-dimensional real inner product space. We call $V$ the \emph{domain} of the perturbation.
            \item $f: V \to \Omega^{0,1}_{W^{2,\kappa-1}}(D_+, \bC^N, \Lambda)$ is a linear map, such that the \emph{perturbed Cauchy-Riemann operator} 
        \begin{equation} \label{eq: D plus f}D^{CR} + f: W^{2,\kappa}(D_{i,j}, \bC^N, \Lambda) \oplus V\to \Omega^{0,1}_{W^{2,\kappa-1}}(D_{i,j}, \bC^N)
        \end{equation}
        is surjective, and for all $v \in V$, $f(v)$ is supported away from the strip-like ends. 
    \end{itemize}
\end{defn}

    \begin{defn}
        For brevity, we say an \emph{abstract disc (with $i$ inputs and $j$ outputs, of rank $N$)} consists of a domain $D_{i,j}$, boundary data of rank $N$, as well as a choice of perturbation datum. We may write such an undecuple as $\bD_{i,j} = (D_{i,j}; L,\eps,\tilde\Lambda; H,\Gamma;J,Y,g; V, f)$.
        \end{defn}
       
    Given an abstract disc with both puncture data and Cauchy-Riemann data, there is a well-defined vector space given by 
    \begin{equation} \label{eq: ind}
        \ker(D^{CR}+f)
    \end{equation}
    
    Associated to a domain with boundary data $(D_{i,j}, L, \eps, \tilde \Lambda)$ is an integer $\mu \in \bZ$ called the \emph{Maslov index}, as in \cite[Section 11h]{Seidel:book} and \cite{Robbin-Salamon:Maslov}. The rank of \eqref{eq: ind} is $\mu(\bD_{i,j}) + \dim(V)$. 
    
    This depends only on the boundary data at the punctures, so we may sometimes write this as $\mu=\mu(\bE)$. Each puncture data $\bE_x$ has a Maslov index, and the Maslov index of the boundary data at the punctures is determined by these: 
   \begin{equation} 
       \mu(\bE) = \left(\sum_i \mu(\bE_i) \right) - \mu(\bE_{out})
   \end{equation}

\subsection{Gluing of abstract discs}\label{sec: glue}
Let $i,j \geq 0$. Assume we are given sets of $\Theta$-oriented puncture data $\bE$ with $i$ inputs and $1$ output, and $\bE'$ with $j$ inputs and either zero or one output, and both of rank $N$. We assume that $\bE_{out} = \bE'_k$ with $1\leq k \leq j$. We may define puncture data with $(i+j-1)$ inputs, and as many outputs as for $\bE'$:
\begin{equation}
    \bE \#_k \bE' := \left(\bE_1', \ldots, \bE'_{k-1}, \bE_1, \ldots, \bE_i, \bE'_{k+1}, \ldots, \bE'_j; [\bE'_{out}]\right)
\end{equation}
(where the $[\cdots]$ for $[\bE'_{out}]$ indicate that this may or may not be present). 

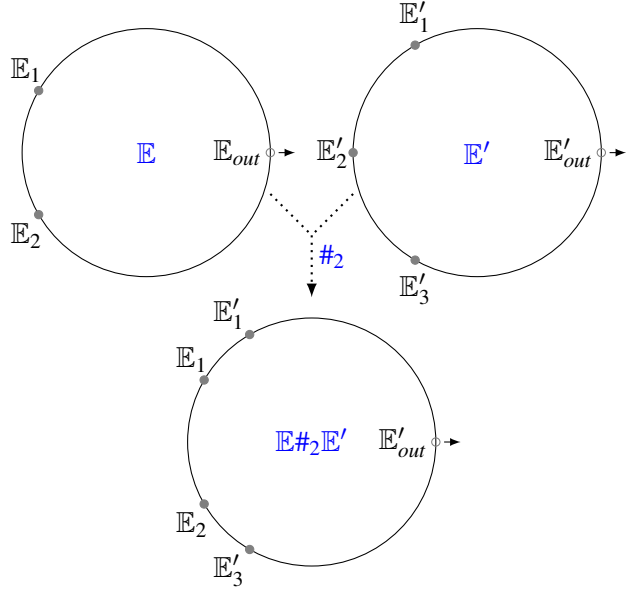
\begin{figure}[ht]\label{Gluing at second input}
\begin{center}
\begin{tikzpicture}[scale=0.55] 

\draw (4,0) circle (3);
\draw[gray] (7,0) circle (0.1);
\draw[->] (7.2,0) -- (7.6,0);
\draw[fill,gray] (1,0) circle (0.1);
\draw[fill,gray] (5/2, 2.6) circle (0.1);
\draw[fill,gray] (5/2, -2.6) circle (0.1);

\draw (6.2,0) node {$\bE'_{out}$};
\draw (0.5,0) node {$\bE'_2$};
\draw (5/2,3.3) node {$\bE'_1$};
\draw (5/2,-3.3) node {$\bE'_3$};

\draw[blue] (4,0) node {$\bE'$};
\draw[blue] (-4,0) node {$\bE$};
\draw[blue] (0,-7) node {$\bE \#_2 \bE'$};

\draw (-4,0) circle (3);
\draw[gray] (-1,0) circle (0.1);
\draw[->] (-0.8,0) -- (-0.4,0);

\draw (-1.8,0) node {$\bE_{out}$};
\draw (-2.9-4,2) node {$\bE_1$};
\draw (-2.9-4,-2) node {$\bE_2$};
\draw[fill,gray] (-2.6-4,3/2) circle (0.1);
\draw[fill,gray] (-2.6-4,-3/2) circle (0.1);

\draw (0,-7) circle (3);
\draw[gray] (3,-7) circle (0.1);
\draw (2.2,-7) node {$\bE'_{out}$};
\draw[->] (3.2,-7) -- (3.6,-7);
\draw[fill,gray] (-2.6,3/2-7) circle (0.1);
\draw (-2.9,-5) node {$\bE_1$};
\draw[fill,gray] (-2.6, -3/2-7) circle (0.1);
\draw (-2.9,-9) node {$\bE_2$};
\draw[fill,gray] (-3/2, -2.6-7) circle (0.1);
\draw (-2,-2.6-7.5) node {$\bE_3'$};
\draw[fill,gray] (-3/2, 2.6-7) circle (0.1);
\draw (-2,2.6-6.5) node {$\bE'_1$};

\draw[thick,dotted]  (1,-1) -- (0,-2);
\draw[thick,dotted]  (-1,-1) -- (0,-2);
\draw[thick,dotted, ->]  (0,-2) -- (0,-3.5);
\draw[blue] (0.5,-2.5) node {$\#_2$};

\end{tikzpicture}
\caption{Gluing at the $k$-th input, $k=2$}
\end{center}
\end{figure}

Let $\bD_{i,1} = (D_{i,1}; L,\eps,\tilde\Lambda; H,\Gamma; J,Y,g;V, f)$  and $\bD'_{j,0/1} = (D'_{j,0/1}; L',\eps',\tilde\Lambda'; H',\Gamma';J',Y',g'; V', f')$ be $\Theta$-oriented abstract discs (as usual, $\bE$ and $\bE'$ are implicit here). For $T, T' \geq 0$  consider the domains $T \cdot \bD_{i,1} := (\ldots, L+T, \ldots)$ and $T' \cdot \bD'_{j,0/1}$; later we will choose $T$ and $T'$ sufficiently large for the conclusion of Lemma \ref{lem: T big ok} to hold.\par 

Now we construct $\bD''_{i+j-1,0/1} = \rho_k(\bD_{i,1}, \bD'_{j,0/1})$ as follows. The domain $D''_{i+j-1,0/1}$ is constructed by removing (the interiors of) the shrunken strip-like ends at the output of $T\cdot D_{i,1}$ and the $k^{\textrm{th}}$ input of $T' \cdot D'_{j,0/1}$, and gluing them together over the new boundary face:
\begin{equation}
    D''_{i+j-1,0/1} := \left(D_{i,1}\setminus \eps_{out}((L+T, \infty) \times [0,1])\right) \bigcup_{[0,1]}
    \left( D'_{j,0/1} \setminus \eps'_k((-\infty, -L-T') \times [0,1])\right)
\end{equation}
We call the union $\eps_{out}([L,L+T] \times [0,1]) \cup \eps'_k([-L-T',-L] \times [0,1])$ the \emph{neck}; this is canonically diffeomorphic to $[-T, T'] \times [0,1]$.\par 
Since each puncture of $D''_{i+j-1,0/1}$ is a puncture of either $D_{i,1}$ or $D'_{j,0/1}$, we equip $D''_{i+j-1,0/1}$ with lengths $L''$ and strip-like ends $\eps''$ by taking them to be either $L$ or $L'$ and $\eps$ and $\eps'$, depending on which of the two original domains the puncture lies in.\par 
The boundary conditions $\tilde\Lambda''$ can be taken to be either $\tilde\Lambda$ or $\tilde\Lambda'$ on each boundary component, except for on the two boundary components of $D''_{i+j-1,0/1}$ touching the neck. On these, we may concatenate $\tilde\Lambda$ and $\tilde\Lambda'$; this is well-defined by the assumption that $\bE_{out} = \bE'_k$. We define the homotopy lifts to $\Theta$, $H''$ and $\Gamma''$, in the same way.\par 
The Cauchy-Riemann data $(J'', Y'', g'')$ we may take to be the Cauchy-Riemann data $(J,Y,g)$ over $D_{i,1}$ and $(J',Y',g')$ over $D'_{j,0/1}$; these glue by the assumption that the Cauchy-Riemann data must be standard over the shrunken strip-like ends. Similarly if $D'_{j,0/1}$ contains an interior marked point and asymptotic marker, they determines ones on the glued surface too.\par 
We take $V'' = V \oplus V'$, and take $f'' = f + f'$; note that this makes sense since we assumed for any $v \in V$ and $v' \in V'$, $f(v)$ and $f(v')$ are supported away from the shrunken strip-like ends.

\begin{lem}\label{lem: T big ok}
    If $T, T' \geq 0$ are sufficiently large, then 
    the perturbed Cauchy-Riemann operator $D^{CR} + f''$ for $\bD''_{i+j-1,0/1}$ is surjective. In particular $f'' = f+f'$ does define a valid perturbation datum for the glued operator.
\end{lem}
\begin{proof}
    This is a standard consequence of linear gluing; we just sketch the argument. For simplicity suppose $T=T'$ and the neck is canonically identified with $[-T,T] \times [0,1]$.  We define an approximate inverse $\widetilde{R}_T$  to the operator $D^{CR} + f''$ as follows:
    
    {\small{
    \[
        \centerline{\xymatrix{
            W^{\kappa,p}(D_{i,1},\bR^{2N}, \Lambda) \oplus V \oplus W^{\kappa,p}(D'_{j,0/1}, \bR^{2N}, \Lambda') \oplus V' 
            \ar[d]_{\mathrm{splice_L}} 
            & 
            & 
            \Omega^{0,1}_{W^{\kappa-1,p}}(D_{i,1},\bR^{2N}) \oplus V \oplus \Omega^{0,1}_{W^{\kappa-1,p}}(D'_{j,0/1}, \bR^{2N}) \oplus V' 
            \ar[ll]_{R \times R'} 
            \\
            W^{\kappa,p}(D''_{i+j-1,0/1},\bR^{2N}, \Lambda'') \oplus V'' 
            &
            & 
            \Omega^{0,1}_{W^{\kappa-1,p}}(D''_{i+j-1,0/1},\bR^{2N}) 
            \ar[u]_{\mathrm{split}_L} \ar@{-->}[ll] 
        }}
    \]
    }}
    where the splitting map cuts off vector fields in the neck with a cut-off function $\beta$ supported in a strip of width $1$ centred in the neck and with $|d\beta| = \mathcal{O}(1)$,  but the splicing map superimposes cut-off vector fields which are cut off more slowly, using a cut-off $\alpha$ supported in $[-T/2,T/2] \times [0,1]$ with $|d\alpha| = \mathcal{O}(1/T)$. Then 
    \begin{equation} \label{eqn: split and splice}
    \widetilde{R}_T = (\mathrm{splice}_T) \circ (R \times R') \circ (\mathrm{split}_T)
    \end{equation}
    with $R$ and $R'$ the right inverses to the stabilised operators on $D_{i,1}$ and $D_{j,0/1}$. 
We claim that $\|(D^{CR}+f'')\widetilde{R}_T (\xi) - \xi\| \to 0$ as $T \to \infty$.  In particular, we obtain a right inverse $R_T$ for the glued operator $D^{CR} +f''$
 by setting $R_T := \widetilde{R}_T \circ ((D^{CR}+f'')\widetilde{R}_T)^{-1}$.   To show the claim, one notes that $(D^{CR}+f'')\widetilde{R}_T (\xi)$ agrees with $\xi$ except in the neck, where the terms are controlled by the decay in $|d\alpha|$, cf. \cite[Appendix A]{Barraud-Cornea}, \cite{Bourgeois-Mohnke} for instance. Note that the approximate right inverses are uniformly bounded in $T$ since that holds for all of the maps in \eqref{eqn: split and splice}.  It follows that the glued operator is surjective.
 \end{proof}

 Let $\chi: \bR \to [0,1]$ be a smooth bump function such that $\chi(t) = 0$ if $t \leq 0$ and $\chi(t) = 1$ if $t \geq 1$. For a domain with boundary data, we let $\chi_L: D_{i,1} \to [0,1]$ be the bump function which is 1 away from the strip-like ends, and on $\eps([0,\pm\infty)_s \times [0,1]_t)$, is given by $\chi(\pm (L+1-s))$ (implicitly in this, and similar, formulae we use the value $L_x$ of $L$ at the strip-like end near some puncture $x$).

Let $\xi$ be an element of the kernel of any of the stabilised Cauchy-Riemann operators introduced above, for instance $\ker(D^{CR}+f)$. Exponential decay of $\xi$ shows that the cut-off vector field $\chi_T\xi$, which is  supported in a width one strip $[T,T+1] \times [0,1]$, is close to the actual kernel element $\xi$ itself, so if the cut-off  vector fields span a subspace $\widetilde{\ker}(D^{CR}+f) \subset W^{\kappa,p}(D_{i,1},\bR^{2N})$ then $L^2$-projection is an isomorphism between the actual kernel $\ker(D^{CR}+f)$ and the space spanned by cut-off fields; similarly for kernel vector fields on $D'_{j,0/1}$. There is an inclusion 
\begin{equation} \label{eqn:L2b}
\widetilde{\ker}(D^{CR}+f) \oplus \widetilde{\ker}(D^{CR}+f') \hookrightarrow W^{\kappa,p}(D''_{i+j-1,0/1},\bR^{2N},\Lambda) + V + V'
\end{equation}
which we compose with $L^2$-orthogonal projection to $\ker(D^{CR}+f'') + V''$ (by the identity on the second factor).
\begin{lem}\label{linear gluing}
    For gluing $T \gg0 $ the map 
    \begin{equation} \label{eqn:L2}
    \ker(D^{CR}+f) \oplus \ker(D^{CR}+f') \to \ker(D^{CR}+f'')
    \end{equation}
    given by splicing cut-off versions of the domain elements and then taking $L^2$-orthogonal projection is an isomorphism. 
\end{lem}
\begin{proof}
    Surjectivity of $(D^{CR}+f'')$ shows that its index agrees with the rank of the kernel, and so additivity of the index then shows that the spaces on the two sides of \eqref{eqn:L2} or \eqref{eqn:L2b} have the same rank, so it suffices to see that the map is injective.  This is proved in \cite[Lemma A.17]{ES2}.
\end{proof}

The linear gluing isomorphisms \eqref{eqn:L2} are nearly associative for gluing longer sequences of discs, cf \cite[Sections 8.3 \& 8.4]{Large}.  The essential case is the following. Suppose we have three domains $D_{i,1}$, $D_{j,1}$ and $D_{k,0/1}$ with associated stabilised Cauchy-Riemann data $D_1^{CR}+f$, $D_2^{CR}+f'$ and $D_3^{CR}+f''$.  Since by hypothesis the stabilisation maps $f,f',f''$ are supported away from all strip-like ends and the one-forms $Y, Y', Y''$ agree exactly with $dt$ in all strip-like ends, for given length parameters $T,T'$ there is a unique glued operator
\[
D_{T,T'}^{CR} + (f+f'+f'').
\]
We now have two isomorphisms
    \[
    \ker(D_1+f) \oplus \ker(D_2+f') \oplus \ker(D_3+f'') \longrightarrow \ker(D_{T,T'} + f+f'+f'')
    \]
The first $\rho(\mathrm{pre}(1,2,3))$ is     obtained by cutting off elements of all three domain kernels and taking the $L^2$-orthogonal projection to the range. The second $\rho(\mathrm{pre}(\mathrm{pre}(1,2),3))$ is obtained by composing the isomorphisms of Lemma \ref{linear gluing} for gluing $D_1+f$ and $D_2+f'$ with length $T$ and then gluing the resulting operator $D_{12}+f+f'$ to $D_3+f''$ with length $T'$ (so this involves two consecutive orthogonal projections). Let $F = f+f'+f''$.

\begin{lem}\label{lem:associative linear gluing}
    For $T$ and $T'$ sufficiently large, the isomorphisms 
    \[
    \ker(D_1+f) \oplus \ker(D_2+f') \oplus \ker(D_3+f'') \longrightarrow \ker(D_{T,T'} + F)
    \]
    given by $\rho(\mathrm{pre}(1,2,3))$ and $\rho(\mathrm{pre}(\mathrm{pre}(1,2),3))$ are isotopic through isomorphisms canonically up to contractible choice.
\end{lem}

\begin{proof}
    For each of the operators we have a vector space $\widetilde{\ker}(D+f)$ spanned by the vector fields cut-off in all strip-like ends. There is an   inclusion
    \begin{equation} \label{pre}
    \mathrm{pre}: \widetilde{\ker}(D_1+f) \oplus \widetilde{\ker}(D_2+f') \oplus \widetilde{\ker}(D_3+f'') \hookrightarrow W^{\kappa,p}(D_{i+j+k-1,1},\bR^{2N})
    \end{equation}
    which is close to an isometry to its image when $T,T' \gg0$ because of the exponential decay of vector fields in the kernels.  Both the given isomorphisms lie in a small ball (in the space of linear embeddings of the spaces in \eqref{pre}) of size controlled by $T,T'$ around $\mathrm{pre}$ in operator norm. That is a convex and hence contractible condition. 
    Compare to \cite[Section 9.4]{Hutchings-Taubes} which gives a detailed argument for gluing cylinders in the presence of non-trivial obstruction bundles (i.e. without stabilising the operators). 
\end{proof}

\subsection{A simplicial set of pre-stable discs}

All the definitions of Section \ref{subsec:abstract discs} make sense if the domain is a pre-stable disc (i.e. a tree of discs meeting at nodes): thus we may talk about \emph{abstract pre-stable discs}, meaning a pre-stable disc such that each component is equipped with the additional data required to make it into an abstract disc (requiring the boundary conditions and homotopy lifts to $\Theta$ to agree at either side of each node), and their ($\Theta$-oriented) puncture data.

A smooth function on a nodal curve is a continuous function whose pullback to the normalisation is smooth. With this convention, one also has well-defined notions of smooth fibre bundle, smooth section, etc over a nodal curve. Sobolev spaces $W^{k,\kappa}$ of functions on a nodal curve are well-defined provided $k\kappa > 2$ so the underlying maps are continuous, by the Sobolev embedding theorem. The Cauchy-Riemann operator on a pre-stable disc $D$ is then given by the direct sum of the Cauchy-Riemann operators for the components, defined as a map of  Sobolev spaces $W^{2,\kappa}(D) \to \Omega^{0,1}_{W^{2,\kappa-1}}(D)$, where by definition of the Sobolev spaces all maps and sections of bundles are continuous across nodes.  In this way, we can also extend the definitions from Section \ref{subsec: CR data} to prestable discs.

Equip each standard simplex $\Delta^k$ with collar neighbourhoods over all boundary faces, compatibly with the chosen collars on lower-dimensional standard simplices. Fix a tangential structure $\Theta \to \widetilde{U/O}(N)$ of rank $N$.

\begin{defn} \label{def: spac dis}
    For $i\geq 0$ and $j\in \{0,1\}$, $V$ a finite-dimensional real inner product space, and $\bE$ a set of graded and $\Theta$-oriented puncture data, there is 
    a simplicial set $\bU_{i,j}^{V,\Theta}(\bE)_{\bullet}$ whose $k$-simplices consist of continuously varying families of $\Theta$-oriented abstract pre-stable discs $\scrD \to \Delta^k$ (with all the extra data), with Cauchy-Riemann data having perturbation domain $V$, and for which the restriction of the family to a collar neighbourhood of a boundary facet of the simplex agrees with the family obtained by pullback under  gluing as in Section \ref{sec: glue}.
\end{defn}

Details of the construction of such a simplicial set are carried out in \cite[Section 10]{Abouzaid:axiomatics} and \cite[After (2.14)]{AGV}.

 \begin{defn}
     For a simplicial set $\mathcal{S}_\bullet$ we will write $\mathcal{S}$ for the topological space given by its geometric realisation.
 \end{defn}

\begin{lem}\label{lem:upshot}
    There is a `index' vector bundle over the geometric realisation
    \begin{equation}
        \bV_{ij}^{V, \Theta}(\bE) \to \bU_{ij}^{V, \Theta}(\bE)
    \end{equation}
    with fibre \eqref{eq: ind} over non-nodal discs. There are then maps of spaces, and isomorphisms of vector bundles covering them:
    \begin{itemize}
        \item Associated to isometric embeddings $V \to W$ with orthogonal complement $(V \subset W)^\perp$:
        \begin{itemize}
            \item[-] Maps of spaces $\bU^V \to \bU^W$, suitably associative with respect to composable pairs of isometric embeddings $V \to W \to Z$.
            \item[-] Isomorphisms of vector bundles $\bV^V \oplus (V\subset W)^\perp \to \bV^W$ over $\bU^V$
            such that, for inclusions $V \to W \to Z$ the following diagram commutes:
            \begin{equation}
                \xymatrix{
                    \bV^V \oplus (V\subset W)^\perp \oplus (W\subset Z)^\perp
                    \ar[r]
                    \ar[d]
                    &
                    \bV^V \oplus (V\subset Z)^\perp 
                    \ar[d]
                    \\
                    \bV^W \oplus (W\subset Z)^\perp
                    \ar[r]
                    &
                    \bV^Z
                }
            \end{equation}
            where we use the natural identification $(V\subset W)^\perp \oplus (W\subset Z)^\perp \cong (V\subset Z)^\perp$. 
            
        \end{itemize}
        \item Associated to concatenation of Riemann surfaces, i.e. gluing of tuples of puncture data. Precisely, there are maps
        \[
         \rho_k: \bU_{i,1}^{V,\Theta}(\bE) \times \bU_{j,0/1}^{W,\Theta}(\bE') \to \bU^{V+W,\Theta}_{i+j-1,0/1}(\bE\#_k \bE')
    \]
        \end{itemize}
    which are covered by isomorphisms of index bundles
     \[
    st_k: \bV_{i,1}^{V,\Theta} \oplus \bV_{j,0/1}^{W,\Theta} \to \rho_k^* \bV_{i+j-1,0/1}^{V+W,\Theta}
    \]
    and which are associative. Both the $\rho_k$ and $st_k$ commute with the maps associated to isometric embeddings $V \to V'$ or $W \to W'$.

\end{lem}

\begin{proof}
    Concatenation of Riemann surfaces takes two Riemann surfaces and concatenates them into a nodal Riemann surface. If both Riemann surfaces are additionally equipped with Cauchy-Riemann data, the concatenated Riemann surface inherits Cauchy-Riemann data by taking it to be given by the original Cauchy-Riemann data on each component; this defines the maps $\rho_k$.

    (\ref{eq: ind}) determines the fibre of the index bundle at each point. Existence of the structure of a vector bundle is largely the output of the discussion of gluing. Over the subspaces $\bU^{V,\Theta,l}_{ij}(\bE)$ whose Riemann surfaces have a fixed number $l$ of nodes, over each point $\bV$ is defined to be the direct sum of (\ref{eq: ind}) over each component of the corresponding domain. These subspaces are preserved under concatenation and over these subspaces the isomorphisms $st_k$ are given by the identity.
    
    Consider the subspaces $\bU^{V,\Theta,[l,l']}_{ij}(\bE)$ consisting of abstract domains with number of nodes lying in the interval $[l, l']$. We extend the index bundles and $st_k$ over these $\bU^{V, \Theta,[l,l']}_{ij}(\bE)$ by inducting on $l'-l$ and using the contractibility of the spaces of choices from Lemma \ref{lem:associative linear gluing} at each stage, noting that concatenation strictly increases the number of nodes.

     Let $V \to W$ be an isometric embedding of finite-dimensional real inner product spaces. Then taking $(V\subset W)^{\perp}$ to be the orthogonal complement of $V$ in $W$, there are maps of simplicial sets $(\bU_{i,j}^{V,\Theta})_{\bullet} \to (\bU^{W,\Theta}_{i,j})_{\bullet}$ sending the tuple $(\ldots, V, f)$ to $(\ldots, W, f+0)$, where
    $$f+0: W = V \oplus (V\subset W)^{\perp} \to \Omega^{0,1}_{W^{2,\kappa-1}} \left(D_{i,j}, \bC^N\right)$$

    It is immediate that this respects index bundles.

\end{proof}

\begin{rmk}\label{rmk: func}
    All of this structure can be made to be functorial under morphisms of tangential structures, i.e. morphisms of based spaces over $\widetilde{U/O}(N)$, by first carrying out the construction for $\Theta=\widetilde{U/O}(N)$ and then pulling back.

\end{rmk}

\subsection{Stabilisation}\label{sec: Nstab}
    We now assume that we are given a sequence $\{\Theta(N)\}_{N \geq N_0}$ of tangential structures of each rank $N \geq N_0$ (for some fixed $N_0$), with maps between them, fitting into commutative diagrams of the following form:
    \begin{equation}\label{eq: com}
        \xymatrix{
            \Theta(N) 
            \ar[r]
            \ar[d]
            &
            \Theta(N+1)
            \ar[r]
            \ar[d]
            &
            \ldots
            \\
            \widetilde{U/O}(N)
            \ar[r]
            &
            \widetilde{U/O}(N+1)
            \ar[r]
            &
            \ldots
        }
    \end{equation}
    
    We construct stabilisation maps $\bU^{V, \Theta(N)}_{ij}(\bE) \to \bU^{V, \Theta(N+1)}_{ij}(\bE)$, compatible with all the structure from Lemma \ref{lem:upshot}. We first require the following lemma:

    \begin{lem}\label{lem: N=1 disc CR}
        Let $\bD$ be a punctured disc with real boundary conditions of rank 1 and Cauchy-Riemann data. Let $D^{CR}$ be the corresponding (unperturbed) Cauchy-Riemann operator.

        Assume the index of $D^{CR}$ is 0. Then $D^{CR}$ is bijective.
     
    \end{lem}
    \begin{proof}
        If $\bD$ has no boundary punctures, this is a consequence of \cite[Theorem C.1.10(iii)]{McDuff-Salamon}. 
        
        Now suppose $\bD$ has boundary punctures, and suppose $\xi$ is a nonzero element of the kernel of $D^{CR}$. For each puncture $p$ of $\bD$, choose a disc $\bD^p$ with real boundary conditions with one puncture, equipped with the puncture data of $p$, labelled as an output if $p$ is an input and vice versa. Letting $\bD'$ be the disc formed from gluing all of these discs together, we may choose the $\bD^p$ such that the Maslov index of $\bD'$ is 0.

        Choose Cauchy-Riemann data for each $\bD^p$. Let $0^p$ be the zero section in each $W^{2,\kappa}(\bD^p,\bC)$. Since $\xi$ decays to 0 near each puncture, we may glue $\xi$ to all the $0^p$ to obtain an element $\xi'$ of the kernel of the glued Cauchy-Riemann operator over $\bD'$. For sufficiently large gluing parameters, since $\xi$ is nonzero, $\xi'$ is too. But this contradicts the case with no punctures, which we have already addressed.
    \end{proof}

    We choose a framed puncture datum (meaning puncture datum oriented with respect to the tangential structure given by a point) $\bF = (\tilde u^\bF, \tilde v^\bF)$ of rank 1 and Maslov index 0.

    Let $\bP^{\Theta(N)}$ be the set of $\Theta(N)$-oriented puncture data.
    
    We define a function 
    \begin{equation}
        \bF \oplus \cdot: \bP^{\Theta(N)} \to \bP^{\Theta(N+1)}
    \end{equation}
    preserving Maslov indices, by sending puncture data $\bE$ to $\bF \oplus \bE$ equipped with the $\Theta(N+1)$-orientation given by applying the map $\Theta(N) \to \Theta(N+1)$; the fact that $\bF$ is framed along with commutativity of (\ref{eq: com}) guarantee that this does indeed define $\Theta(N+1)$-oriented puncture data.

    \begin{prop}\label{prop: Nstab}
        There are (cofibrant) maps of spaces
        \begin{equation*}
            \Xi: \bU^{V,\Theta(N)}_{ij}(\bE) \to \bU^{V,\Theta(N+1)}_{ij}(\bF \oplus\bE)
        \end{equation*}
        and isomorphisms of vector bundles covering them
        \begin{equation*}
            \bV^{V,\Theta(N)}_{ij}(\bE) \to \bV^{V,\Theta(N+1)}_{ij}(\bF \oplus\bE)
        \end{equation*}
        which are compatible with all the structure in Lemma \ref{lem:upshot}, meaning the following diagrams of spaces commute:

        \begin{equation}\label{eq:1}
            \xymatrix{
                \bU^{V,\Theta(N)} 
                \ar[r]
                \ar[d]
                &
                \bU^{V,\Theta(N+1)}
                \ar[d]
                \\
                \bU^{W,\Theta(N)}
                \ar[r]
                &
                \bU^{W,\Theta(N+1)}
            }
        \end{equation}
        
        where $V \to W$ is an isometric embedding, and
        \begin{equation}\label{eq:2}
            \xymatrix{
                \bU^{V,\Theta(N)}_{i1}(\bE) \times \bU^{V',\Theta(N)}_{j,0/1}(\bE')
                \ar[r]
                \ar[d]
                &
                \bU^{V,\Theta(N+1)}_{i1}(\bF \oplus\bE) \times \bU^{V',\Theta(N+1)}_{j,0/1}(\bF \oplus\bE')
                \ar[d]
                \\
                \bU^{V\oplus V', \Theta(N)}_{i+j-1}(\bE \#_k \bE') 
                \ar[r]
                &
                \bU^{V\oplus V', \Theta(N+1)}_{i+j-1}(\bF \oplus(\bE \#_k \bE')) 
            }
        \end{equation}
        as well as the corresponding diagram of vector bundles covering them.
        
    \end{prop}
    \begin{proof}
        We first construct the maps of spaces, by constructing maps on the level of the underlying simplicial sets. Let $k \geq 0$ and fix a $k$-simplex $\sigma$ in the set of $k$-simplices, $\bU^{V,\Theta(N)}_{ij}(\bE)_k$. 

        We choose framed real boundary conditions of rank 1 over the corresponding family of domains, with puncture data $\bF$ at each puncture. We also choose Cauchy-Riemann data over this family; we may make all these choices since the space of choices is  contractible (and in particular, if it has already been chosen over some sub-simplicial complex of $\partial \Delta^k$, we can extend that choice). By Lemma \ref{lem: N=1 disc CR} we may take the perturbation data to be 0. 

        Now direct summing all of this data with the data specified by $\sigma$ defines a simplex in $\bU^{V,\Theta(N+1)}_{ij}(\bF \oplus \bE)_k$. This is compatible with (\ref{eq:1}).

        Since the space of choices at each stage is contractible, we may perform this construction for each (nondegenerate) $k$-simplex, inductively in $k$. By performing each step for fixed $k$ inductively on the maximum number of nodes of any domain in $\sigma$ (similarly to the argument in Lemma \ref{lem:upshot}), we may guarantee that (\ref{eq:2}) commutes.

        The map on simplicial sets we have constructed is levelwise injective, and so the induced map on realisations is cofibrant.

        By Lemma \ref{lem: N=1 disc CR}, the disc of rank 1 that we direct sum with does not contribute to the index bundle, so we may take the map $\bV^{\Theta(N)} \to \bV^{\Theta(N+1)}$ to send $s$ to $0 \oplus s$; this is compatible with all the relevant structure.
 \end{proof}
 
    We write $\bU^{V, \Theta(\infty)}_{ij}(\bE)$ for the colimit of $\bU^{V, \Theta(N)}_{ij}(\bE)$ as $N$ goes to $\infty$. By cofibrancy of the stabilisation maps, this computes the homotopy colimit. 
    \begin{defn}\label{rmk: N inf fin}
        Let $\Theta \to \widetilde{U/O}$ be a tangential structure of rank $\infty$. Assume the map is a fibration; otherwise, we first choose a fibrant replacement. Then we may form a sequence $\{\Theta(N)\}_N$ by setting $\Theta(N)$ to be the pullback of this fibration along the inclusion $U/O(N) \to U/O$. We define $\bU^{V, \Theta}_{ij}(\bE)$ to be $\bU^{V, \Theta(\infty)}_{ij}(\bE)$ build from this sequence.
    \end{defn}

    The index bundles for each $N<\infty$ together assemble to form an index bundle $\bV^{V, \Theta(\infty)}_{ij}(\bE)$. The compatibility of the stabilisation maps $\Xi$ from Proposition \ref{prop: Nstab} shows that:
    \begin{lem}
        The statement of Lemma \ref{lem:upshot} holds for tangential structures $\Theta$ of infinite rank.
    \end{lem}

\subsection{Homotopy type of spaces of abstract discs}

Fix $\Theta$ a tangential structure of rank $1 \leq N \leq \infty$.

\begin{lem}\label{lem: V conn est}
    Let $V \to W$ be an isometric embedding of vector spaces. Then the connectivity of the map $\bU^{V,\Theta}_{ij}(\bE) \to \bU^{W,\Theta}_{ij}(\bE)$ goes to $\infty$ as the dimension of $V$ does.
\end{lem}
\begin{proof}
    For a fixed vector space $B$, the connectivity of the space of surjective linear maps $A \to B$ goes to $\infty$ as $\operatorname{Dim}(A)$ does, since this is homotopy equivalent to the space of $\operatorname{Dim}(B)$-frames in $A$.

    For a $k$-parameter family of Cauchy-Riemann operators of fixed index $p$, after a generic perturbation of the 1-form the cokernel has rank at most $k-p$: this follows from the same argument as \cite[Theorem 3]{Wang:Stability}, which reduces this to the finite-dimensional case. 

    Then for fixed $k$ and $V$ sufficiently high-dimensional, there is always a (highly-connected family of) surjections from $V$ to the cokernel of the Cauchy-Riemann operator. The lemma then follows.
    
\end{proof}

The space $\bU_{i,j}^{V,\Theta}(\bE)$ involves a fixed choice of finite-dimensional vector space $V$. Using Lemma \ref{lem:upshot}, there is a topological space $\bU_{i,j}(\Theta)$ in which one passes to the homotopy colimit over finite-dimensional inner product spaces $V$ ordered by isometric embeddings. Lemma \ref{lem: V conn est} and \cite[Lemma 2.5.1]{Brun} together imply that the connectivity of the map $\bU_{i,j}^{V, \Theta} \to \bU_{i,j}(\Theta)$ goes to $\infty$ as $\dim V \to \infty$.

\begin{lem}\label{lem:htpy type}

    When $i+j>0$, there is a homotopy equivalence
    \begin{equation*}
        \bU_{i,j}(\Theta) \to (\Omega \Theta)^{i+j}
    \end{equation*}

There is also a homotopy equivalence
\[
\bU_{0,0}(\Theta) \to \mathcal{L}\Theta
\]
where $\Omega \Theta$ and $\mathcal{L}\Theta$ denote the based and free loop spaces respectively.
\end{lem}
\begin{proof}
    First assume $N < \infty$. Assume $i+j>0$. Evaluating the homotopy lift $H$ along each of the $i+j$ components of the boundary gives a map
    \begin{equation}\label{eq: ev}
        \bU_{i,j}(\Theta) \to \cP_{h_{u_{out}},h_{u_1}} \Theta \times \ldots \times \cP_{h_{u_i},h_{v_{out}}} \Theta
    \end{equation}
    where $\cP_{ab}\Theta$ is the space of paths from $a$ to $b$ in $\Theta$. Picking paths from $a$ and $b$ to the basepoint induces an equivalence from $\cP_{ab}\Theta$ and $\Omega \Theta$. The connectivity of the composition of (\ref{eq: ev}) with the inclusion $\bU^{V,\Theta}_{i,j} \to \bU_{i,j}(\Theta)$ goes to $\infty$ as $\dim V \to \infty$, by the same connectivity estimate on spaces of surjective linear maps as above, as well as contractibility of spaces of Cauchy-Riemann data; it follows that (\ref{eq: ev}) is an equivalence.
    
    The case $i=j=0$ is similar.

    If $N=\infty$, the result follows from the finite rank case since the above equivalences are compatible with stabilisation in $N$ and $\Theta$ is the (homotopy) colimit of $\Theta(N)$ as $N \to \infty$.
\end{proof}

We have dropped the puncture data from the notation, to reflect that it does not affect the homotopy type.

    \begin{cor}\label{cor: fr cont}
        Each $\bU^{fr}_{i,1}$ is contractible.
    \end{cor}
    
   Consider the tangential structure of infinite rank $\Theta=\widetilde{U/O}$. For $i=0$ and $j=0$, composing the classifying map for the index bundle $\bV_{0,1}^\Theta$ with an inverse to the equivalence from Lemma \ref{lem:htpy type} gives a well-defined map in the homotopy category $\Omega \Theta \to BO \times \bZ$.
    \begin{lem}\label{lem: conn est}
        The classifying map for $ \bV^{\widetilde{U/O}}$:
        $$\bU_{01}(\widetilde{U/O}) \to BO$$
        is an equivalence.  
    \end{lem}
    \begin{proof}
        Using the fact that the inclusion $U/O(N) \to U/O$ is $(N-3)$-connected and taking $N \to \infty$, this follows from \cite[Theorem A]{deSilva}; see also (the proof of) \cite[Proposition 3.15]{P}.
    \end{proof}

\subsection{Adaptation to tangential pairs}\label{sec: 2tan}
    In this section we explain how to adapt the rest of Section \ref{sec: abstr dis} to incorporate tangential pairs as in Section \ref{sec: Lag Floer conj}. As in Definition \ref{def: tan}, we will generally work with tangential pairs satisfying an additional orientation and gradedness condition.
    \begin{rmk}
        Recall from \cite[Section 11(i)]{Seidel:book} that there are two Lie groups $\operatorname{Pin}^\pm(n)$ which are double covers of $O(n)$. We write $\operatorname{Pin}(n)$ to mean either $\operatorname{Pin}^+(n)$ or $\operatorname{Pin}^-(n)$.

        They are always diffeomorphic, but not isomorphic as groups; for example $\operatorname{Pin}^+(1) \cong C_2 \times C_2$ but $\operatorname{Pin}^-(1) \cong C_4$. 
    \end{rmk}
    \begin{defn}
        Let $1 \leq N \leq \infty$. Let $S_\pm U(N)$ be the group of unitary matrices of determinant $\pm 1$. Its classifying space $BS_\pm U(N)$ is the homotopy fibre of the map $BU(N) \to K(\bZ,2)$ classifying twice the universal first chern class. Note that since $2c_1(F \otimes \bC)=0$ for any real vector bundle $F$, the map $BO(N) \to BU(N)$ admits a natural lift to a map $BO(N) \to BS_\pm U(N)$.

        A \emph{(graded) tangential pair (of rank $N$)} consists a pair of based spaces ($\Theta, \Phi)$, along with maps fitting into a commutative diagram:
        \begin{equation*}
            \xymatrix{
                \Theta 
                \ar[r]
                \ar[d]
                &
                \Phi
                \ar[d]
                \\
                BO(N) 
                \ar[r]
                &
                BS_\pm U(N)
            }
        \end{equation*}
        We assume that the map $\Theta \to \Phi$ is a fibration. As with tangential structures, we work with \emph{graded} tangential pairs, meaning $\Phi$ lives over $BS_\pm U(N)$ instead of $BU(N)$.

        An \emph{oriented} tangential pair (of rank $N$) consists of a pair of spaces $\Theta, \Phi$ fitting into a commutative diagram:
        \begin{equation*}
            \xymatrix{
                \Theta 
                \ar[r]
                \ar[d]
                &
                \Phi
                \ar[d]
                \\
                B\operatorname{Pin}(N) 
                \ar[r]
                &
                BS_\pm U(N)
            }
        \end{equation*}
        such that the map $\Theta \to \Phi$ is 1-connected.
    \end{defn}
    For concreteness, for $N<\infty$ we may use the specific model for $BU(N)$ given by the space of $N$-dimensional complex subspaces of $\bC^\infty$, and $BO(N)$ to be the space of $N$-dimensional real subspaces of $\bR^\infty$. These each carry a tautological vector bundle. There are natural stabilisation maps $BU(N) \to BU(N+1)$ sending $V$ to $\bC \oplus V$, and similarly for $BO(N) \to BO(N+1)$.

    \begin{rmk}
        The homotopy fibre of $BO(N) \to BS_\pm U(N)$ is exactly $\widetilde{U/O}(N)$, and the homotopy fibre of $B\operatorname{Pin}(N) \to BS_\pm U(N)$ is exactly $\widetilde{U/O}^{or}(N)$.

        In particular, from a(n oriented) tangential pair $(\Theta, \Phi)$ where $\Phi$ is a single point, we obtain a(n oriented) tangential structure in the sense of Definition \ref{def: tan}.
    \end{rmk}

    \begin{defn}
        Let $(\Theta, \Phi)$ be a tangential pair of rank $N<\infty$. A \emph{($\Theta,\Phi)$-orientation} on a domain $D_{i,j}$ consists of 
        \begin{itemize}
            \item Strip-like ends near each boundary puncture, as in Definition \ref{def: boun dat}.
            \item A complex vector bundle $E$ of rank $N$ over $D_{i,j}$, classified by a map $D_{i,j} \to BU(N)$.
            \item Complex vector spaces $E_x \in BU(N)$ for each puncture, and trivialisations $E \cong E_x$ over the strip-like end near each puncture $x$.
            \item A real vector bundle $F$ of rank $N$ over $\partial D_{i,j}$, classified by a map $\partial D_{i,j} \to BO(N)$ and locally constant in the strip-like ends, along with an isomorphism $F \otimes \bC \cong E$ over $\partial D_{i,j}$, compatible with the classifying maps.
            \item Compatible homotopy lifts (defined as in Definition \ref{def: hom lift}) of the classifying spaces of $F$ and $E$ to $\Theta$ and $\Phi$ respectively.
        \end{itemize}
        Cauchy-Riemann data may be defined similarly to Section \ref{subsec: CR data}: this consists of choosing a complex structure $J$ on $E$, a Hermitian connection on $E$ (so one can take $\overline \partial_J$ to be the $(0,1)$-part of covariant differentiation), a 1-form $Y \in \Omega^{0,1}(D_{i,j},E)$, and a metric $g$ as before. This data gives rise to an associated Cauchy-Riemann operator, which is the linear map:
        \begin{equation}\label{eq: CR pair}
            D^{CR} := \overline \partial_J +Y: W^{2,\kappa}(D_{i,j},E,F) \to \Omega^{0,1}_{W^{2, \kappa-1}}(D_{i,j}, E)
        \end{equation}
        Perturbation data is similarly defined as a linear map $f$ from a finite-dimensional vector space $V$ to the codomain of (\ref{eq: CR pair}) such that $D^{CR}+f$ is surjective.

    \end{defn}
    One can adapt the definition of puncture data to incorporate tangential pairs similarly.

    Using these ingredients, for $N<\infty$ one may define spaces of abstract discs with $(\Theta,\Phi)$-orientations and Cauchy-Riemann data $\bU^{V,\Theta,\Phi}_{ij}(\bE)$ similarly to Definition \ref{def: spac dis}; these spaces admit index bundles $\bV^{V,\Theta,\Phi}_{ij}(\bE) \to \bU^{V,\Theta,\Phi}_{ij}(\bE)$, as well as all the same maps as in Lemma \ref{lem:upshot} (and by the same proof). 

    Similarly, these spaces admit the same stabilisation maps $\Xi$ as in Proposition \ref{prop: Nstab} (again by the same proof), so we may incorporate tangential pairs of rank $\infty$ the same way as in Section \ref{sec: Nstab}.

    Letting $\bU_{i,j}(\Theta,\Phi) = \operatorname{hocolim}_V \bU_{i,j}^{V,\Theta,\Phi}$ (for fixed puncture data $\bE$), the analogue of Lemma \ref{lem:htpy type} is:

    \begin{lem}\label{lem: htpy type pairs}
        When $i+j > 0$, $\bU_{ij}(\Theta,\Phi)$ is the homotopy fibre of the map:
        \begin{equation*}
            (\Omega \Theta)^{i+j} \to \Omega \Phi
        \end{equation*}
        $\bU_{00}(\Theta, \Phi)$ is the homotopy pullback of the following diagram:
        \begin{equation}\label{eq: pair pull}
            \xymatrix{
                &
                \Phi 
                \ar[d]
                \\
                \cL \Theta
                \ar[r]
                &
                \cL \Phi
            }
        \end{equation}
    \end{lem}
    \begin{proof}
        For $i+j>0$, the evaluation map from $\bU_{ij}(\Theta,\Phi)$ lands in the space of $(i+j)$-tuples of paths in $\Theta$ (with appropriate fixed endpoints), along with a contraction of the loop obtained by concatenating these paths and applying the map $\Theta \to \Phi$. This space is a model for the homotopy fibre as in the statement of the lemma. This map is an equivalence by the same argument as in Lemma \ref{lem:htpy type}.

        The case $i=j=0$ is similar.
    \end{proof}

\subsection{Interior punctures}\label{sec: int pun}

    There is a similar story involving interior punctures, which we sketch here. Fix a tangential pair $(\Theta, \Phi)$ of rank $N < \infty$.

    \begin{defn}
        An \emph{abstract disc/rational curve with interior punctures} consists of a disc or rational curve $\Sigma$, with a two finite ordered lists of disjoint input punctures $p_1, \ldots, p_i \in \partial \Sigma$ and $q_1, \ldots, q_j \in \Sigma^\circ$, along with either an output puncture $r$ (on the interior or boundary) or marked point with asymptotic marker $(r,v)$.

        \emph{Boundary data}, \emph{$(\Theta,\Phi$)-orientations} and \emph{perturbation data} (for a fixed vector space $V$) on $\Sigma$ are defined just as in the disc case, with strip-like ends over interior punctures as well as boundary punctures. 
        
    \end{defn}
    Now the Cauchy-Riemann operator is a map
    \begin{equation*}
        D^{CR}: W^{2,\kappa}(\Sigma; E,F) \to \Omega^{0,1}_{W^{2,\kappa-1}}(\Sigma, E)
    \end{equation*}
    where $E$ and $\Lambda$ are the real and complex part of the boundary conditions respectively. Gluing at interior punctures works essentially the same as with boundary punctures, and there is an analogous space of such Riemann surfaces with interior punctures and possibly interior nodes, with an analogous version of Lemma \ref{lem:upshot}. One can compute the homotopy type of these spaces analogously to Lemma \ref{lem: htpy type pairs}; for example, the space of $(\Theta, \Phi)$-oriented rational curves with 0 input punctures and 1 output puncture is homotopy equivalent to $\Omega^2\Phi$.

\subsection{Thom ring spectra}\label{sec: Thom}
    
    Let $(\Theta,\Phi)$ be a tangential pair of rank $1 \leq N \leq \infty$. Roughly, we may define spectra $\bT^{\Theta,\Phi}_{ij}(\bE)$ to be the Thom spectrum of the index bundle over $\bU_{ij}^{V, \Theta,\Phi}(\bE)$ as $V \to \infty$.
    
    More precisely, we may define it to be the orthogonal spectrum with $V^{th}$ space given by the Thom space: 
    \begin{equation*} 
        \operatorname{Thom}\left(\bV^{V, \Theta,\Phi}_{ij}(\bE) \to \bU^{V, \Theta,\Phi}_{ij}(\bE)\right)
    \end{equation*}
    which may be equipped with natural $O(V)$-action and structure maps.

    By Lemma \ref{lem:upshot}, these admit suitably associative maps 
    \begin{equation}
        \rho_k:\bT_{i,1}^{\Theta,\Phi}(\bE) \wedge \bT_{j, 0/1}^{\Theta,\Phi}(\bE') \to \bT^{\Theta,\Phi}_{i+j-1,0/1}(\bE \#_k \bE')
    \end{equation}

    In particular, the collection $\left\{\bT_{i1}^{\Theta,\Phi}(\cdot)\right\}_{i \geq 1}$ forms a nonsymmetric coloured operad in spectra (with colours given by the set of puncture data), which $\{\bT_{01}^{\Theta,\Phi}(\cdot)\}$ is a category over.

    In particular, by Corollary \ref{cor: fr cont}, when $(\Theta,\Phi)=fr$ corresponds to framings, this operad is an $A_\infty$-operad. Choosing some fixed puncture data $\bE_*$, we obtain a spectrum $R := \bT_{01}^{\Theta,\Phi}(\bE_*)$ for each $(\Theta,\Phi)$; by functoriality under the tangential structure as in Remark \ref{rmk: func}, $R$ is a module over the operad $\{\bT^{fr}_{i1}(\cdot)\}_{i \geq 1}$ and so is an $A_\infty$ ring spectrum, independent of the choice of puncture data up to equivalence. 

    \begin{rmk}
        Using methods of Section \ref{sec: int pun}, one can form the Thom spectrum $A$ of the index bundle over the space of $\Phi$-oriented rational curves with no inputs and one output puncture, and using gluing and functoriality as above, equip it with the structure of a $\bE_2$ ring spectrum (for some model of the $\bE_2$ operad).

        It is expected that under Conjecture \ref{conj: SH}, $SH(X; A)$ can be given the structure of an $\bE_2$-ring spectrum over the $\bE_2$-algebra $A$: for symplectic cohomology over a Novikov ring, this $\bE_2$-structure has been constructed in detail in \cite{AGV}.
    \end{rmk}

\section{Background on flow categories}\label{sec: back}

\subsection{Manifolds with faces}

In this section we recap some definitions regarding manifolds with corners from \cite[Section 3]{PS}. Let $M$ be a manifold with corners.
\begin{defn}
    A \emph{codimension $k$ face} of $M$ is a disjoint union of closures of components of the set of codimension $k$ points in $\partial M$. A \emph{boundary face} is a codimension 1 face.
    
    A \emph{system of boundary faces} for $M$ consists of a collection of boundary faces $\{B_i\}_i$, whose union covers $\partial M$ but which have disjoint interiors.
    
    $M$ is a \emph{manifold with faces} if it admits a system of boundary faces. This is equivalent to any point $p \in \partial M$ of codimension $k$ touching exactly $k$ boundary faces. We define the \emph{carapace} of $M$ to be the union of its compact boundary faces.
\end{defn}
Any face of a manifold with faces is itself naturally a manifold with faces.\par 
Let $M$ be a manifold with faces, and choose a Riemannian metric on $M$ which makes all boundary faces intersect orthogonally. Over each boundary face $F \subseteq M$, there is now a canonical inward-pointing normal vector $\nu^F$ along $F$, which is tangent to $\partial M$ along $\partial F$.
\begin{defn}
    A \emph{framed function} on a manifold with faces $M$ is a pair $(f, s)$, where $f: M \to \bR$ is a smooth map and $s: \bR \to TM|_{f^{-1}\{1\}}$ is a map of vector bundles over $f^{-1}\{1\}$, such that
    \begin{itemize}
        \item 1 is a regular value of $f$
        \item $f$ is negative on the carapace of $M$.
        \item $f^{-1}(-\infty, 2]$ is compact.
        \item $df$ vanishes along $\nu^F$ for each boundary face $F \subseteq M$.
        \item $s|_F$ factors through $TF \subseteq TM$ over each boundary face $F\subseteq M$.
        \item The composition $\bR \xrightarrow{s} TM|_{f^{-1}\{1\}} \xrightarrow{df} \bR$ is the identity.
    \end{itemize}
\end{defn}
By \cite[Proposition 3.29]{PS}, framed functions exist on any manifold with faces $M$, and if framed functions are already specified over some of the faces of $M$, there exists a framed function on $M$ extending the given ones. The restriction of a framed function to a face is also a framed function.

If $M$ is a (possibly non-compact) manifold with faces and $(f, s)$ is a framed function, then $f^{-1}\{1\}$ is a compact manifold with faces, and there is a canonical isomorphism of vector bundles over $f^{-1}\{1\}$
$$S = di + s: Tf^{-1}\{1\} \oplus \bR \to TM$$
where $i: f^{-1}\{1\} \hookrightarrow M$ is the inclusion map. These are compatible on faces.

A \emph{collar neighbourhood} of a codimension $k$ face $F \subseteq M$ is an embedding $\cC: F \times [0, \eps)^k \hookrightarrow M$ onto a neighbourhood of $F$; we will always assume that the inwards derivative in each of the $k$ directions is given by the inwards-pointing normals $\nu^{G_i}$ where $G_1, \ldots, G_k$ are the boundary faces touching $F$. By \cite[Lemma 3.19]{PS}, there always exist collar neighbourhoods of all faces $F \subseteq M$, compatible on overlaps; similarly to before if some of them are already specified the rest can be chosen extending this choice, and these collar neighbourhoods are unique up to isotopy. Given all of this, we will often implicitly choose such a system of collar neighbourhoods.

\subsection{Unoriented flow categories}\label{sec: unor flow cat}

In this section, we briefly recall some notions from \cite[Sections 3, 4 \& 5]{PS}; see \textit{loc. cit.} for more details.
\begin{defn}
    A \emph{flow category} $\cF$ consists of a finite set (which we also denote by $\cF$), a function $|\cdot|: \cF \to \bZ$, and compact manifolds with faces $\cF_{xy}$ for each $x, y \in \cF$ of dimension $|x|-|y|-1$, along with maps 
    $$c: \cF_{xy} \times \cF_{yz} \to \cF_{xz}$$
    satisfying a suitable associativity condition, such that each map $c$ is the inclusion of a (codimension 1) boundary face of $\cF_{xz}$, and that varying these over all $y \in \cF$ provides a system of boundary faces of $\cF_{xz}$.
\end{defn}
In particular, a flow category defines a topologically enriched category\footnote{\href{https://www.math.stonybrook.edu/~jpardon/joke.pdf}{https://www.math.stonybrook.edu/$\sim$jpardon/joke.pdf}}.

\begin{defn}
    There is a \emph{shift} operation on flow categories. For a flow category $\cF$, we write $\cF[i]$ for the flow category which is the same as $\cF$, except all the gradings are increased by $i$, i.e. $|\cdot|^{\cF[i]} = |\cdot|^\cF + i$.
\end{defn}
\begin{defn}
    A \emph{morphism} $\cW$ between two flow categories $\cF$ and $\cG$ consists of compact manifolds with faces $\cW_{xy}$ for each $x \in \cF$ and $y \in \cG$ of dimension $|x|-|y|$, along with maps
            $$c = c^\cW: \cF_{xz} \times \cW_{zy} \rightarrow \cW_{xy}$$
            and
            $$c = c^\cW: \cW_{xz} \times \cG_{zy} \rightarrow \cW_{xy}$$
            which are suitably associative and compatible with the $c^\cF$ and $c^\cG$, and whose images together form systems of boundary faces for each $\cW_{xy}$.
\end{defn}
\begin{rmk}
    For brevity, we will often write $\cF_{x_1 \ldots x_i}$ for $\cF_{x_1 x_2} \times \ldots \times \cF_{x_{i-1} x_i}$ and $\cW_{x_1\ldots x_i;y_1 \ldots y_j}$ for $\cF_{x_1 \ldots x_i} \times \cW_{x_i y_1} \times \cG_{y_1 \ldots y_j}$, and similarly for other structures which we will encounter later, such as bordisms or bilinear maps.
\end{rmk}
\begin{defn}
     A \emph{bordism} $\mathcal{R}$ between two such morphisms $\mathcal{W}$ and $\mathcal{V}$ consists of compact smooth manifolds with faces $\mathcal{R}_{xy}$ (of dimension $|x|-|y|$ + 1) for $x$ in $\mathcal{F}$ and $y$ in $\mathcal{G}$, along with maps
            $$c=c^\cR: \cW_{xy}, \cV_{xy} \rightarrow \cR_{xy}$$
            and 
            $$c = c^\cR: \cF_{xz} \times \cR_{zy} \rightarrow \cR_{xy}$$
            and
            $$c = c^\cR: \cR_{xz} \times \cG_{zy} \rightarrow \cR_{xy}$$
            compatible with $c^\cF$, $c^\cG$, $c^\cW$, $c^\cV$, and together forming systems of boundary faces for each $\cR_{xy}$.

            We write $[\cF, \cG]$ for the set of morphisms $\cF \to \cG$ up to the equivalence relation generated by bordism. This is naturally an abelian group under disjoint union.
\end{defn}

\subsection{Composition}\label{sec:comp}

Let $\cF,\cG,\cH$ be flow categories, and $\cW: \cF \to \cG$ and $\cV: \cG \to \cH$ morphisms. We define their composition as follows.\par 
First, choose collar neighbourhoods $\cC: \cF_{xx'} \times \cF_{x'x''} \times [0,\eps) \hookrightarrow \cF_{xx''}$ onto a neighbourhood of each boundary face; do the same for $\cG, \cH, \cW$ and $\cV$. We require that these are suitably compatible, as in \cite[Section 3.2]{PS}.\par 
For each $x \in \cF$ and $z \in \cH$, define $\tilde{\cQ}_{xz}$ to be the glued-together manifold:
\begin{equation}\label{eq: cQ}
    \tilde{\cQ}_{xz} := \operatorname{colim}\left(\bigsqcup_{y,y' \in \cG} \cW_{xy} \times \cG_{yy'} \times \cV_{yz} \times [0, \varepsilon)^2 \rightrightarrows \bigsqcup_{y \in \cG} \cW_{xy} \times \cV_{yz} \times [0, \varepsilon)\right)
\end{equation}
where the maps in the diagram send $(a,b,c,u,v)$ to $(a, \cC(b, c, u), v)$ and $(\cC(a,b,v),c,u)$ respectively.

Since $\tilde \cQ_{xz}$ is obtained from manifolds with corners by gluing along open subsets, we find that (after checking Hausforffness, as in \cite[Lemma 3.21]{PS}) $\tilde \cQ_{xz}$ is a smooth manifold with corners. It is non-compact, but if we compatibly choose framed functions $(f_{xz}, s_{xz})$ on each $\tilde \cQ_{xz}$, setting $\cP_{xz} := f_{xz}^{-1}\{1\}$ defines a morphism of flow categories $\cP: \cF \to \cH$.

This is well-defined up to bordism, and defines the composition operation in a (non-unital) category $(\Flow, [\cdot,\cdot])$ of (unoriented) flow categories. In Section \ref{sec: Thet flow} we define categories $\Flow^{\Theta}$ of flow categories equipped with tangential structures, generalising the example of $\Flow^{fr}$ studied in \cite{PS}.

The following criteria for identifying representatives of composition maps will be useful later.

\begin{lem}[{\cite[Lemma 4.20]{PS}}]
    Suppose $\cW, \cV, \cT$ are morphisms of flow categories as follows:
    \begin{equation}\label{eq: comp recog diag}
        \xymatrix{
            \cF \ar[rr]^\cT \ar[dr]_{\cW} &&
            \cH \ar[dl]^\cV \\
            & \cV &
        }
    \end{equation}
    Suppose there are compact manifolds with faces $\cR_{xz}$ for each $x \in \cF, z \in \cH$ of dimension $|x|-|z|$, along with embeddings of boundary faces 
    $$\cW_{xy} \times \cV_{yz}, \cT_{xz}, \cR_{xx';z}, \cR_{x;z'z} \hookrightarrow \cR_{xz}$$
    for $y\in \cG, x' \in \cF, z' \in \cH$, only overlapping in the evident ways, which together form a system of boundary faces for $\cR_{xz}$.\par 
    Then (\ref{eq: comp recog diag}) commutes in $\Flow$. Equivalently, $\cT$ is a representative of $\cV \circ \cW$.
\end{lem}

\subsection{Index bundles}\label{sec: ind bun}
    In order to make sense of tangential structures on flow categories and related structures, we first need to package the data of their tangent bundles in an appropriate way. We recap some material from \cite[Section 4.2]{PS}; see \textit{loc. cit.} for more detail.
    \begin{defn}
        Let $\cM_{xy}$ be a manifold with faces.

        Its \emph{abstract index bundle} is the vector bundle:
        $$I^{\cM}_{xy} := T\cM_{xy} \oplus \bR \tau_y^{\cM} $$
        over $\cM_{xy}$ where $\tau^\cM_y$ is a formal generator of the line $\bR \tau^\cM_y$. 
    \end{defn}
    \begin{rmk} 
        For the purposes of this subsection, subscripts of the form $\cdot_{xy}$ are purely decorative; we choose to include them here since all manifolds who we consider an index bundle of are indexed by similar subscripts.
    \end{rmk}
    All boundary faces $F \subseteq \cM_{xz}$ we encountered in Section \ref{sec: unor flow cat} are of one of two forms: $F$ is either equipped with a product decomposition $F = \cA_{xy} \times \cB_{yz}$ for some other manifolds $\cA_{xy}$ and $\cB_{yz}$, in which case we label $F$ as \emph{broken}, or it is not, in which case we label $F$ as \emph{unbroken}. Unbroken boundary faces are always disjoint from each other.\par
    
    In the first case, there is an isomorphism of vector bundles $\psi: I^\cA_{xy} \oplus I^\cB_{yz} \to I^\cM_{xz}$ over $F$. This is given by $di$ (where $i: F \to \cM_{xz}$ is the inclusion) on $T\cA_{xy} \oplus T\cB_{yz}$, sends $\tau_z$ to $\tau_z$ and sends $\tau_y$ to the inwards-pointing normal $\nu^F$ along $F$.\par 
    If $F$ is an unbroken face, there is an isomorphism of vector bundles $\psi: I^F \oplus \bR \sigma \to I^\cM_{xz}$ over $F$ (where $\sigma$ is a generator of the abstract line $\bR \sigma$). This is given by $di$ on $TF$, sends $\tau_z$ to $\tau_z$ and sends $\sigma$ to $\pm \nu^F$; the sign depends on whether $F$ is an \emph{incoming} or \emph{outgoing} unbroken boundary face, as in \cite[Section 4.2]{PS}. All boundary faces considered in Section \ref{sec: unor flow cat} are broken, except for the two ends $\cW_{xy}$ and $\cW'_{xy}$ of $\cR_{xy}$ where $\cR$ is a bordism between two morphisms $\cW$ and $\cW'$; then these are defined to be unbroken, with the former incoming and the latter outgoing.\par 
    These isomorphisms are compatible with each other in the sense that the natural diagrams one can construct with them over higher-codimensional faces of $\cM_{xz}$ all commute, cf \cite[Proposition 4.27]{PS}.

\section{Tangential structures on flow categories}\label{sec: Thet flow}
    Framed flow categories provide a model for the category of spectra \cite[Proposition 1.10]{AB2}, well-suited to Floer-theoretic applications. In this section, we give a flow-categorical model for (the homotopy category of) the category of bimodules over more general Thom ring spectra.

    Until Section \ref{sec: flow adapt}, we fix $\Theta$ some tangential structure of any rank $1 \leq N \leq \infty$.

\subsection{$\Theta$-orientations}\label{sec: flow or}

    \begin{defn}\label{def: Theta or}
        Let $\cF$ be a flow category. A \emph{$\Theta$-orientation} on $\cF$ consists of:
        \begin{itemize}
            \item $\Theta$-oriented puncture data $\bE_x = \bE^\cF_x$ of Maslov index $|x|$.
            \item Vector spaces with inner product $V_{xx'}$.
            \item Isometric embeddings
            \begin{equation*}
                \iota=\iota_{xx'x''}: V_{xx'} \oplus V_{x'x''} \hookrightarrow V_{xx''}
            \end{equation*}
            \item Maps of spaces 
            \begin{equation*}
                \rho=\rho_{xx'}: \cF_{xx'} \to \bU_{xx'}
            \end{equation*}
            \item Isomorphisms of vector bundles over $\cF_{xy}$
            \begin{equation*}
                st=st_{xx'}: I^\cF_{xx'} \oplus V_{xx'} \to \rho_{xy}^*\bV_{xx'}
            \end{equation*}
            In future, we omit the subscripts and $\rho^*$ when unambiguous.
        \end{itemize}
        for each $x, x', x'' \in \cF$. We write $\bU_{xx'}$ as shorthand for $\bU^{\Theta,V_{xx'}}_{11}(\bE_x,\bE_{x'})$ (and similarly for $\bV_{xx'}$); note that the maps $\iota$ along with concatenation of abstract discs together induce maps $\bU_{xx'} \times \bU_{x'x''} \to \bU_{xx''}$.
        
        We require that:
        \begin{itemize}
            \item The $\iota_{xx'}$ are suitable associative in the sense that the following diagram of isometric embeddings commutes:
            \begin{equation*}
                \xymatrix{
                    V_{xx'} \oplus V_{x'x''} \oplus V_{x''x'''}
                    \ar[r]
                    \ar[d]
                    &
                    V_{xx'} \oplus V_{x'x'''}
                    \ar[d]
                    \\
                    V_{xx''} \oplus V_{x''x'''}
                    \ar[r]
                    &
                    V_{xx'''}
                }
            \end{equation*} 
            \item The $\rho_{xx'}$ are suitably associative in the sense that the following diagram commutes:
            \begin{equation*}
                \xymatrix{
                    \cF_{xx'} \times \cF_{x'x''} 
                    \ar[rr]_{\rho_{xx'} \times \rho_{x'x''}}
                    \ar[d]
                    &&
                    \bU_{xx'} \times \bU_{x'x''}
                    \ar[d]
                    \\
                    \cF_{xx''}
                    \ar[rr]_{\rho_{xx''}}
                    &&
                    \bU_{xx''}
                }
            \end{equation*}
            \item The $st_{xx'}$ are suitably associative in the sense that the following diagram of isomorphisms of vector bundles over $\cF_{xx'} \times \cF_{x'x''}$ commutes:
            \begin{equation*}
                \xymatrix{
                    I^\cF_{xx'} \oplus V_{xx'} \oplus I^\cF_{x'x''} \oplus V_{x'x''} \oplus V_{xx'x''}^\perp
                    \ar[rr]_{st \oplus st}
                    \ar[d]
                    &&
                    \bV_{xx'} \oplus \bV_{x'x''} \oplus V_{xx'x''}^\perp
                    \ar[d]
                    \\
                    I^\cF_{xx''} \oplus V_{xx''}
                    \ar[rr]_{st}
                    &&
                    \bV_{xx''}
                }
            \end{equation*}
            Here $V_{xx'x''}^\perp$ is the orthogonal complement of $V_{xx'} \oplus V_{x'x''}$ in $V_{xx''}$. 
        \end{itemize}
        We will often use a superscript $\cdot^\cF$ to disambiguate between the same data associated to another $\Theta$-oriented flow category.
    \end{defn}
    
    \begin{defn}\label{def: Theta or mor}
        Let $\cF, \cG$ be $\Theta$-oriented flow categories. A \emph{$\Theta$-orientation} on a morphism $\cW: \cF \to \cG$ consists of, for each $x, x' \in \cF$ and $y, y' \in \cG$:
        \begin{itemize}
            \item Vector spaces with inner product $V_{xy}$.
            \item Isometric embeddings
            \begin{equation}\label{eq: isom emb}
                \xymatrix{
                    \iota=\iota_{xx'y}: V_{xx'} \oplus V_{x'y} \to V_{xy} &\,\,\,\, \iota = \iota_{xy'y}: V_{xy'} \oplus V_{y'y} \to V_{xy}
                }
            \end{equation}
            
            \item Maps of spaces
            \begin{equation}
                \rho=\rho_{xy}: \cW_{xy} \to \bU_{xy}
            \end{equation}
            \item Isomorphisms of vector bundles over $\cW_{xy}$
            \begin{equation}
                st=st_{xy}: I^\cW_{xy} \oplus V_{xy} \to \bR \oplus \bV_{xy}
            \end{equation}
        \end{itemize}
        satisfying similar coherence conditions to Definition \ref{def: Theta or} and  \cite[Definition 4.32(1)]{PS}.

        A \emph{$\Theta$-orientation} on a bordism $\cR$ between two $\Theta$-oriented morphisms $\cW$ and $\cW'$ consists of vector spaces with inner product $V^\cR_{xy}$, isometric embeddings $V^{\cW}_{xy}, V^{\cW'}_{xy} \to V^\cR_{xy}$ as well as ones similar to (\ref{eq: isom emb}), maps of spaces $\rho: \cR_{xy} \to \bU_{xy}$ and isomorphisms of vector bundles $st: I^\cR_{xy} \oplus V^\cR_{xy} \to \bR^2 \oplus \bV_{xy}$ over $\cR_{xy}$. These are required to satisfy appropriate coherence conditions similar to Definition \ref{def: Theta or} and \cite[Definition 4.32(2)]{PS}. A \emph{$\Theta$-orientation} on $\tilde \cQ$ (defined as in Section \ref{sec: unor flow cat}) is defined similarly. 

    \end{defn}

    \begin{example}\label{ex: stab Theta or}
        Let $\cW$ be a $\Theta$-oriented morphism. Given vector spaces $E_{xy}$ with inner product for each $x \in \cF$ and $y \in \cG$, along with suitably associative isometric embeddings $E_{x'y} \to E_{xy}$ and $E_{xy'} \to E_{xy}$ whenever $\cF_{xx'}$ or $\cG_{y'y}$ are nonempty, we may replace each $V_{xy}$ with $V_{xy} \oplus E_{xy}$, compose $\rho_{xy}$ with the maps $\bU_{11}^{V_{xy}} \to \bU_{11}^{V_{xy} \oplus E_{xy}}$ associated to the inclusion $V_{xy} \to V_{xy} \oplus E_{xy}$ and similarly compose for the $st_{xy}$, to obtain a new $\Theta$-oriented flow category $\cW^E$, which we call the \emph{stabilisation} of $\cW$ by $E$; similarly to \cite[Example 4.38]{PS}. $\cW$ and $\cW^E$ are always $\Theta$-oriented bordant.

        An example of this is where all $E_{xy}=E$ are given by some fixed vector space with inner product $E$.

    \end{example}
    \begin{example}
        Let $\cW: \cF \to \cG$ be a $\Theta$-oriented morphism. Let $\cW^\bR$ be the $\Theta$-oriented flow category given by stabilising $\cW$ as in Example \ref{ex: stab Theta or} by the constant vector space $\bR$.
        
        We define $\overline{\cW}$ for the $\Theta$-oriented morphism obtained from $\cW^\bR$ by precomposing with the map $-1: \bR \to \bR$. More precisely, we define $st^{\overline \cW}_{xy}$ to be the composition:
        \begin{equation*}
            st^{\overline \cW}_{xy} : I^{\overline \cW}_{xy} \oplus V^{\overline \cW}_{xy} 
            =
            I^\cW_{xy} \oplus V^\cW_{xy} \oplus \bR_1 
            \xrightarrow{Id \oplus (-1)}
            I^\cW_{xy} \oplus V^\cW_{xy} \oplus \bR_1 
            \xrightarrow{st^\cW_{xy} \oplus Id}
            \bR \oplus \bV^\cW_{xy} \oplus \bR_1 
            =
            \bR \oplus \bV^{\overline \cW}_{xy}
        \end{equation*}
        where $\bR_1$ is the copy of $\bR$ that we are stabilising by.
    \end{example}
    \begin{defn}
        We write $[\cF, \cG]^\Theta$ for the set of bordism classes of $\Theta$-oriented morphisms $\cF \to \cG$. 
    \end{defn}
    $[\cF, \cG]^\Theta$ is an abelian group under disjoint union: the unit given by the empty morphism, and $\overline \cW$ is inverse to $\cW$. A nullbordism of $\cW \sqcup \overline \cW$ is constructed similarly to \cite[Example 4.36]{PS}. 

\subsection{Composition}\label{sec: comp}
    Let $\cW: \cF \to \cG$ and $\cV: \cG \to \cH$ be $\Theta$-oriented flow categories. In this section we construct a $\Theta$-orientation on their composition.\par 
    
    Let $\cC$, $\tilde{\cQ}_{xz}$, $(f_{xz}, s_{xz})$ and $\cP_{xz}$ be as in Section \ref{sec:comp}. We first construct a $\Theta$-orientation on $\tilde Q$, similarly to \cite[Lemma 4.41]{PS}.

    Inductively in $|x|-|z|$, choose vector spaces with inner product $V^{\tilde \cQ}_{xz}$, along with embeddings
    \begin{equation}
        \xymatrix{
            \iota: V_{xy}^\cW \oplus V_{yz}^\cV \to V_{xz}^{\tilde \cQ}
            &
            \iota: V_{xx'}^\cF \oplus V_{x'z}^{\tilde \cQ} \to V_{xz}^{\tilde \cQ}
            &
            \iota:  V_{xz'}^{\tilde \cQ} \oplus V_{z'z}^\cH \to V_{xz}^{\tilde \cQ}
        }
    \end{equation}
    which are suitably associative. This choice is unique up to stabilisation and isomorphism.

    Over each $\cW_{xy} \times \cV_{yz} \times [0,\eps)_s \subseteq \tilde \cQ_{xz}$, we define $\rho^{\tilde \cQ}_{xz}$ to be the composition:
    \begin{equation}
        \rho_{xz}^{\tilde \cQ}: \cW_{xy} \times \cV_{yz} \times [0, \eps)_s \xrightarrow{\rho_{xy}^\cW \times \rho_{yz}^\cV} 
        \bU_{xy} \times \bU_{yz} \to 
        \bU_{xz}
    \end{equation}
    These agree on overlaps of the form $\cW_{xy} \times \cG_{yy'} \times \cV_{y'z} \times [0, \eps)^2_{u,v}$ by associativity of the $\rho^\cW, \rho^\cG, \rho^\cV$, and so glue together to define $\rho_{xz}^{\tilde \cQ}$ over all of $\tilde \cQ_{xz}$.

    Similarly we define isomorphisms of vector bundles $st^{\tilde \cQ}_{xz}$ over each $\cW_{xy} \times \cV_{yz} \times [0, \eps)^2 \subseteq \tilde \cQ_{xz}$ to be the composition:
    \begin{equation}
        st^{\tilde \cQ}_{xz}: I_{xz}^{\tilde \cQ} \oplus V_{xz}^{\tilde \cQ} 
        \xrightarrow{\phi^{-1}}
        I_{xy}^\cW \oplus I_{yz}^\cV \oplus V_{xz}^{\tilde \cQ}
        \cong 
        I_{xy}^{\cW} \oplus I_{yz}^\cV \oplus V_{xy}^\cW \oplus V_{yz}^\cV \oplus V_{xyz}^\perp 
        \xrightarrow{st \oplus st}
        \bV_{xy} \oplus \bV_{yz} \oplus V_{xyz}^\perp
        \to\bV_{xz}
    \end{equation}
    where $V_{xyz}^\perp$ is the orthogonal complement to the embedding $V^\cW_{xy} \oplus V^\cV_{yz} \to V^{\tilde \cQ}_{xz}$. These glue on overlaps by the same argument as in \cite[Lemma 4.41]{PS} to define isomorphisms of vector bundles $st^{\tilde\cQ}_{xz}$ over all of $\tilde \cQ_{xz}$, which again are suitably associative.\par 

    Using the isomorphism $I_{xz}^\cP \oplus \bR \cong I_{xz}^{\tilde \cQ}$ determined by the framed function then gives a $\Theta$-orientation on $\cP$, similarly to \cite[Lemma 4.42]{PS}. This $\Theta$-orientation is canonical up to stabilisation and isomorphism; as in \cite[Lemma 4.42]{PS} this implies the $\Theta$-oriented bordism class of $\cP$ is independent of choices.

    Composition thus defines well-defined maps 
    \begin{equation}
        [\cF, \cG]^\Theta \otimes [\cG, \cH]^\Theta \to [\cF, \cH]^\Theta
    \end{equation}
    and we arrive at a (potentially non-unital) category $\operatorname{Flow}^\Theta$, whose objects are $\Theta$-oriented flow categories and whose morphisms are given by $[\cdot, \cdot]^\Theta$.

    \begin{rmk}
        When $\Theta$ is contractible, this category is equivalent to the category of framed flow categories considered in \cite{PS}.
    \end{rmk}

Let $\cF$ be a $\Theta$-oriented flow category with puncture data $\{\bE_x\}_{x\in\cF}$ , and let $\{\bE'_x\}_{x\in \cF}$  be another collection of puncture data with $\mu(\bE'_x) = |x|$ for every $x$.  Then we may define a (canonical up to homotopy) $\Theta$-oriented flow category $\cF'$, with underlying unoriented flow category $\cF$ but puncture data $\{\bE'_x\}_{x\in \cF}$ as follows. 
We choose $\bD_{1,1}^x \in \bU_{1,1}^{\Theta}(\bE_x,\bE'_x)$ for each $x\in\cF$. Gluing of abstract discs together with the given $\Theta$-orientation maps $\cF_{xy} \to \bU(\bE_x,\bE_y)$ then yields maps $\cF_{xy} \to \bU(\bE'_x,\bE'_y)$;  these maps may be homotoped (inductively in $|x| - |y|$) so that the required compatibilities of Definition \ref{def: Theta or} hold. Performing a similar construction to morphisms (and bordisms) induces isomorphisms
\begin{equation*}
    [\cF,\cG]^{\Theta} \simeq [\cF',\cG]^{\Theta}
\end{equation*}
and similarly for morphisms to $\cF$ respectively $\cF'$, so that: 
\begin{lem}\label{lem:choice of puncture data}
    $\cF$ and $\cF'$ are equivalent in $\mathrm{Flow}^{\Theta}$.
\end{lem}

    \begin{rmk}
        Let $\{\Theta(N)\}_N$ be a sequence of tangential structures as in Section \ref{sec: Nstab}, and let $\cF$ be a $\Theta(N)$-oriented flow category.

        Composing the maps $\rho^\cF$ with the stabilisation maps $\Xi$ from Proposition \ref{prop: Nstab}, we obtain a $\Theta(N+1)$-oriented flow category $\Xi(\cF)$. The same construction may be applied to morphisms and bordisms, and so we obtain functors
        \begin{equation*}
            \Xi: \operatorname{Flow}^{\Theta(N)} \to \operatorname{Flow}^{\Theta(N+1)}
        \end{equation*}
        These are compatible with the additional structures of flow modules and bilinear maps, cf. Sections \ref{sec: flow mod} and \ref{sec: bilin}.
    \end{rmk}
    Let $R$ be the spectrum defined as in Section \ref{sec: Thom}. Recalling Remark \ref{rmk:reverse bundle}:

    \begin{prop}\label{prop: bord pt bimod}
        If $\cF = *[i]$ is a $\Theta$-oriented flow category with one object $\ast$ in degree $i$ and $\cG = +[j]$ is a $\Theta$-oriented flow category with one object $+$ in degree $j$, then $[\cF,\cG]^{\Theta} = \pi_{i-j}(R \otimes R^{op}) \cong \Omega_{i-j}^{\Theta \times \Theta}$ is isomorphic to the bordism groups of manifolds with stable tangent bundle classified by the index bundle over $\Omega \Theta \times \Omega \Theta$.

    \end{prop}
\begin{proof}
    The data of a morphism $\cW$ is a single closed $(i-j)$-manifold $\cW_{*+}$, whose tangent bundle is stably classified by a map to $\bU_{11}^\Theta \simeq \Omega \Theta \times \Omega \Theta$. This associates to a morphism a class in $\Omega_{i-j}^{\Theta \times \Theta}$, which depends only on the bordism class of $\cW$. That this association is onto and yields a homomorphism follows from a standard Pontrjagin-Thom argument.
\end{proof}
    This computation motivates:
    \begin{conj}\label{conj: PT/CJS bimod}
        Let $R$ be the ring spectrum described in Section \ref{sec: Thom}. Then $\operatorname{Flow}^\Theta$ is equivalent to (the homotopy category of) the category of perfect $R-R$ bimodules, $\operatorname{Perf}(_R\operatorname{mod}_R)$.
    \end{conj}
    \begin{rmk}
        To get a similar model for the category of left (or right) $R$-modules, one can impose an additional condition, that all maps $\cM \to \bU_{11}$ land in the subspace $\bU_{11}^L$ (or $\bU_{11}^R$) of abstract discs where the totally real subspace and lift to $\Theta$ are constant along the left (or right) boundary component of the abstract disc.
    \end{rmk}
\subsection{Flow modules}\label{sec: flow mod}

    \begin{defn}
        A \emph{$\Theta$-oriented right flow module} $\cR$ \emph{of degree $i$}
        over a $\Theta$-oriented flow category $\cF$ consists of 
        \begin{itemize}
            \item A morphism of unoriented flow categories $*[i] \to \cF$.
            \item Vector spaces with inner product $V_{*x}$.
            \item Isometric embeddings $\iota=\iota_{*xx'}: V_{*x} \oplus V_{xx'} \to V_{*x'}$.
            \item Maps of spaces $\rho=\rho_{*x}: \cR_{*x} \to \bU_x$, where we set $\bU_x := \bU_{01}^{\Theta,V_{*x}}(\bE_x)$.
            \item Isomorphisms of vector bundles over $\cR_{*x}$
            \begin{equation}
                st=st_{*x}: I_{*x}^{\cR} \oplus V_{*x} \to \bR \oplus \bV_x \oplus \bR^i
            \end{equation}
        \end{itemize}
        for each $x, x' \in \cF$. These are required to satisfy similar coherence conditions to Definition \ref{def: Theta or mor}.

        We define a \emph{$\Theta$-oriented bordism} between two right flow modules $\cR$ and $\cR'$ of degree $i$ to be a bordism $\cB$ similarly to Definition \ref{def: Theta or mor}: this consists of an unoriented bordism between $\cR$ and $\cR'$, along with vector spaces and isometric embeddings as in Definition \ref{def: Theta or mor}, and maps of spaces $\cB_{*x}: \cB_{*x} \to \bU_x$ and isomorphisms of vector bundles $I^\cB_{*x} \oplus V_{*x} \to \bR^2 \oplus \bV_{*x} \oplus \bR^i$, satisfying similar coherence conditions to Definition \ref{def: Theta or mor}.

        A $\Theta$-oriented \emph{left} flow module of degree $i$ over $\cF$, $\cL$, consists of a morphism of unoriented flow categories $\cF \to *[-i]$, equipped with data similar to the right-hand case, except with index bundles $I^\cL_{x*}$ classified by maps $\cL_{x*} \to \bU_{10}^{\bV_{x*}}(\bE_x)$. Bordisms of $\Theta$-oriented left modules are defined similarly.

        We write $\Omega_i^\Theta(\cF)$ and $\Omega^i_\Theta(\cF)$ for the set of bordism classes of $\Theta$-oriented right and left modules over $\cF$ of degree $i$, respectively. These form abelian groups under disjoint union.
    \end{defn}

    Composition as in Section \ref{sec: comp} gives well-defined maps:
    \begin{equation*}
        \xymatrix{
            \Omega_*^\Theta(\cF) \otimes [\cF, \cG]^\Theta \to \Omega_*^\Theta(\cG)
            &&
            [\cF, \cG]^\Theta \otimes \Omega^*_\Theta(\cG) \to \Omega^*_\Theta(\cF)
        }
    \end{equation*}
    These are compatible with composition in $\operatorname{Flow}^\Theta$, and in particular define functors $\Omega_*^\Theta(\cdot)$ and $\Omega^*_\Theta(\cdot)$. 
    
    Let $R$ be defined as in Section \ref{sec: Thom}. The same proof as Proposition \ref{prop: bord pt bimod} shows:
    \begin{prop}\label{prop: bord pt}
        If $\cF = *[j]$ is a $\Theta$-oriented flow category with one object $*$ in degree $j$, then there are isomorphisms:
        \begin{equation*}
            \xymatrix{
                \Omega_i^\Theta(\cF) \cong \pi_{i-j}(R) 
                &&
                \Omega^i_\Theta(\cF) \cong \pi_{i+j}(R)
            }
        \end{equation*}
    \end{prop}
    \begin{rmk}\label{rmk:not a morphism group}
        In the case of framed flow categories, there is a natural bijection between bordism classes of framed morphisms $[*[i], \cF]$ and right flow modules $\Omega^{fr}_i(\cF)$. However, as we see from Propositions \ref{prop: bord pt bimod} and \ref{prop: bord pt}, this does not hold for other tangential structures.
    \end{rmk}
    The following definition will be important when constructing the open-closed map in Section \ref{sec: unit OC}.
    \begin{defn}
        A \emph{$\Theta$-oriented capsule} $\cM$ of degree $i$ consists of a closed manifold $\cM_{**}$, a vector space with inner product $V_{**}$, along with a map $\rho: \cM_{**} \to \bU_{**}:=\bU^{\Theta,V_{**}}_{00}$, and an isomorphism of vector bundles $st: I^\cM_{**} \to \bR \oplus \bV_{**} \oplus \bR^i$ covering $\rho$. 

        We write $\Omega_i^{\Theta, \circ}$ for the set of equivalence classes $\Theta$-oriented capsules under the natural notion of bordism of such. 
        
    \end{defn}
    The same proof as Proposition \ref{prop: bord pt} shows:
    \begin{prop}\label{prop: capsule comp}
        Let $\tilde R$ be the Thom spectrum of the index bundle over the free loop space $\cL \Theta \simeq \bU_{0,0}(\Theta)$. Then there is an isomorphism $\Omega^{\Theta, \circ}_i \cong \pi_i \tilde R$.
    \end{prop}
    Assuming Assumption \ref{asmp}, \cite{BCS} implies that this is exactly $\pi_* THH(R)$.

\begin{rmk}\label{rmk:reverse bundle 2}
    For any ring spectrum $R$, there is a spectral sequence $HH_*(\pi_*(R)) \Rightarrow \pi_*(THH(R))$. The graded ring $\pi_*(R)$ is commutative, and $HH_*(\pi_*(R^{op})) \cong HH_*(\pi_*(R))$.  Following on from Remark \ref{rmk:reverse bundle}, there is a natural isomorphism $THH(R^{op}) \to THH(R)^{op}$ arising from reversing the order of all terms in the cyclic bar complex. Thus, the homotopy groups of $THH(R)$ and $THH(R^{op})$ agree, and one can define the capsule bordism groups starting either from the Thom spectrum $\tilde{R}$ of the index bundle or from the Thom spectrum of minus the index bundle. 

    More geometrically, if $\Ind_{\scrL}$ is the index bundle over $\scrL U/O$ and we write $\scrL U/O = U/O \times \Omega U/O$, then $\Ind_{\scrL}(x,\gamma) \sim \Ind_{\Omega}(\gamma) \oplus Re(x)$, and then again $\Ind_{\scrL} + \iota^*\Ind_{\scrL}$ is stably trivial, using Remark \ref{rmk:reverse bundle} and the fact that $Re: U/O \to BO$ is 2-torsion.
\end{rmk}
    
    For any $\Theta$-oriented flow category $\cF$, composition gives a map
    \begin{equation}\label{eq: capsule concat}
        \Omega_i^\Theta(\cF) \otimes \Omega_\Theta^j(\cF) \to \Omega_{i+j}^{\Theta,\circ}
    \end{equation}

    Motivated by the computations in Propositions \ref{prop: bord pt} and \ref{prop: capsule comp} which essentially verify this for flow categories with one object, we conjecture:
    \begin{conj}\label{conj: flow mod PT}
        Under the equivalence of categories $\operatorname{Flow}^\Theta \simeq \operatorname{Perf}(_R\operatorname{mod}_R)$ from Conjecture \ref{conj: PT/CJS bimod}, if $\cF \in \operatorname{Flow}^\Theta$ corresponds to a bimodule $M$, then there are functorial isomorphisms:
        \begin{equation*}
            \xymatrix{
                \Omega_*^\Theta(\cF) \cong \pi_*(M \otimes_{R\otimes R^{op}} R_\Delta)
                &&
                \Omega^*_\Theta(\cF) \cong \operatorname{Hom}_{R\otimes R^{op}}(M, R_\Delta)
            }
        \end{equation*}
        where $R_\Delta$ is the diagonal bimodule and $\operatorname{Hom}_{R \otimes R^{op}}$ is the set of homotopy classes of maps of bimodules.

        Under these isomorphisms along with Proposition \ref{prop: capsule comp}, concatenation as in (\ref{eq: capsule concat}) corresponds to the composition
        \begin{multline*}
            \pi_*(M \otimes_{R \otimes R^{op}}R_\Delta) \otimes Hom_{R \otimes R^{op}}(M, R_\Delta)
            \xrightarrow{Id \otimes (\cdot \otimes_{R \otimes R^{op}}R_\Delta)}
            \pi_*(M \otimes_{R \otimes R^{op}} R_\Delta)
            \otimes \operatorname{Hom}_\bS(M \otimes_{R \otimes R^{op}}R_\Delta, R_\Delta \otimes_{R \otimes R^{op}} R_\Delta)
            \\
            \xrightarrow{\textrm{compose}}
            \pi_*(R_\Delta \otimes_{R \otimes R^{op}} R_\Delta) 
            =: \pi_* THH(R)
        \end{multline*}
    \end{conj}

\subsection{Bilinear maps}\label{sec: bilin}
    
    Let $\cF$, $\cG$ and $\cH$ be $\Theta$-oriented flow categories. We recall a definition from \cite[Section 5.1]{PS}:
    \begin{defn}
        An (unoriented) bilinear map $\cB: \cF \times \cG \to \cH$ consists of manifolds $\cB_{xy;z}$ for each $x \in \cF, y \in \cG$ and $z \in \cH$, along with maps 
        \begin{equation}
            \xymatrix{
                \cB_{xx',y;z} \to \cB_{xy;z} 
                &
                \cB_{x,yy';z} \to \cB_{xy;z}
                &
                \cB_{x,y;z'z} \to \cB_{xy;z}
            }
        \end{equation}
        where we write $\cB_{xx',y;z}$ for $\cF_{xx'} \times \cB_{x'y;z}$, and similar for $\cB_{x,yy';z}$ and $\cB_{x,y;z'z}$. We require these maps to satisfy a suitable associativity relation, and that the collection of these maps together form a system of boundary faces for each $\cB_{xy;z}$. All of these boundary faces are defined to be broken, in the sense of Section \ref{sec: ind bun}.
    \end{defn}
    
    We incorporate $\Theta$-orientations in much the same way as before, with each $I^\cB_{xy;z}$ classified by a map to a space of $\Theta$-oriented abstract discs:
    \begin{defn}
        A \emph{$\Theta$-orientation} on a bilinear map $\cB: \cF \times \cG \to \cH$ consists of 
        \begin{itemize}
            \item Vector spaces with inner product $V_{xy;z}$.
            \item Isometric embeddings 
            \begin{equation}
                \xymatrix{
                    V_{xx'} \oplus V_{x'y;z} \to V_{xy;z}
                    &
                    V_{yy'} \oplus V_{xy';z} \to V_{xy;z}
                    &
                    V_{xy;z'} \oplus V_{z'z} \to V_{xyz}
                }
            \end{equation}
            \item Maps of spaces $\rho=\rho_{xy;z}: \cB_{xyz} \to \bU_{xy;z} := \bU_{21}^{V_{xy;z}}(\bE_x, \bE_y, \bE_z)$
            \item Isomorphisms of vector bundles over $\cB_{xy;z}$
            \begin{equation}
                st=st_{xy;z}: I^\cB_{xy;z} \oplus V_{xy;z} \to \bR \oplus \bV_{xy;z}
            \end{equation}
        \end{itemize}
        These are required to satisfy similar compatibility conditions to Definition \ref{def: Theta or mor}.
    \end{defn}
    Let $\cB: \cF \times \cG \to \cH$ be a $\Theta$-oriented bilinear map. Composition defines a bilinear map of graded abelian groups:
    \begin{equation}\label{eqn:composition in general}
        \cB_*: \Omega^\Theta_*(\cF) \otimes \Omega^\Theta_*(\cG) \to \Omega_*^\Theta(\cH)
    \end{equation}
    \begin{rmk}
        The category of $R-R$ bimodules is monoidal under $\otimes_R$. Under Conjecture \ref{conj: PT/CJS bimod}, we expect that the two notions of bilinar map coincide.
    \end{rmk}
    We next give a condition for associativity of bilinear maps. Let 
    \begin{equation}\label{eq: 6 flow cats}
        \cF,\cG,\cH,\cI,\cJ,\cK
    \end{equation}
    be $\Theta$-oriented flow categories, and 
    
    \begin{equation}\label{eq: 4 bilin}
        \cW: \cF \times \cG \to \cI,\,\, \cV: \cG \times \cH \to \cJ,\,\, \cX: \cI \times \cH \to \cK,\,\, \cY: \cF \times \cJ \to \cK
    \end{equation} 
    $\Theta$-oriented bilinear maps. We obtain two trilinear maps of graded groups:
    \begin{equation}\label{eq: 2 trilin}
        \cX_* \circ \cW_*, \cY_* \circ \cV_*: \Omega^\Theta_*(\cF) \otimes \Omega^\Theta_*(\cG) \otimes \Omega^\Theta_*(\cH) \to \Omega^\Theta_*(\cK)
    \end{equation}
    \begin{defn}
        An \emph{associator} $\cZ$ for the data (\ref{eq: 6 flow cats}) and (\ref{eq: 4 bilin}) consists of compact manifolds $\cZ_{xyz;w}$ with faces for $x \in \cF, y \in \cG, z \in \cH$ and $w \in \cK$, of dimension $|x|+|y|+|z|-|w|+1$. These are required to have a system of boundary faces given by 
        
        \begin{equation*}
            \cF_{xx'} \times \cZ_{x'yz; w}, \cG_{yy'} \times \cZ_{xy'z;w}, \cH_{zz'} \times \cZ_{xyz'; w}, \cW_{xy;u} \times \cX_{uz; w}, \cV_{yz; v} \times \cY_{xv; w}, \cZ_{xyz;w'} \times \cK_{w'w}
        \end{equation*}
            
        for $x,x' \in \cF, y,y' \in \cG, z,z' \in \cH, u \in \cI, v \in \cJ, w,w' \in \cK$. The inclusions of these faces are required to be suitably associative.

        A \emph{$\Theta$-orientation} on $\cZ$ consists of vector spaces with inner product $V_{xyz;w}$, maps of spaces $\rho=\rho_{xyz;w}: \cZ_{xyz;w} \to \bU_{xyz;w} := \bU_{31}^{V_{xyz;w}}(\bE_x,\bE_y,\bE_z;\bE_w)$ and isomorphisms of vector bundles $st=st_{xyz;w}: I^\cZ_{xyz;w} \oplus V_{xyz;w} \to \bR^2 \oplus \bV_{xyz;w}$ covering $\rho_{xyz;w}$. These are required to be compatible with each other in the natural way. 
    \end{defn}
    Similarly to \cite[Lemma 5.19]{PS} (and by the same proof, incorporating $\Theta$-orientations in the usual way), we have:
    \begin{lem}
        If there exists a $\Theta$-oriented associator $\cZ$ for the data (\ref{eq: 6 flow cats}) and (\ref{eq: 4 bilin}), then the two trilinear maps (\ref{eq: 2 trilin}) agree.
    \end{lem}
        
\subsection{Truncation}

    In this section we incorporate $\Theta$-orientations into the theory of \emph{truncated} flow categories, introduced in \cite[Section 5.4]{PS}.

    \begin{defn}
        Let $i \geq 0$, and let 
        \[
        \mathpzc{Object} \in \left\{ \textrm{flow category, (left/right) module, capsule, morphism, bordism, bilinear map, associator} \right\}. 
        \]
        A \emph{pre-$\tau_{\leq i}$ $\mathpzc{Object}$} is defined the same as a $\mathpzc{Object}$, except we only require the manifolds in the definition to exist when their dimension would be $\leq i$ (and when they do exist, they are equipped with the same extra data and satisfy the same conditions).

        For $i \leq j$, to any pre-$\tau_{\leq j}\mathpzc{Object}$, $\cA$, there is a corresponding pre-$\tau_{\leq i}$ $\mathpzc{Object}$, $\tau_{\leq i} \cA$, obtained by discarding all manifolds of dimension $> i$.

        An \emph{extension} of a pre-$\tau_{\leq i}$ $\mathpzc{Object}$ $\cA$, is a pre-$\tau_{\leq i+1}$ $\mathpzc{Object}$ $\cA'$, such that $\tau_{\leq i} \cA' = \cA$.

        A \emph{$\tau_{\leq i}$ $\mathpzc{Object}$} is a pre-$\tau_{\leq i}$ $\mathpzc{Object}$, $\cA$, such that there exists an extension $\cA'$ of $\cA$. Note that $\cA'$ is not part of the data.

        A \emph{$\Theta$-orientation} on a pre-$\tau_{\leq i}$ $\mathpzc{Object}$ is the same data as a $\Theta$-orientation on a $\mathpzc{Object}$, except only equipped to the manifolds part of the data of the pre-$\tau_{\leq i}$ $\mathpzc{Object}$. For $i \leq j$ and $\cA$ a $\Theta$-oriented pre-$\tau_{\leq j}$ $\mathpzc{Object}$, we may define $\tau_{\leq i} \cA$ similarly to before. A \emph{$\Theta$-orientation} on a $\tau_{\leq i}$ $\mathpzc{Object}$ is a $\Theta$-oriented pre-$\tau_{\leq i}$ $\mathpzc{Object}$ that admits a $\Theta$-oriented extension.

        We add a superscript $\cdot^{\leq i}$ or subscript $\cdot_{\leq i}$ to indicate working with $\tau_{\leq i}$ $\mathpzc{Object}$s; for example we write $[\cF, \cG]^{\Theta, \leq i}$ for the set of bordism classes of $\Theta$-oriented $\tau_{\leq i}$ morphisms $\cF \to \cG$.
    \end{defn}
    Composition is compatible with truncation. In particular, we arrive at a (potentially non-unital) category $\operatorname{Flow}^{\Theta, \leq i}$ of $\Theta$-oriented $\tau_{\leq i}$ flow categories.
\subsection{Morse complex}
    
    Let $\cM$ be a $\Theta$-oriented flow category, where we now assume in addition that $\Theta$ is an oriented tangential structure.

    For each $x \in \cM$, we choose a vector space with inner product $V_x$, along with isometric embeddings $V_x \oplus V_{xx'}^\cM \to V_{x'}$ for each $x,x' \in \cM$, satisfying a suitable associativity condition. We can make this choice, by first choosing all $V_x$ to have sufficiently large dimension, and choosing the embeddings inductively in $|x|-|x'|$. 

    Let $\bU_x = \bU_{01}^{\Theta, V_x}(\bE_x)$ and let $\bV_x$ be the index bundle over $\bU_x$. Composing the maps $\cM_{xx'} \to \bU_{xx'}$ with gluing of abstract discs gives maps $c=c_{xx'}:\bU_x \times \cM_{xx'} \to \bU_{x'}$, along with isomorphisms of vector bundles $st=st_{xx'}: \bV_x \oplus I^\cM_{xx'} \to \bV_{x'}$ covering $c_{xx'}$, satisfying a suitable associativity condition.

    Recall that $\bU_x \simeq \Omega\Theta$ has the homotopy type of the based loop space. Since $\Theta$ maps trivially to $K(\bZ/2,2)$, the map on the loop space classifying the first Stiefel-Whitney class vanishes, so $\bU_x^{\Theta}$ is orientable \cite{Georgieva}, and the determinant line $\det(\bV_x)$ is trivial and so admits a set $F(x)$ of two trivialisations (using the fact that $\bU_x$ is connected). We define the Morse complex
    \[
        CM_i(\cM) = \left(\bigoplus\limits_{|x|=i, \phi \in F(x)} \bZ \phi \right) / A_i
    \]
    with $A_i$ generated by elements of the form $\sum_{\phi \in F(x)} \phi$.  
    
    The isomorphisms $st$ induce isomorphisms of determinant lines
    \begin{equation} \label{eqn:determinant}
        \det(st): \det(\bV_x) \otimes \det(I^\cM_{xx'}) \to \det(\bV_x').
    \end{equation}
    
    We define a linear map $\partial$:
    \[
        F(x) \ni \phi_x \mapsto \partial \phi_x := \sum_{|y| = |x|-1} \Gamma(\cM_{xy}) \psi_y 
    \]
    where $\Gamma(\cdot)$ is a signed count of points, relative to the orientation of $\cM_{xy}$ defined by $\phi_x$, $\psi_y$ and \eqref{eqn:determinant}. The linear map $\partial$ descends to a map $CM_i(\cM) \to CM_{i-1}(\cM)$. As in \cite[Lemma 5.20]{PS} one checks this is a differential.

    \begin{defn}
        Suppose $\cM$ is $\Theta$-oriented, for some oriented tangential structure $\Theta$. 
        The \emph{Morse homology} $HM_*(\cM;\bZ)$ is the homology of the chain complex $(CM_*(\cM);\partial)$.

        Given an abelian group $G$, we define $CM_*(\cM; G) := CM_*(\cM; \bZ) \otimes_\bZ G$.
    \end{defn}
    Suppose now $\cM_i$ are $\Theta$-oriented, for $i=0,1,2$. Again as in \cite{PS}, one checks that

    \begin{itemize}
        \item a $\Theta$-oriented morphism from $\cM_0$ to $\cM_1$ defines a chain map $CM_*(\cM_0) \to CM_*(\cM_1)$;
        \item a $\Theta$-oriented bilinear map $\cR: \cM_0 \times \cM_1 \to \cM_2$ induces a degree zero product $CM_*(\cM_0) \otimes CM_*(\cM_1) \to CM_*(\cM_2)$ which descends to homology;
        \item if we are given $\Theta$-oriented flow categories and bilinear maps as in (\ref{eq: 6 flow cats}) and (\ref{eq: 4 bilin}) and an associator $\cZ$ for this data, then the two induced trilinear maps of groups $HM_*(\cF) \otimes HM_*(\cG) \otimes HM_*(\cH) \to HM_*(\cK)$ agree.
    \end{itemize}
    These are all compatible with composition.

    \begin{lem}\label{lem: morse 0-trun}
        Let $\cF,\cG$ be $\Theta$-oriented flow categories. Then there is a natural isomorphism
        \[
        [\cF,\cG]^{\Theta,\leq 0} \to \Hom(CM_*(\cF),CM_*(\cG))
        \]   
    \end{lem}
    
    \begin{proof}
        $\Theta$ is 1-connected so $\Omega \Theta$ is connected, and hence so is $\bU_{i,1}^V(\bE)$ for any $i$, $\bE$, and any $V$ of sufficiently high dimension.
        
        Therefore for a 0-dimensional manifold $\cM$, homotopy classes of $\Theta$-orientations on it (classified by a map $\cM \to \bU_{i,1}^V(\bE)$ for some $i$,$V$, $\bE$) biject with homotopy classes of framings/orientations on $\cM$. Furthermore $\cM$ bounds a $\Theta$-oriented manifold if and only if the $\cM$, equipped with the corresponding framing/orientation, bounds a framed/oriented manifold.

        Then the result follows from \cite[Lemma 5.41]{PS}.
    \end{proof}
    The same argument shows:
    \begin{lem}
        Let $\cG$ be a $\Theta$-oriented flow category. Then there is a natural isomorphism:
        \begin{equation*}
            \Omega_*^{\Theta, \leq 0}(\cF) \cong HM_*(\cF)
        \end{equation*}
    \end{lem}

\subsection{Obstructions to lifting}\label{sec: ob lift}
    Assume that $\Theta$ is an oriented tangential structure. Let $\cR$ be a $\Theta$-oriented $\tau_{\leq n}$ right module over a $\Theta$-oriented flow category $\cG$. In this subsection we formulate a complete obstruction to $\cR$ being the truncation of a $\Theta$-oriented $\tau_{\leq n+1}$ right module. This closely follows \cite[Theorem 5.48]{PS}, with $\Theta$-orientations incorporated.\par 

    \begin{asmp}\label{asmp: punc dat}
        For simplicity, we choose framed puncture data $\bE_j$ of Maslov index $j$ for each $j \in \bZ$, and assume (without loss of generality, by Lemma \ref{lem:choice of puncture data}) that the puncture data for each $y \in \cG$ is given by $\bE_y = \bE_{|y|}$. Let $\cE_j$ be the $\Theta$-oriented flow category with one object which is equipped with the puncture data $\bE_j$.
    \end{asmp}  
    
    The obstruction $[\cO]=[\cO(\cR)]$ lies in the following Morse homology group of $\cG$ with coefficients:
    $$[\cO] \in HM_{i-n-2}\left(\cG; \Omega^\Theta_{i}(\cE_{i-n-1})\right)$$
    Recall the group $\Omega^\Theta_{i}(\cE_{i-n-1})$ was computed in Proposition \ref{prop: bord pt} and, in particular, only depends on $n+1$, up to isomorphism.

    We begin by building a chain-level representative for this class $[\cO]$. This is of the form
    $$\cO = \sum\limits_{\substack{y \in \cG\\
    i-|y|=n+2}} 
    [\cZ(y)] \cdot y$$
    where the $\cZ(y)$ are $\Theta$-oriented right modules over $\cE_{i-n-1}$ of degree $i$, which we now construct.

    Choose an extension $\cR'$ of $\cR$ (this exists by assumption). For $y \in \cG$ with $i-|y|=n+2$, define $\tilde \cY_{*y}$ to be the following coequaliser:
    $$\tilde \cY_{*y} := \mathrm{coeq} \left( \bigsqcup\limits_{y',y'' \in \cG} \cR'_{*y''} \times \cG_{y''y'} \times \cG_{y'y} \times [0, \eps)^2_{u,v} \rightrightarrows 
    \bigsqcup\limits_{y' \in \cG} \cR'_{*y'} \times \cG_{y'y} \times [0, \eps)_s\right)$$
    where the two maps send $(a,b,c,u,v)$ to $(a, \cC(b, c, u), v)$ and $(\cC(a, b, v), c, u)$ respectively.\par 
    Each $\tilde \cY_{*y}$ is a smooth manifold with faces, and furthermore all its faces are compact. Choose a framed function $(f_{*y}, s_{*y})$ on each $\tilde \cY_{*y}$, and let $\cZ_{*y} = f_{*y}^{-1}\{1\}$. This is a closed manifold of dimension $n+1$, and we consider it as a(n unoriented) right flow module over $\cE_{i-n-1}$, written as $\cZ(y)$, independent of the choice of framed function up to bordism. We next equip $\cZ(y)$ with a $\Theta$-orientation, similar to Section \ref{sec: comp}.\par 

    We first choose vector spaces with inner product $U_y$ for each $y \in \cG$, along with embeddings $V^{\cR'}_{*y'}\oplus V^\cG_{y'y} \to U_y$, agreeing on each $V^{\cR'}_{*y''} \oplus V^\cG_{y''y'y}$.

    Next we define maps $\rho=\rho_{*y}: \tilde \cY_{*y} \to \bU_y := \bU^{U_y}_{01}(\bE_y)$. Over each $\cR'_{*y'} \times \cG_{y'y} \times \{0\}$, we define $\rho_{*y}$ to be the composition of $\rho_{*y'} \times \rho_{y'y}$ with gluing of abstract discs and extension along the appropriate embedding of vector spaces. These glue on overlaps to define $\rho$ over the carapace of each $\tilde \cY_{*y}$; since the carapace is a deformation retract of the whole space, we may choose some extension to the rest of $\tilde \cY_{*y}$ (and this is unique up to homotopy).

    We next construct isomorphisms of vector bundles $st=st_{*y}: I^{\tilde \cY}_{*y} \to \bR \oplus \bV_y \oplus \bR^i$ covering each $\rho$. Over each $\cR'_{*y'} \times \cG_{y'y} \times \{0\}$, it is given by the composition
    \begin{equation}
        I^{\tilde \cY}_{*y} \xleftarrow{\psi} 
        I^{\cR'}_{*y'} \oplus I^\cG_{y'y}
        \xrightarrow{st \oplus st}
        (\bR \oplus \bV_{y'} \oplus \bR^i) \oplus \bV_{y'yy} \to \bR \oplus \bV_y \oplus \bR^i  
    \end{equation}
    where the last arrow is given by gluing of index bundles. Again, these glue together to define $st$ over the carapace; we choose an extension to the rest of $\tilde \cY_{*y}$.

    Using the isomorphism $I^{\tilde \cZ(y)}_{*y} \oplus \bR \cong I^{\tilde \cY}_{*y}$ coming from our choice of framed function, we obtain a $\Theta$-orientation on each $\cZ(y)$.

    \begin{lem}\label{lem: ob closed}
        The chain $\cO$ is closed, and therefore represents a class in $HM_{i-n-2}\left(\cG; \Omega^\Theta_{i}(\cE_{i-n-1})\right)$.
    \end{lem}
    \begin{proof}
        The proof is the same as that of \cite[Theorem 5.48]{PS}, with $\Theta$-orientations incorporated in the same way as in the construction of $\Theta$-orientations on $\tilde{\cY}_{*y}$ and $\cZ(y)$ above. 
    \end{proof}

    While the chain $\cO$ may have depended on the choice of extension $\cR'$, its homology class $[\cO]$ does not. We write $\cO(\cR')$ for the obstruction chain $\cO$ defined using the extension $\cR'$ of $\cR$.
    \begin{lem}\label{lem: ext 0}
        If $\cR'$ is an extension of $\cR$ such that the obstruction chain $\cO(\cR')$ is 0, there is an extension of $\cR'$ to a $\Theta$-oriented pre-$\tau_{\leq n+2}$-right module.
    \end{lem}
    If this holds, it follows that $\cR'$ is in fact a $\tau_{\leq n+1}$-morphism.
    
    \begin{proof}
        Suppose $\cO(\cR') = 0$. Then each $\cZ(y)$ is $\Theta$-oriented nullbordant; let $\cS(y)$ be a $\Theta$-oriented nullbordism for $\cZ(y)$. Then define a pre-$\tau_{\leq n+2}$-module $\cR''$ to have manifolds
        $$\cR''_{*y} = \begin{cases}
            \cR'_{*y} & \textrm{ if } i-|y| \leq n+1\\
            f_{*y}^{-1}(-\infty, 1] \bigcup\limits_{\cZ_{*y}(y)} \cS_{*y}(y)
            & \textrm{ if } i-|y| = n+2
        \end{cases}$$
        equipped with the natural $\Theta$-orientation on each of these manifolds. Here $f_{*y}$ is the function $\tilde \cY_{*y} \to \bR$ chosen earlier in the construction of the obstruction chain $\cO$, and noting that the gluing in the case $i-|y| = n+2$ is along a closed manifold (instead of just a boundary face).
    \end{proof}
    \begin{lem} \label{lem: ob exact}
        Let $\cR'$ be an extension of $\cW$, and $A = dB \in CM_{i-n-2}\left(\cG; \Omega^\Theta_{i}(\cE_{i-n-1})\right)$ an exact chain. Then there is another extension $\cR''$ of $\cR$ such that $\cO(\cR'') = \cO(\cR') + A$.
    \end{lem}
    \begin{proof}
        Let $\cB(y)$ be representatives for the coefficients of the chain $B$, for $y \in \cG$ with $i-|y| = n-1$. Let $\cR''$ be the extension of $\cR$ defined by:
        $$\begin{cases}
            \cR'_{*y} & \textrm{ if } i-|y| \leq n\\
            \cR'_{*y} \sqcup \cB_{*y}(y) & \textrm{ if } i-|y| = n+1
        \end{cases}$$
        which has a natural $\Theta$-orientation. Then $\cO(\cR'') = \cO(\cR') + A$. 
    \end{proof}
    
    \begin{lem}\label{lem: ob other way}
        Suppose $\cR'$ admits an extension $\cR''$. Then the obstruction chain  $\cO(\cR')$ vanishes: $\cO(\cR') = 0$.
    \end{lem}
    
    \begin{proof}
        By \cite[Lemma 3.19]{PS}, there are embeddings of each $\tilde \cY_{*y}$ onto a neighbourhood of the carapace of each $\cR''_{*y}$ when $i-|y| = n+2$. Removing the image of $f_{*y}^{-1}(-\infty, 1)$ provides a nullbordism of $\cZ_{*y}$. 
    \end{proof}

    Lemmas \ref{lem: ob closed}, \ref{lem: ext 0}, \ref{lem: ob exact} and \ref{lem: ob other way} combine to show that:
    \begin{thm}\label{thm: ob lift}
        Let $\cR$ be a $\Theta$-oriented $\tau_{\leq n}$-right module over $\cG$. Then the homology class
        $$[\cO(\cR)] \in HM_{i-n-2}\left(\cG; \Omega^\Theta_{i}(\cE_{i-n-1})\right)$$
        vanishes if and only if $\cR = \tau_{\leq n} \cR'$ for some $\Theta$-oriented $\tau_{\leq n+1}$-morphism $\cR'$.
    \end{thm}
    Theorem \ref{thm: ob lift} is a flow-categorical analogue of the following result in homotopical algebra, see \cite[Proposition 4.49]{PS}. Given $n \geq 0$ and a connective ring spectrum $R$ with $\pi_0 R \cong \bZ$ and a bounded below right $R$-module $M$, there is an exact sequence (see \cite[Proposition 4.49]{PS}):
    \begin{equation}
        \pi_i \left(M\otimes_R \tau_{\leq n+1} R\right) \to \pi_i \left(M\otimes_R \tau_{\leq n} R\right) \to \pi_{i-n-2} \left(M \otimes_R H\pi_{n+1}R\right)
    \end{equation}
    where $\tau_{\leq \cdot}$ denotes the standard $t$-structure on the category of spectra, $\otimes_R$ denotes derived tensor product over $R$, and $HA$ denotes the Eilenberg-Maclane spectrum associated to an abelian group $A$.

\subsection{A Whitehead-type theorem}
    Assume $\Theta$ is an oriented tangential structure, and let $n \geq 0$. Let $\cF$ and $\cG$ be $\Theta$-oriented flow categories, and without loss of generality, we again work under Assumption \ref{asmp: punc dat} for both $\cF$ and $\cG$.

\begin{prop}\label{prop: flow white head}
    Let $\cW: \cF \to \cG$ be a $\Theta$-oriented $\tau_{\leq n}$-morphism, such that the map given by composing with $\tau_{\leq 0} \cW$:
    \begin{equation*}
        \Omega_*^{\Theta, \leq 0}(\cF) \to \Omega_*^{\Theta, \leq 0}(\cG)
    \end{equation*}
    is an isomorphism. Then for $k \leq n$, the map induced by postcomposing with $\tau_{\leq k}\cW$:
    \begin{equation*}
        \Omega_*^{\Theta, \leq k}(\cF) \to \Omega_*^{\Theta, \leq k}(\cG)
    \end{equation*}
    is surjective.
\end{prop}
This is a version of \cite[Theorem 5.51]{PS} with $\Theta$-orientations incorporated- similarly to Theorem \ref{thm: ob lift}, the proof will involve incorporating $\Theta$-orientations into the proof of \cite[Theorem 5.51]{PS}. We will use this in the proof of Theorem \ref{thm: fuk ob}.\par 
The rest of this section is dedicated to the proof of Proposition \ref{prop: flow white head}.

\begin{proof}[Proof of Proposition \ref{prop: flow white head}]
    We proceed by induction on $k$. The proposition holds for $k=0$ by assumption. Now suppose that the proposition holds for some $k<n$; we will show it also holds for $k+1$. Let $\cB$ be a $\Theta$-oriented $\tau_{\leq k+1}$ right module over $\cG$, of degree $i$. Let $G = \Omega_i^{\Theta, \leq k+1}(\cE_{i-n-1})$.\par 

    We first choose an extension $\cB'$ of $\cB$; this is a $\Theta$-oriented pre-$\tau_{\leq k+2}$ right module over $\cG$. By assumption, the bordism class of $\tau_{\leq k} \cB$ is in the image of the map given by composition with $\tau_{\leq k}\cW$. This means there is some $\Theta$-oriented $\tau_{\leq k}$ right module $\cA$ over $\cF$ of degree $i$, such that the composition $\tau_{\leq k}\cW \circ \cA$ of $\tau_{\leq k}\cW$ and $\cA$ is bordant to $\cB$, via some bordism $\cR$. 

    We choose an extension $\cA'$ of $\cA$, as well as a $\tau_{\leq k}$ bordism between $\tau_{\leq k}\cW \circ \cA$ and $\tau_{\leq k}\cB$, and an extension $\cR'$ of $\cR$.

    More concretely, $\cA'$ and $\cR'$ consist of the following manifolds:
    \begin{itemize}
        \item $\cA'_{*x}$ for $x \in \cF$ with $i-|x| \leq k+1$, of dimension $i-|x|$.
        \item $\cR'_{*y}$ for $y \in \cG$ with $i-|y| \leq k$, of dimension $i-|y|+1$.
    \end{itemize}
    such that each $\cA'_{*x}$ has a system of boundary faces given by the $\cA'_{*;x'x} =\cA'_{*x'} \times \cF_{x'x}$, and each $\cR'_{*y}$ has a system of boundary faces given by the $\cB_{*y}$, $\cR'_{*;y'y}$ and $\cA'_{*x} \times \cW_{xy}$. Here $x, x' \in \cF$ and $y,y' \in \cG$. There are also compatible $\Theta$-orientations on all these manifolds.

    Suppose we can find extensions $\cA'',\cR'',\cB'$ of $\cA',\cR',\cB$ respectively, which are compatible in the same way (so these consist of all the same data as above but we also have manifolds $\cA''_{*x}$ when $i-|x|=k+2$ and $\cR''_{*y}$ when $i-|y|=k+1$). We call the triple $(\cA', \cR', \cB)$ an \emph{extension triple} for $(\cA, \cR, \cB)$, and we call the triple $(\cA'', \cR'', \cB')$ a \emph{double extension triple}. Note that no conditions on $\cA'$ and $\cR'$ depend on the extension $\cB'$, they only depend on $\cB$; this is why we write $\cB$ in the triple rather than $\cB'$. If such a double extension triple exists, then $\cA'$ is a $\tau_{\leq k+1}$-right module, whose composition with $\tau_{\leq k+1} \cW$ is bordant to $\cB$, and so we are done. 
    
    We proceed by finding an obstruction class which determines whether there exists some extension triple for $(\cA, \cR, \cB)$ which extends further to a double extension triple. This does not use the assumption that composition with $\tau_{\leq 0} \cW$ is an isomorphism; however, this assumption will imply that the group that this obstruction class lives in vanishes.
    
    We define the obstruction chain:
    $$(\cO^1, \cO^2) = \left(\cO^1(\cA',\cR',\cB), \cO^2(\cA',\cR',\cB)\right) \in CM_{i-k-2}(\cF; G) \oplus CM_{i-k-1}(\cG; G)$$
    as follows.\par 
    $\cO^1$ is defined to be the obstruction chain $\cO(\cA')$ to $\cA'$ having a $\Theta$-oriented extension, as defined in Section \ref{sec: ob lift}. In particular, by Lemma \ref{lem: ob closed}, it is closed with respect to the Morse differential.\par 
    For each $y \in \cG$ with $i-|y|=k+1$, we define $\tilde \cX_{*y}$ to be the following coequaliser (cf. \cite[Page 61]{PS}): 
    \begin{equation}
        \tilde \cX_{*y} := \operatorname{coeq}
        \left( \vcenter{
            \xymatrix{
                \bigsqcup\limits_{y' \in \cG} \cB_{*y'} \times \cG_{y'y} \times [0, \eps)\times [0, -\eps) \ar[r] \ar[dr] &
                \cB_{*y} \times [0, -\eps)\\
                \bigsqcup\limits_{y',y'' \in \cG} \cR'_{*;y''y'y} \times [0, \eps)^2 \ar@<+.5ex>[r] \ar@<-.5ex>[r] &
                \bigsqcup\limits_{y'\in \cG} \cR'_{*;y'y} \times [0, \eps) \\
                \bigsqcup\limits_{\substack{x \in \cF\\ y' \in \cG}} \cA'_{*x} \times \cW_{x;y'y} \times [0, \eps)^2 \ar[ur] \ar[r] &
                \bigsqcup\limits_{x \in \cF} \cA'_{*x} \times \cW_{xy} \times [0, \eps) \\
                \bigsqcup\limits_{x, x' \in \cF} \cA_{*;x'x} \times \cW_{xy} \times [0, \eps)^2 \ar@<+.5ex>[ur] \ar@<-.5ex>[ur]  
            }}
            \right)
    \end{equation}
    As shown in \cite{PS}, each $\tilde{\cX}_{*y}$ is a manifold with faces of dimension $k+2$, with compact boundary. It has a system of boundary faces given by the $\cB_{*y} \times \{0\}$, $\cR'_{*y'} \times \cG_{y'y} \times \{0\}$ for $y' \in \cG$, and $\cA'_{*x} \times \cW_{xy} \times \{0\}$ for $x \in \cF$. Choose framed functions $(f_{*y}, s_{*y})$ on each $\tilde \cX_{*y}$, and let $\cT_{*y} = f_{*y}^{-1}\{1\}$; this is a closed manifold of dimension $k+1$. $\cT_{*y}$ can be equipped with a $\Theta$-orientation in the usual way. We view $\cT_{*y}$ as a $\Theta$-oriented right flow module over $\cE_{i-n-1}$ of degree $i$, which we write as $\cT(y)$.
   
    Similarly to Section \ref{sec: ob lift}, the class $[\cT(y)]$ of $\cT(y)$ in $G$ is independent of choices (for fixed extension triple $(\cA',\cR',\cB)$). \par 
    $\cO^2 = \cO^2(\cA',\cR',\cB)$ is defined to be the sum
    $$\cO^2 = \sum\limits_{\substack{y\in \cG\\
    i-|y|=k+1}} [\cT(y)] \cdot y \in CM_{i-k-1}(\cG;G)$$

    \begin{lem}\label{lem: ob2 ob 0}
        If there is an extension triple $(\cA', \cR', \cB)$ such that the obstruction chain $(\cO^1(\cA', \cR',\cB),\cO^2(\cA',\cR',\cB))$ vanishes, then $[\cB]$ is in the image of the composition map with $\tau_{\leq k+1} \cW$.
    \end{lem}
    \begin{proof}
        We construct an extension $(\cA'', \cR'', \cB')$ to the extension $(\cA', \cR', \cB)$ as follows. $\cA''$ is constructed as in Lemma \ref{lem: ext 0}. $\cR''$ is similar: we let $\cS(y)$ be a $\Theta$-oriented nullbordism of each $\cT(y)$; we then define 
        $$\cR''_{*y} := \begin{cases}
            \cR'_{*y} & 
            \textrm{ if } i-|y| \leq k+1\\
            f_{*y}^{-1}(-\infty,1] \cup_{\cT(y)_{*y}} \cS(y) &
            \textrm{ if } i-|y|=k+2
        \end{cases}$$
        where $f_{*y}: \tilde \cX_{*y} \to \bR$ was the function we chose when defining $\cT(y)$. $\cR''_{*y}$ admits a natural $\Theta$-orientation; $(\cA'', \cR'', \cB')$ together now form the required extension.
    \end{proof}
    \begin{lem}\label{lem: ob2 closed}
        The obstruction chain
        $$(\cO^1, \cO^2) = \left(\cO^1(\cA',\cR',\cB), \cO^2(\cA',\cR',\cB)\right) \in CM_{i-k-2}(\cF; G) \oplus CM_{i-k-1}(\cG; G)$$
        is closed with respect to the mapping cone differential, and therefore represents a class in the homology of the mapping cone of $CM_*(\cW)$.
    \end{lem}
    While the obstruction chain may depend on the choice of extension triple, the homology class does not.
    \begin{proof}
        $\cO^1$ is closed by Lemma \ref{lem: ob closed}, so it remains to show that $CM_*(\cW; G)(\cO^1) + d\cO^2 = 0$. In the case of unoriented flow categories, this is proved in \cite[Section 5.5]{PS}; there a version incorporating framings was also shown. The same proof, with $\Theta$-orientations incorporated in the usual way, proves Lemma \ref{lem: ob2 closed} for any $\Theta$.
    \end{proof}
    
    \begin{lem}\label{lem: ob2 ext exact}
        Let $(\cA', \cR', \cB)$ be an extension triple. For any chain $(A, R) \in CM_{i-k-1}(\cF; G) \oplus CM_{i-k}(\cG; G)$
        there is another extension triple $(\cA'', \cR'', \cB)$ such that 
        \begin{equation}\label{eq: ext trip exact}  
            \left(\cO^1(\cA'',\cR'',\cB), \cO^2(\cA',\cR'',\cB)\right) = \left(\cO^1(\cA',\cR',\cB), \cO^2(\cA,\cR',\cB)\right) + d^{\operatorname{cone}}(A,R)
        \end{equation}
        where $d^{\operatorname{cone}}$ is the mapping cone differential.
    \end{lem}
    \begin{proof}
        For $x \in \cF$ with $|x| = i-k-1$, we let $A(x)$ be a $\Theta$-oriented right flow module representing the $x$-coefficeint of $A$, and we let $\cA''_{*x} = \cA'_{*x} \sqcup A(x)_{*x}$. \par 
        Similarly, for $y \in \cG$ with $|x|=i-k$, we let $R(y)$ be a representative of the $y$-coefficient of $R$, and we let $\cR''_{*y} = \cR'_{*y} \sqcup R(y)_{*y}$. Then $(\cA'', \cR'', \cB)$ is another extension triple satisfying (\ref{eq: ext trip exact}).
    \end{proof}
    Proposition \ref{prop: flow white head} then follows from Lemmas \ref{lem: ob2 ob 0}, \ref{lem: ob2 closed}, \ref{lem: ob2 ext exact} and the fact that the mapping cone of $CM_*(\cW)$ is acyclic (by assumption).
\end{proof}
    While we don't need it for the proof of Proposition \ref{prop: flow white head}, the following lemma provides some motivation for the proof of Proposition \ref{prop: flow white head}.
    \begin{prop}
        We work in the setting (and notation) of the proof of Proposition \ref{prop: flow white head}. Then if there is an extension $(\cA'', \cR'', \cB')$ of $(\cA',\cR', \cB)$, the obstruction $(\cO^1(\cA',\cR',\cB), \cO^2(\cA',\cR',\cB)$ vanish.
    \end{prop}
    \begin{proof}[Sketch proof]
        For $x \in \cF$ with $|x|=i-k-1$, removing an appropriate small open neighbourhood of the boundary of $\cA'_{*x}$ gives a nullbordism of $\cZ(x)_{*x}$, and so the $x$-coefficient of $\cO^1$ vanishes. Similarly for $y \in \cG$ with $|y|=i-k$, removing an appropriate small open neighbourhood of the boundary of $\cR'_{*y}$ provides a nullbordism of $\cT_{*y}$, so the $y$-coefficient of $\cO^2$ vanishes. The nullbordisms in both cases come equipped with $\Theta$-orientations.
    \end{proof}

\subsection{Adaptations for tangential pairs}\label{sec: flow adapt}
    In all of the constructions throughout Section \ref{sec: Thet flow}, we may replace the (possibly oriented) tangential structure with a (possibly oriented) tangential pair $(\Theta, \Phi)$, with no other changes. However, some of the computations differ, as we now explain.

    Let $F$ be the fibre of the map $\Theta \to \Phi$. As in Section \ref{sec: Lag Floer conj}, let $A$ be the ring spectrum which is the Thom spectrum of the index bundle over $\Omega F$ and $R$ the Thom spectrum of the index bundle over $\Omega^2 \Phi$.  Then in analogue with Proposition \ref{prop: bord pt bimod}, we have:
    \begin{prop}\label{prop: flow pt pair case}
        Let $\cF, \cG$ be $(\Theta,\Phi)$-oriented flow categories, each with one object, in degrees $i$ and $j$ respectively. Then the set of bordism classes of $(\Theta,\Phi)$-oriented maps $\cF \to \cG$ is isomorphic to
        $\pi_{i-j}(R \otimes_A R^{op})$.
    \end{prop}
    Motivated by this, we conjecture that the category of $(\Theta,\Phi)$-oriented flow categories agrees with the category of $A$-linear $R-R$ bimodules.
    \begin{proof}[Proof assuming Assumption \ref{asmp}]
        A standard Pontrjagin-Thom argument shows the set of bordism classes of $(i-j)$-manifold with tangent bundle stably classified by the index bundle over $\bU_{11}(\Theta,\Phi)$ agrees with $\pi_{i-j}$ of the Thom spectrum of the negative of the index bundle. Further using Remark \ref{rmk:reverse bundle}, we will identify this Thom spectrum with $R \otimes_A R^{op}$.

        Lemma \ref{lem: htpy type pairs} implies that $\bU_{11}(\Theta, \Phi)$ is the homotopy fibre of the map $\Omega \Theta \times \Omega \Theta \to \Omega \Phi$ which concatenates the loops and then applies the map $\Theta \to \Phi$. We may identify this homotopy fibre with the fibre product $\Omega \Theta \times_{\Omega \Phi} \Omega \Theta$, which we may identify with $\Omega F \otimes_{\Omega^2\Phi} \Omega F$. In this latter expression, $\cdot \otimes_{\Omega^2 \Phi} \cdot$ denotes the (derived) tensor product of a left and a right module over $\Omega^2\Phi$, noting that $\Omega \Phi$ acts on $F$ on both sides; one could model this for example with a bar construction (more standard notation for this tensor product would be $\cdot \times_{\Omega^2\Phi} \cdot$, but this conflicts with the notation for a fibre product in the former expression).
        
        The expression $\Omega F \otimes_{\Omega^2\Phi} \Omega F$ is a homotopy colimit; since the Thom spectrum functor is a homotopy colimit (here we use the model for Thom spectra from \cite{Antolin-Camarena-Barthel}) and (homotopy) colimits commute with each other, we may identify the Thom spectrum over $\bU_{11}(\Theta, \Phi)$ with $R \otimes_A R^{op}$.
    \end{proof}
   By the same proof as above, Proposition \ref{prop: bord pt} holds with the same statement. Similarly, there is an isomorphism between the capsule bordism group $\Omega_i^{\Theta,\Phi,\circ}$ and the $i^{th}$ homotopy group of the Thom spectrum of the index bundle over $\bU_{00}(\Theta,\Phi)$. Combining with Lemma \ref{lem: htpy type pairs}, we conclude: 
   \begin{cor}
       There is an isomorphism between $\Omega_i^{\Theta,\Phi,\circ}$ and $\pi_i \bT_{00}(\Theta,\Phi)$, the homotopy groups of the Thom spectrum of the index bundle over the homotopy limit of (\ref{eq: pair pull}).
   \end{cor}

   \begin{cor}\label{cor: rel THH}
       Assume $\Phi$ is an infinite loop space.
       
       Under Assumption \ref{asmp}, there is an isomorphism between $\Omega^{\Theta,\Phi,\circ}_i$ and the relative topological Hochschild groups $\pi_i THH_A(R)$ (cf. \cite[Chapter IX]{EKMM}).
   \end{cor}
   We expect that the infinite loop space hypothesis is unnecessary.
   \begin{proof}[Sketch proof]
       Since $\Phi$ is an infinite loop space, $A$ is (up to weak equivalence) commutative. By \cite[Theorem 1.8]{ABGHLM} (cf. also \cite[Proposition 3.4]{BMY}), we may then identify: $THH_A(R) \simeq THH(R) \otimes_{THH(A)} R$ . Using \cite{BCS} to identify (nonrelative) $THH$ of Thom spectra as a Thom spectrum and the fact that the Thom spectrum functor commutes with (homotopy) colimits, we may then identify $THH_A(R)$ with the Thom spectrum of an appropriate bundle over $\cL F \otimes_{\cL \Omega \Phi} \Omega F$, which is equivalent to the homotopy limit of (\ref{eq: pair pull}).
   \end{proof}

\subsection{Flow modules over a space}\label{sec: rel flow mod}

    We may also consider left flow modules and capsules equipped with an evaluation map to an auxiliary space $X$ (cf. \cite[Section 7.6]{PS}), compatible with a fixed vector bundle over this space.

    Let $(\Theta, \Phi)$ be a tangential pair of rank $N<\infty$, with $F$ the homotopy fibre of the map $\Theta \to \Phi$. Let $\cF$ be a $(\Theta,\Phi)$-oriented flow category. Let $Z$ be a topological space, and let $E$ be a complex vector bundle over $Z$ of rank $N$, classified by a map $Z \to BU(N)$. Assume this classifying map is equipped with a lift to $\Phi$, $f: Z \to \Phi$.

    \begin{defn}\label{defn:over a space}
        A $(\Theta,\Phi)$-oriented left module over $\cF$ of rank $i$ \emph{living over $(Z, f)$} consists of:
        \begin{itemize}
            \item A $(\Theta,\Phi)$-oriented left module $\cL$ of rank $i$.
            \item ``Evaluation'' maps $ev_{x*}^\cL: \cL_{x*} \to Z$, which are compatible with breaking.
            
        \end{itemize}
        such that for each $x \in \cF$, the following diagram commutes:
        \begin{equation}\label{eq: ev com rel}
            \xymatrix{
                \cL_{x*}
                \ar_{ev_{x*}}[r]
                \ar_{\rho_{x*}}[d]
                &
                Z
                \ar_f[d]
                \\
                \bU_{x*}
                \ar_\lambda[r]
                &
                \Phi
            }
        \end{equation}
        where the map $\lambda$ is described as follows. Let $\bD_{j0}$ be a $(\Theta,\Phi)$-oriented abstract disc with $j$ inputs and no outputs, with underlying domain $D$ and homotopy lift to $\Phi$ given by a map $h: D \to \Phi$. Let $p$ be the interior marked point in $D$. Then we set $\lambda(\bD) = h(p) \in \Phi$. In words, (\ref{eq: ev com rel}) says that the $\Phi$-structure on $ev^*E$ and the $\Phi$-structure obtained by restricting the complex part of the Cauchy-Riemann data to the interior marked point agree.

        Similarly, a $\Theta$-oriented capsule $\cM$ of rank $i$ \emph{living over $(Z, f)$} consists of a $(\Theta,\Phi)$-oriented capsule $\cM$, along with an evaluation map $ev^\cM_{**}: \cM_{**} \to Z$ such that the analogue of (\ref{eq: ev com rel}) commutes.

        Bordisms of left flow modules and capsules over $(Z,f)$ are defined similarly. We write $\Omega^*_{\Theta,\Phi}(\cF; Z, f)$ and $\Omega^{\Theta,\Phi,\circ}_*(Z,f)$ for the corresponding left module and capsule bordism groups respectively.
    \end{defn}
    Note that the capsule bordism groups $\Omega^{\Theta,\Phi,\circ}_*$ do not necessarily agree with the capsule bordism groups living over a point: instead they compute $\Omega^{\Theta,\Phi,\circ}_*(\Phi)$, where $\Phi$ is equipped with the tautological $\Phi$-structure and corresponding vector bundle.

    However, there is a forgetful map which forgets the map to $Z$, from $\Omega_*^{\Theta,\Phi,\circ}(Z,f)$ to $\Omega_*^{\Theta,\Phi,\circ}$.

\begin{rmk}[Stabilisation]
    Assume $(\Theta(N), \Phi(N))$ and $(\Theta(N+1), \Phi(N+1))$ are tangential pairs of rank $N<\infty$ and $N+1$ respectively, with a map of pairs $(\Theta(N), \Phi(N)) \to (\Theta(N+1), \Phi(N+1))$ compatible with stabilisation in $B\operatorname{Pin}(N)$ and $BS_\pm U(N)$.

    Then given a $\Phi(N)$-structure on $E$, there is a canonical induced $\Phi(N+1)$-structure on $\bC \oplus E$, which we write as $\bC \oplus f$. Using this, we obtain stabilisation maps for left module and capsule bordism over a space:
    \begin{equation*}
        \Xi: \Omega^*_{\Theta(N), \Phi(N)}(\cF; Z, f) \to \Omega^*_{\Theta(N+1), \Phi(N+1)}\left(\Xi(\cF); Z, \bC \oplus f\right)
    \end{equation*}
    and
    \begin{equation*}
        \Xi: \Omega_*^{\Theta(N), \Phi(N),\circ}(Z,f) \to \Omega_*^{\Theta(N+1),\Phi(N+1), \circ}(Z,\bC \oplus f)
    \end{equation*}

    Similarly for tangential pairs of rank $N=\infty$, we may define left module and capsule bordism by taking the sequence of tangential pairs defined by pullback similarly to Definition \ref{rmk: N inf fin}. Then assuming $\cF$ has a $(\Theta(N_0), \Phi(N_0))$ structure, $E$ has rank $N_0$ and $E$ is equipped with a $\Phi(N_0)$-structure $f$ for some $N_0$, we define $\Omega^*_{\Theta,\Phi}(\cF; Z, f)$ and $\Omega_*^{\Theta, \Phi}(Z, f)$ to be the colimit as $N \to \infty$ of the maps $\Xi$ above.
\end{rmk}

    Now let $(\Theta, \Phi)$ be a tangential pair of rank $1 \leq N \leq \infty$. Generalising Proposition \ref{prop: capsule comp}, we have:
    \begin{prop}\label{prop: rel capsule}
        The capsule bordism groups $\Omega_*^{\Theta,\Phi,\circ}(Z,f)$ are isomorphic to the bordism groups of manifolds with tangent bundle stably classified by the index bundle over the homotopy limit of the diagram
        \begin{equation}\label{eq: hompul}
            \xymatrix{
                &
                Z
                \ar[d]
                \\
                \cL \Theta
                \ar[r]
                &
                \cL \Phi
            }
        \end{equation}

    \end{prop}
    \begin{proof}
        By pulling back the fibration $\Theta \to \Phi$ along $f$, we obtain a new tangential structure $f^*\Theta \to Z$. Then there is a natural isomorphism:
        \begin{equation*}
            \Omega_*^{\Theta,\Phi,\circ}(Z,f) \cong \Omega_*^{f^*\Theta, Z, \circ}
        \end{equation*}
        and so we deduce the result from Proposition \ref{prop: capsule comp}.
    \end{proof}
    Under Assumption \ref{asmp}, Corollary \ref{cor: rel THH} identifies these groups with $THH_{\operatorname{Thom}(\Omega^2 Z \to BU)}(R)$.
    \begin{cor}
        \begin{enumerate}
            \item If $F$ is contractible, $\Omega^{\Theta,\Phi,\circ}_*(Z,f) \cong \Omega^{fr}_*(Z, E)$, where the RHS is consists of bordism classes of manifolds mapping to $Z$ with a framing relative to $E$. For $N < \infty$, we may identify this with $\pi_{*-N}(\operatorname{Thom}(-E \to Z))$ (where $\operatorname{Thom}$ denotes the Thom spectrum rather than space).

            \item If $\Phi$ is contractible, $\Omega^{\Theta,\Phi,\circ}_*(Z,f)$ is isomorphic to the bordism groups of manifolds over $Z$ with tangent bundles classified by maps to $\bU_{00}^F \simeq \cL F$, equivalently $\pi_{*+N}(Z \wedge \bT_{00}(F))$.

            \item If $Z$ is contractible, the homotopy pullback of (\ref{eq: hompul}) is $\cL F$ and therefore $\Omega^{\Theta,\Phi,\circ}_*(Z,f)$ is isomorphic to $\Omega_*^{F, \circ}$, the capsule bordism groups with respect to the tangential structure $F$.
        \end{enumerate}
        
    \end{cor}
    \begin{cor}\label{cor: aaa}
        Assume Assumption \ref{asmp}. 

        \begin{enumerate}
        
            \item If $\Phi$ is contractible, $\Omega^{\Theta,\Phi,\circ}_*(Z,f)$ is isomorphic to $\pi_*(Z \wedge THH(R))$.

            \item If $Z$ is contractible, the homotopy pullback of (\ref{eq: hompul}) is $\cL F$ and therefore $\Omega^{\Theta,\Phi,\circ}_*(Z,f) \cong \pi_*(THH(R))$.
        \end{enumerate}
    \end{cor}
    If $\Phi$ and $\Theta$ are contractible, we recover the framed bordism groups of $Z$, as in \cite{PS} (note $THH(\bS) \simeq \bS$).

    If $\Theta$ and $\Phi$ admit the structures of loop spaces (compatible with the diagram (\ref{eq: strong})), $R$ is then an $\bE_2$-ring spectrum, and there is a splitting \cite[Theorem 2]{BCS}
    \begin{equation*}
        THH(R) \cong R \wedge T
    \end{equation*}
    where $T$ is the Thom spectrum of a virtual bundle over $F$, and is a suspension spectrum if furthermore $F$ is a double loop space \cite[Theorem 3]{BCS}. 

    In particular, if $\Phi$ is contractible and (\ref{eq: strong}) is a diagram of double loop spaces, one expects a map $\Omega^{\Theta,\Phi,\circ}_*(Z, f) \to \Omega^{\Theta,\Phi}_*(Z)$. We produce this in the complex-oriented setting:

    \begin{lem} \label{lem:MU splitting}
        Let $1 \leq N \leq \infty$, and $(\Theta, \Phi)$ be (a fibrant replacement of) either:
        \begin{itemize}
            \item $(\operatorname{pt}, BS_\pm U(N))$
            \item $(B\operatorname{Pin}(N), B\operatorname{Pin}(N) \times BS_\pm (N))$, where the map $\Phi \to BS_\pm U(N)$ forgets the first factor and the map $\Theta \to \Phi$ is given by the identity on the first factor and complexification in the second.
        \end{itemize}
        
        Then there is a natural map to complex bordism relative to the bundle $E$:
        \begin{equation*}
            \Omega_*^{\Theta,\Phi,\circ}(Z,f) \to \Omega^U_*(Z, E) 
        \end{equation*}
        where $\Omega^U_*(Z, E)$ denotes bordism of manifolds over $Z$, say $h:M \to Z$ equipped with stable complex structures on $TZ-h^*E$. If $N < \infty$, the right hand side can be identified with $MU_{*-N}(\operatorname{Thom}(-E \to Z))$. 
    \end{lem}
    The existence of the `projection' of Lemma \ref{lem:MU splitting} reflects, in our setting, the existence of the natural splitting $THH(MU) \to MU$ \cite{Blumberg-THH}.

    \begin{proof}
        In the first case, a $(\Theta,\Phi)$-oriented capsule over $Z$ includes a manifold whose tangent bundle is classified by a map to the space of Cauchy-Riemann operators over a disc, with trivial real part and possibly non-trivial complex part. By \cite[Proposition 3.15]{P}, such a manifold admits a natural stable complex structure. A similar argument applies to bordisms.

        In the second case, the same argument applies by working with bordisms relative to $E$.
    \end{proof}

    Consider the rank $N$ complex Lagrangian Grassmannian $Sp/U(N)$, which we view as the space of totally real subspaces of $\bH^N$ with respect to $K$ ($I$,$J$,$K$ are the three complex structures). There is an inclusion map $\zeta: Sp/U(N) \to U/O(2N)$ by forgetting $I$ and $J$. Since $Sp/U(N)$ is simply connected, there is a canonical lift $Sp/U(N) \to \widetilde{U/O}(N)$. Recall that a version of Bott periodicity implies that $\Omega Sp/U \simeq U/O$.

    Let $\eta: U/O \to BO$ be the forgetful map, classifying the real part of a real subspace of $\bC^N$. Let $\Omega^\eta_*$ be the bordism groups of manifolds equipped with the tangential structure $\eta$.

    \begin{lem}\label{lem:Sp/U splitting}
        Let $\Phi$ be contractible, and $\Theta = Sp/U \to \widetilde{U/O}(N)$ be the infinite complex Lagrangian Grassmannian equipped with the map above. Then there is a natural map from $(\Theta,\Phi)$-capsule bordism to $\eta$-bordism over $Z$:
        \begin{equation*}
            \Omega^{Sp/U,\circ}_*(Z) \to \Omega^\eta_*(Z)
        \end{equation*}
    \end{lem}
    \begin{proof}
        Let $M$ be a manifold equipped with a map $h: M \to \cL (Sp/U)$ ($\cL$ denotes the free loop space), such that $TM$ is stably isomorphic to the index bundle of $\zeta \circ h$, as well as a map $M \to Z$. This is a representative of an element of $\Omega^{Sp/U,\circ}_*$

        We will produce a map $f: M \to U/O$, such that $TM$ is stably isomorphic to the real part of $f$, i.e. $TM$ is classified by $\eta \circ f$. This provides the required splitting, noting the same procedure can also be applied to bordisms.

        For some $N\gg 0$, $h$ factors through $\cL (Sp/U(N))$. Picking appropriate complements of vector bundles, we find that there exist maps
        \begin{align*}
            \alpha: M \to \Omega (Sp/U)(N')
            \\
            \beta: M \to Sp/U(N'')
        \end{align*}
        such that $N'+N''=N$, and there is a stable isomorphism between $TM$ and the real part of $\zeta \circ \beta$ direct summed with the index bundle of $(\Omega \zeta) \circ \alpha$.

        The real part of $\zeta \circ \beta$ is (by construction) the real part of a map $M \to U/O$, so it suffices to show the same for the index bundle of $(\Omega \zeta) \circ \alpha$.

        We define a map $F: U/O(N') \to \Omega (Sp/U)(N')$ as follows. For $V \subseteq \bC^{N'}$ a real subspace, we define $F(V)$ to be the loop sending $t \in [0,1]$ to $e^{\pi t K}(V \oplus JV) \subseteq \bH^{N'}$ (where we view $\bH$ as $\bC_I \oplus J \bC_I$). The connectivity of this map goes to $\infty$ as $N' \to \infty$; we may therefore lift $\alpha$ to a map $g: M \to U/O(N')$.

        By treating $Sp/U(N')$ as a subspace of $U/O(2N')$, \cite[Theorem A]{deSilva} (cf. also \cite[Claim 3.17]{P}) implies that the real part of $g$ is stable isomorphic to the index bundle of $(\Omega \zeta) \circ \alpha$. The result follows.
    \end{proof}

\subsection{Complex structures}
    Consider the tangential pair where $\Theta$ is contractible and $\Phi=BS_\pm U(N)$. $\operatorname{Flow}^{\Theta,\Phi}$ can be viewed as a model for complex-oriented flow categories, but it is not \textit{a priori} the same as Abouzaid-Blumberg's model. Here we briefly sketch how to go from our model to theirs.

    In \cite{AB}, a \emph{stable complex structure} on a manifold $\cM$ consists of a pair of vector spaces $V^\pm$ and a pair of complex vector bundles $E^\pm$, along with an isomorphism 
    \begin{equation*}
        T\cM \oplus V^- \oplus E^+ \cong V^+ \oplus E^-
    \end{equation*}

    Fix some (closed-string) puncture datum $\bF$ of Maslov index 0. Fix $i \geq 0,j \in \{0,1\}$, a tuple of framed puncture data $\bE$ with $i$ boundary input punctures and $j$ boundary output punctures. Fix also vector spaces $V,V'$. Let $\bU = \bU_{i+1,j}^{V,\Theta,\operatorname{pt}}(\bE \cup \bF)$, $\bU' = \bU_{01}^{V', \Theta,\Phi}(\bF)$, $V''=V\oplus V'$ and $\bU'' = \bU_{ij}^{V'',\Theta,\Phi}(\bE)$. Write $\bV,\bV',\bV''$ for the corresponding index bundles. 
    
    Gluing and functoriality in the tangential pair together define a map $\bU \times \bU' \to \bU''$, with an isomorphism $\bV \oplus \bV' \cong \bV''$ covering it. The connectivity of this map goes to $\infty$ as the dimensions of $V$ and $V'$ do.
    
    Let $\cM$ be a manifold with stable tangent bundle classified by a map to $\bU''$. Assuming the dimensions of $V$ and $V''$ are sufficiently high compared to $\operatorname{Dim}(\cM)$, we may lift this to $(f,g): \cM \to \bU' \times \bU''$. We obtain a decomposition 
    \begin{equation}\label{eq: split dec}
        T\cM \oplus V'' \cong f^*\bV \oplus g^*\bV'
    \end{equation}
    Since $\bU$ consists of framed discs, $f^*\bV$ is a trivial bundle. 

    Fix some $\bD \in \bU_{01}^{0,\Theta,\Phi}(\bF)$, a rational curve with one output puncture with puncture datum $\bF$ and perturbation data with domain 0. The kernel of the corresponding Cauchy-Riemann operator over $\bD$ is 0. Gluing with $\bD$ gives an isomorphism between $g^*\bV'$ and the index bundle of a family of Cauchy-Riemann operators over a family of closed rational curves; turning off the inhomogeneous term allows us to deform these operators to complex-linear ones, and hence equip $g^*\bV'$ with a complex structure. Thus the decomposition (\ref{eq: split dec}) endows with $\cM$ with a stable complex structure in the sense of \cite{AB}.

    By allowing arbitrary numbers of interior punctures (and dealing with some coherence issues involving how exactly one glues caps), these steps can be performed compatibly with gluing, allowing one to obtain a complex-oriented ($\bZ$-graded, finite) flow category in the sense of \cite{AB} from a $(\operatorname{pt},BS_\pm U(N))$-oriented flow category.

\section{Floer theory with tangential structures}\label{sec: Floer tang}

\subsection{Setting}

Consider a stably framed  Liouville manifold $X^{2d}$, with a fixed homotopy class of stable framing $TX \oplus \bC^{N-d} \cong \bC^N$. 

Recall the \emph{Floer flow category} $\cM^{LK}$ associated to a pair of graded closed exact Lagrangians $L$ and $K$ has objects given by Hamiltonian chords $x$ from $L$ to $K$ (for some fixed choice of Hamiltonian depending on $L$ and $K$), morphism spaces given by moduli spaces of Floer strips, composition given by concatenation and grading given by the Conley-Zehnder index. That this constitutes an (unoriented) flow category was proved by Large \cite{Large} and Fukaya-Oh-Ohta-Ono \cite{FO3:smoothness}.

Now fix a (graded) tangential structure $\Theta \to \widetilde{U/O}(N)$ of rank $N < \infty$ (we extend to the case $N=\infty$ in Remark \ref{rmk: Floer N inf}). The goal of this section is (i) to show that if the stable Gauss maps of Lagrangians $L$ and $K$  lift to $\Theta$, then the moduli spaces arising in the Floer flow category $\cM^{LK}$  admit $\Theta$-orientations; and (ii) to thence define a $\Theta$-oriented spectral Fukaya category with such Lagrangians as objects. 

\begin{rmk}\label{rmk:geometric bounded}
    An important point is that the discussion in this section uses only that $L$ and $K$ intersect in finitely many points; in particular, it applies to non-compact Lagrangians in a `partially wrapped' setting. More precisely, we can allow Lagrangians which are  geometrically bounded (for instance, cylindrical at infinity), where we remove intersections at infinity by `infinitesimal wrapping', i.e. by choosing regular Floer data which incorporates Hamiltonian perturbations which are very small time Reeb flows at infinity.   Such choice of wrapping ensures that the flow categories we consider still have finitely many objects, and geometric boundedness implies that monotonicity and Gromov compactness arguments apply.
 \end{rmk}

 The construction of the spectral Fukaya category uses the existence of  $\Theta$-orientations on all moduli spaces arising in \cite[Section 8]{PS} (so those arising from continuation maps, spaces of holomorphic triangles, etc).  For concreteness we will focus the exposition on the case of moduli spaces of strips, and leave the other cases to the reader.

The construction of the $\Theta$-orientation on $\cM^{LK}$ requires maps 
\begin{equation} \label{eqn: Floer to abstract}
\cM_{xy} \to \bU_{xy}^\Theta
\end{equation}
to the space of abstract strips with puncture data determined by $x,y$, and 
which are compatible with gluing (and covered by maps of index bundles).  These maps should be covered by maps of index bundles and satisfying various associativity compatibilities. 
  
The linearisation of a Floer solution gives rise to an operator $\overline{\partial} + Y$ on the disc for which $Y$ decays exponentially to $dt$ in the strip-like ends, but does not agree with it. The map \eqref{eqn: Floer to abstract} thus uses cut-offs, and then   those are made suitably associative by interpolating in collar neighbourhoods of boundary strata.  The arguments are variations on the case of stable framings discussed in \cite[Section 8]{Large}.

\subsection{Charts}

We consider a pair of transverse Lagrangian branes $L,K$ and a regular moduli space $\cM_{xy} = \cM^{LK}_{xy}$ of (possibly broken) Floer strips between Hamiltonian chords $x,y$.

We recall the construction of charts on $\cM_{xy}$ from \cite{Large}, and reviewed in more detail in \cite{PS}. We will write $\mathring{\cM}_{xy}$ for the interior of the moduli space, i.e. the open subset of unbroken Floer strips.

Suppose $\dim_{\bR}(\cM_{xy}) = k$. The chart will be centred on a particular point $u\in \mathring{\cM}_{xy}$; in order to construct the chart, we will pick a representative of $u$, i.e. a lift to an actual map, since $\cM_{xy}$ is a quotient by an $\bR$-translation action. By abuse of notation, we will also call this representative $u$. We fix a collection $\bf{H}$ of $(k+1)$ real hypersurfaces $H_i \subset X$ and points $z_i = (s_i,t_i) \in \bR\times [0,1]$ for which 
\begin{enumerate}
\item $-L(u) = s_0 < s_1 < \cdots < s_k = L(u)$, where $L_{\bf{H}}(u) = L(u) = \frac12(s_k-s_0)$;
    \item the $z_i$ are regular in the sense of \cite{FHS};
    \item $ u(z_i) \in H_i$.
\end{enumerate}
We call such a collection a \emph{centred system of hypersurfaces} (the centring refers to the equalities $s_0=-L(u)$ and $s_k=L(u)$ in (1) which fix the representative map in our equivalence class, which was \emph{a priori} defined only up to translation).

There is an open $u\in V= V_{\bf{H}} \subset \cM_{xy}$ such that for any $v\in V$, there are unique points $z_i(v) = (s_i(v),t_i)$ with $v(z_i(v)) \in H_i$. The map
\[
v \mapsto (s_1(v)-s_0(v),\ldots, s_k(v)-s_{k-1}(v))
\]
defines a local isomorphism $V_{\bf{H}} \to \bR^k$; note that since we take differences of the $s_i$-coordinates this is a chart in the moduli space of maps mod  the translation action.

We now consider gluing. We have a curve $(u_1,u_2,\ldots,u_l) \in \mathring{\cM}_{x_0x_1} \times \cdots \times \mathring{\cM}_{x_{l-1}x_l}$ and we have picked centred systems of hypersurfaces $\bf{H}^j$ for $j=1,2,\ldots, l$ defining $V_{\bf{H}_i} \subset \cM_{x_{i-1}x_i}$. Let $U=V_{\bf{H}_1} \times \cdots \times V_{\bf{H}_l}$. Gluing defines a map 
\[
U \times [T_0,\infty)^{l-1} \to \cM_{x_0x_l}
\]
covering a punctured neighbourhood of the boundary point $(u_1,\ldots,u_l) \in \cM_{xy}$ and extending continuously to allow gluing parameter $T=\infty$; see \cite[Proposition 8.4]{PS} and \cite[Corollary 6.8 and Proposition 6.9]{Large} for an explicit description of the gluing scheme.

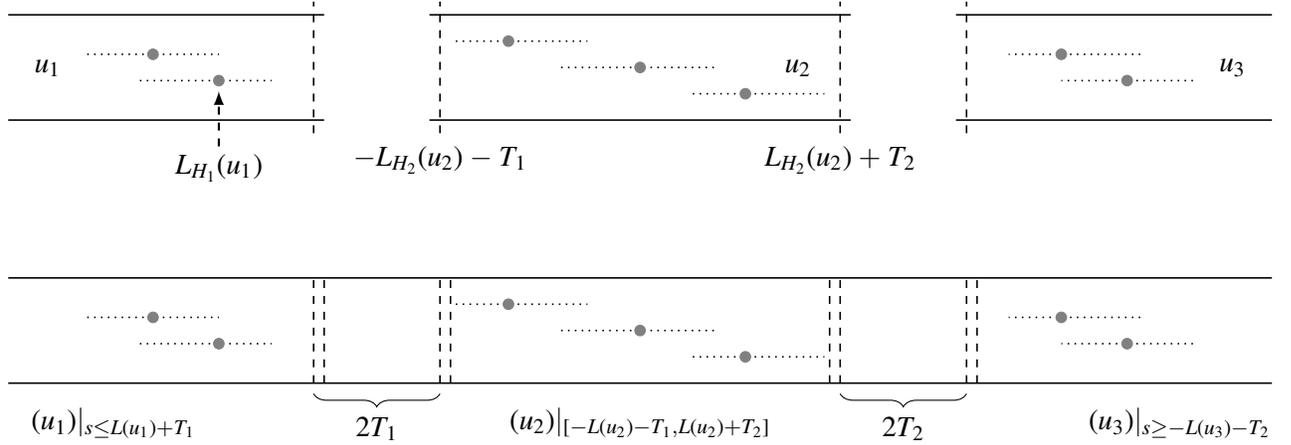
\begin{figure}[ht]
\begin{center}
\begin{tikzpicture}[scale=0.7] \label{Fig:gluing}

\draw[semithick] (-4,1) -- (4,1);
\draw[semithick] (-4,-1) -- (4,-1);
\draw[semithick,dotted] (-3.5,0.5) -- (-1,0.5);
\draw[semithick,dotted] (-1.5,0) -- (1.5,0);
\draw[semithick,dotted] (1,-0.5) -- (3.5,-0.5);
\draw[fill,gray] (-2.5,0.5) circle (0.1);
\draw[fill,gray] (0,0) circle (0.1);
\draw[fill,gray] (2,-0.5) circle (0.1);
\draw[semithick, dashed] (-3.8,-1.25) -- (-3.8,1.25);
\draw[semithick, dashed] (3.8,-1.25) -- (3.8,1.25);
\draw (-3.8,-1.75) node {$-L_{H_2}(u_2)-T_1$};
\draw (3.8,-1.75) node {$L_{H_2}(u_2) + T_2 $};

\draw[semithick] (-12,1) -- (-6,1);
\draw[semithick] (-12,-1) -- (-6,-1);
\draw[semithick,dotted] (-10.5,0.25) -- (-8,0.25);
\draw[semithick,dotted] (-9.5,-0.25) -- (-7,-0.25);
\draw[fill,gray] (-8,-0.25) circle (0.1);
\draw[fill,gray] (-9.25,0.25) circle (0.1);
\draw[thick,dashed,arrows=->] (-8, -1.5) -- (-8,-0.45);
\draw (-8,-1.9) node {$L_{H_1}(u_1)$};
\draw[semithick,dashed] (-6.2,-1.25) -- (-6.2,1.25);

\draw[semithick] (6,1) -- (12,1);
\draw[semithick] (6,-1) -- (12,-1);
\draw[semithick,dotted] (10.5,-0.25) -- (8,-0.25);
\draw[semithick,dotted] (9.5,0.25) -- (7,0.25);
\draw[fill,gray] (9.25,-0.25) circle (0.1);
\draw[fill,gray] (8,0.25) circle (0.1);
\draw[semithick,dashed] (6.2,-1.25) -- (6.2,1.25);

\draw[semithick] (-12,-4) -- (12,-4);
\draw[semithick] (-12,-6) -- (12,-6);
\draw[semithick,dotted] (-3.5,-4.5) -- (-1,-4.5);
\draw[semithick,dotted] (-1.5,-5) -- (1.5,-5);
\draw[semithick,dotted] (1,-5.5) -- (3.5,-5.5);
\draw[fill,gray] (-2.5,-4.5) circle (0.1);
\draw[fill,gray] (0,-5) circle (0.1);
\draw[fill,gray] (2,-5.5) circle (0.1);
\draw[semithick,dotted] (-10.5,-4.75) -- (-8,-4.75);
\draw[semithick,dotted] (-9.5,-5.25) -- (-7,-5.25);
\draw[fill,gray] (-9.25,-4.75) circle (0.1);
\draw[fill,gray] (-8,-5.25) circle (0.1);
\draw[semithick,dotted] (10.5,-5.25) -- (8,-5.25);
\draw[semithick,dotted] (9.5,-4.75) -- (7,-4.75);
\draw[fill,gray] (8,-4.75) circle (0.1);
\draw[fill,gray] (9.25,-5.25) circle (0.1);
\draw[semithick, dashed] (-3.8,-6) -- (-3.8,-4);
\draw[semithick, dashed] (3.8,-6) -- (3.8,-4);
\draw[semithick, dashed] (-3.6,-6) -- (-3.6,-4);
\draw[semithick, dashed] (3.6,-6) -- (3.6,-4);
\draw[semithick,dashed] (-6.2,-6) -- (-6.2,-4);
\draw[semithick,dashed] (-6,-6) -- (-6,-4);
\draw[semithick,dashed] (6.2,-6) -- (6.2,-4);
\draw[semithick,dashed] (6.4,-6) -- (6.4,-4);

\draw (-11.25,0) node {$u_1$};
\draw (3,0) node {$u_2$};
\draw (11.25,0) node {$u_3$};
\draw (-10,-6.75) node {$(u_1)|_{s\leq L(u_1)+T_1}$};
\draw (0,-6.75) node {$(u_2)|_{[-L(u_2)-T_1, L(u_2)+T_2]}$};
\draw (10.25,-6.75) node {$(u_3)|_{s \geq -L(u_3)-T_2}$};
\draw [decorate,decoration={brace,amplitude=5pt,mirror}]
  (-6.2,-6.2) -- (-3.8,-6.2) node[midway,yshift=-1.2em]{$2T_1$};
\draw [decorate,decoration={brace,amplitude=5pt,mirror}]
  (3.8,-6.2) -- (6.2,-6.2) node[midway,yshift=-1.2em]{$2T_2$};

\end{tikzpicture}
\end{center}
\caption{Gluing with lengths $T_1,\ldots, T_{n-1}$, case $n=3$}
\end{figure}

We view this as giving a continuous map
\[
U \times [T_0,\infty]^{l-1} \to \cM_{x_0x_l}, \qquad U \times [T_0,\infty)^{l-1} \hookrightarrow \mathring{\cM}_{x_0x_l}.
\]
   Restricted to the image in $\mathring{\cM}$, the transition map between two charts of the above form, associated to two different choices of centred systems of hypersurfaces $\bf{H}^j$, decays exponentially towards
\begin{equation} \label{eqn:topological_tangent}
(u,{\bf{T}}) \mapsto (u, {\bf{T}}+f(u))
\end{equation}
for some $f: U \to \bR^{l-1}$ and 
as all $T_i \to \infty$.

Taking gluing profile $\psi: [0,1/T_0) \to (T_0,\infty]$, $t\mapsto 1/t$,  and working in co-ordinates $(u,\bf{r} = \psi({\bf{T}}))$ gives rise to charts with smooth transition functions which decay exponentially to
\begin{equation}\label{eqn:smooth_tangent}
(u,{\bf{r}}) \mapsto (u, \psi^{-1}(\psi({\bf{r}}) + f(u)))
\end{equation}

\begin{cor}
    The Floer flow category $\cM^{LK}$ admits the structure of a smooth unoriented flow category.
\end{cor}

\subsection{Linearised operators}

We will fix a weight $\kappa \geq 2$ and work with the Sobolev spaces $W^{2,\kappa}$ of functions with $\kappa$ derivatives in $L^2$. (On a two-dimensional surface once $\kappa\geq 2$ these functions are continuous and one can define restriction to the boundary without using a trace.)

 If $E\to \Sigma$ is an almost complex bundle on a surface and $F \subseteq E|_{\partial \Sigma}$ is a totally real subbundle, then real Cauchy-Riemann operators in $E$ are operators
    \[
    D: \Omega^0_{W^{2,\kappa}}(E, F) \to \Omega^{0,1}_{W^{2,\kappa-1}}(E)
    \]
    satisfying 
    \[
    D(f\xi) = \overline{\partial}_J(f)\xi + f\cdot D\xi.
    \]
    The space of all such forms an affine space over the space of linear maps $E \to \overline{\mathrm{Hom}}_{\bC}(T\Sigma, E)$.
    
Fix a symplectic connection $\nabla$ on $TX$.   Let $u \in \mathring{\cM}_{xy}$ and pick a lift of $u$ to a map (not modulo translation), which has domain $Z=\bR\times [0,1]$; for instance, a complete system of hypersurfaces near $u$ determines a distinguished such lift, by taking the associated co-ordinates $s_i$ to be centred (meaning they satisfy the conditions (1)-(3) listed previously). We write $u^*TL \sqcup u^*TK$ for the totally real subbundle of $u^*TX|_{\partial Z}$ given by $u^*TL$ over $Z \times \{0\}$ and by $u^*TK$ over $Z \times \{1\}$. The linearised operator $D_u: W^{2,k}(u^*TX, u^*TL \sqcup u^*TK) \to W^{2,k-1}(\Omega^{0,1}_Z \otimes u^*TX)$ at $u$ is a real Cauchy-Riemann operator of the form
\[
D_u: \xi \mapsto \nabla^{0,1}(\xi) + Y
\]
 for a zero-th order operator $Y$.
    Note that if there is no Hamiltonian term in the Floer equation, the first term is independent of the choice of connection and is exactly $\overline{\partial}_J$, compare to \cite{OhBook}. If we furthermore use the connection to give a unitary trivialisation of $u^*TX \to Z$ for $u: Z \to X$ holomorphic, the linearised operator has the standard form
    \begin{equation} \label{eqn:linearised standard}
        \xi \mapsto \overline{\partial}_J(\xi) + Y, \quad Y \in \Omega^{0,1}_Z \otimes \mathrm{End}\,\bR^{2n}
    \end{equation}

\begin{rmk}\label{rmk:kernel defined}
The elements of $\ker(D_u)$ depend on the choice of map in the equivalence class of $u\in\cM_{xy}$ under the $\bR$-action only up to translation, so this vector space is defined up to canonical isomorphism (e.g. under changing the choice of complete system of hypersurfaces). 
\end{rmk}
\begin{rmk} \label{rmk:asymptotics}
    In general, for the perturbed Floer equation, $Y$ is asymptotic to an operator $Y_x$ arising from linearising the Floer equation at a translation-invariant solution $u(s,t) = x(t)$ with $x(t)$ a Hamiltonian chord.  When $H=0$ and the asymptotic is an isolated transverse intersection, $Y$ is asymptotically holomorphic to the constant one-form $dt$ in the strip-like end. 
\end{rmk}

Appealing to Remark \ref{rmk:kernel defined}, there is a bundle 
\[
\Ind_{xy} \to \mathring{\cM}_{xy}
\]
with fibre $\ker(D_u)$.  Since the stabilisation of the operator to a Cauchy-Riemann operator in an $N$-dimensional vector bundle over the disc doesn't change its kernel or cokernel, we can view this as associated to the stabilised operator.

At each intersection point of $L,K$ we have a Darboux chart in which the almost complex structure is standard, and in which the stabilised tangent spaces are standard.

Assume we have fixed lifts of the  Gauss maps of $L,K$ from $U/O(N)$ to $\Theta$. 

    \begin{lem} \label{lem linearised discs as abstract discs}
    A choice of sufficiently large non-negative function $(L_x,L_y): \cM_{xy} \to \bR_{\geq 0}^2$ 
        determines a map $a_{xy}$ from  $\mathring{\cM}_{xy}$ to the space $\bU^{\Theta}_{1,1}$ of abstract discs with Cauchy-Riemann data (with trivial perturbation domain $V=\{0\}$) and a $\Theta$-lift. Moreover, $a_{xy}^*\bV^{\Theta} = \Ind_{xy}$.
    \end{lem}

\begin{proof}
The function $(L_x,L_y)$ determines the length parameters for strip-like ends.  We choose the strip-like ends so that any choice of positive $(L_x,L_y)$ ensures that  the  strip-like ends lie inside open neighbourhoods of the intersection points in which the almost complex structure is standard and integrable. Remark \ref{rmk:asymptotics} says that the operator $Y$ in the linearisation is asymptotic to $dt$, and we cut it off to equal $dt$ using a cut-off function supported in a strip of width one at the boundary of the shrunken strip-like ends.  This associates an abstract disc to each element of $\mathring{\cM}_{xy}$.
Since we are assuming that the moduli space is regular, the trivial perturbation data $(V,f) = 0$ suffices to make the linearised operators surjective. The required $\Theta$-lifts come from the hypothesised Gauss map data on $L$ and $K$. The final statement holds essentially by definition of $\bV_{xy}$ and the fact that, once the lengths are sufficiently large, $L^2$-orthogonal projection identifies the kernel of $D_u$ with the kernel of the operator obtained by cutting off $Y$ near the ends, by Lemma \ref{linear gluing}.
\end{proof}

Our next steps are to extend this to the whole of $\cM_{xy}$, compatibly with gluing.

\subsection{Tangent and index bundles}

We have constructed a map $a_{xy}: \mathring{\cM}_{xy} \to \bU^{\Theta}_{1,1}$. Even with gluing of abstract discs this is not quite sufficient to give the maps \eqref{eqn: Floer to abstract} because $a_{xy}$ is defined only on the interior (top open stratum of unbroken curves).  Essentially equivalently to extending this to the compactification, we will use systems of collars to give preferred embeddings of the boundary strata into the interior.  We will start by constructing a preferred system of normals, which yields a compatible system of collars by \cite[Lemma 3.19]{PS}.

The interior $\mathring{\cM}$ of a moduli space $\cM = \cM_{x_0x_l}$ is naturally smooth and has a tangent bundle, defined without any choice of gluing profile.  On the other hand, having chosen $\psi(t) = 1/t$ as gluing profile, we have equipped the compactification $\cM$ with the structure of a smooth manifold with faces, which has a globally defined tangent bundle.  The following is part of \cite[Proposition 7.1]{Large}.

\begin{lem} \label{lem:tangent}
The bundle $T\mathring{\cM}$ extends, canonically up to isomorphism, to define a topological vector bundle $T^{top}\cM$ over the compactification $\cM$.  \par
     
 Fix the gluing profile $\psi(t) = 1/t$. There are distinguished normal directions to boundary strata, and there are isomorphisms $T^{top}\cM \to T\cM$ which are compatible with passing to strata.
      \end{lem}

\begin{proof}
    The first statement follows from the fact that the transition function from \eqref{eqn:topological_tangent} has derivative which is asymptotically lower triangular $\left(\begin{array}{cc}
        1 & 0 \\ Df & 1
    \end{array} \right) $, in particular belongs to a contractible space of invertible matrices. The second statement follows because in the co-ordinates $(u,\bf{r})$ the derivative of the transition function is asymptotically $\left(\begin{array}{cc}
        1 & 0 \\ 0 & 1
    \end{array} \right) $. This shows that the normal directions given by the collar co-ordinates are asymptotically unique (independent of the choice of system of hypersurfaces) as the gluing parameters go to $\infty$, hence are intrinsic given the choice of gluing profile.  It follows that the Floer flow spaces have a distinguished system of normals, in the sense of \cite[Definition 3.13]{PS}.
\end{proof}

Such a system of normals yields preferred systems of collars which extend the system of normals i.e. have the given infinitesimal behaviour along boundary faces.

Fix a co-ordinate neighbourhood of $u$ coming from a centred system of hypersurfaces as before.
\begin{lem}
    The subspace $\{\xi \in u^*TX \, | \, \xi(z_i) \in TH_i \, \forall i\}$ is a complement to $\ker(D_u)$; in particular $D_u$ admits a unique bounded right inverse with image that subspace. Therefore the map
    \[
    \ker(D_u) \to \oplus_j (\nu_{H_j/X})|_{u(z_j)}
    \]
    given by composing evaluation at $z_j$ with the projection $TX \to TX/TH_j =: \nu_{H_j/X}$ is an isomorphism.
\end{lem}

It follows that picking local trivialisations of the normal bundles $\nu_{H_j/X}$ gives a local trivialisation of the analytic index bundle $\Ind$, i.e. the bundle over $\mathring{\cM}$ with fibre $\ker(D_u)$.

Consider a broken solution $u = (u_1,\ldots,u_n)$ lying in a boundary stratum; as usual we have picked representatives in the $\bR$-translation classes of these maps. We define
\[
\ker(D_u) := \oplus_j \ker(D_{u_j})
\]
Now consider the gluing maps $G_{\bf{H}}(u,T)$ or $G_{\bf{H}'}$ from a neighbourhood $V:=\prod_j V_{\bf{H}^j}$ of $(u_1,\ldots,u_n) \in \cM_{x_0x_1}\ldots \cM_{x_{n-1}x_n}$, respectively a corresponding neighbourhood $V'$ for the systems $\bf{H}'$,  associated to systems of hypersurfaces $\bf{H} = \{\bf{H}_j\}$ respectively $\bf{H}' = \{\bf{H}'_j\}$ for the $u_i$.  Shrinking, we can take $V=V'$, so we obtain a transition map
\[
\phi_{\bf{H},\bf{H}'}: V \times (T_0,\infty)^{n-1} \to V \times (T_0,\infty)^{n-1}
\]
The systems of hypersurfaces $\bf H$ or $\bf H'$ induce systems for the glued curves $u_1 \# u_2 \# \cdots \# u_n$ for all sufficiently large gluing parameters $(T_1,\ldots, T_{n-1})$. 
Trivialising the normal bundles to all hypersurfaces thus gives explicit local trivialisations
\[
G_{\bf{H}}^*\Ind_{x_0x_n} \cong \bR^k \times V \times (T_0,\infty)^{n-1}
\]
where $d$ is the index, and hence  transition maps
\[
\chi_{\bf{H},\bf{H}'}: G_{\bf{H}}^*\Ind_{x_0x_n} \to G_{\bf{H}'}^*\Ind_{x_0x_n}
\]

\begin{rmk}\label{rmk:R factor extends}
 A glued curve $u_1 \# u_2 \# \cdots \# u_n$ has a translation vector field $\partial_s$ and these assemble to define a global section of the pullback $\Gamma(G_{\bf{H}}^*\Ind_{x_0x_n})$, a vector bundle over $V\times (T_0,\infty)^{n-1}$.  The section $\partial_s$ converges at infinity to the element 
    \[
    (\partial_{s(u_1)}, \ldots, \partial_{s(u_n)}) \in \oplus_j \ker(D_{u_j}) = \oplus_j \Ind_{x_{j-1}x_{j}}
    \]
    cf. the discussion after equation (12.21) in \cite{Seidel:book}.
\end{rmk}

\subsection{Collars} \label{Sec:collaring}

This section follows \cite[Sections 8.2-8.4]{Large}. We will construct truncations of open moduli spaces which are diffeomorphic to their compactifications; cf. \cite{Ekholm-Smith} where truncations were also used in discussing smooth structures and index bundles on compactified spaces of Floer discs. 

Following Lemma \ref{lem:tangent}, the flow category $\cM$ admits a  (preferred and unique up to contractible choice) system of collars, meaning that if $F$ is a face of some $\cM_{xy}$ which touches $D(F)$ codimension one faces, there are embeddings 
\[
F \times [0,\epsilon)^{D(F)} \hookrightarrow \cM_{xy}
\]
which are compatible with passing to deeper faces, and which restrict on product faces $\cM_{xy} \times \cM_{yz}$ to product collars in the sense of \cite[Lemma 3.18]{PS}.

Let $u = (u_1,\ldots, u_n) \in \cM_{x_0x_1} \times \cdots \times \cM_{x_{n-1}x_n}$.  There is a collar map
\[
\scrC: \cM_{x_0x_1} \times \cdots \times \cM_{x_{n-1}x_n} \times [0,\epsilon)^{n-1} \to \cM_{x_0x_n}
\]
In particular, for any $T \gg 0$ sufficiently large, there is an embedded copy
\[
\cM_{x_0x_1} \times \cdots \times \cM_{x_{n-1}x_n} \times \{T\}^{n-1} \hookrightarrow \cM_{x_0x_n}
\]
associated to a collar co-ordinate $1/T$.

Open subsets of collar neighbourhoods are in the images of gluing maps which, in our atlas of charts, have been defined locally rather than on whole neighbourhoods of strata. One can relate the two constructions asymptotically as follows.

A \emph{system of lengths} for the flow category $\cM^{LK} $ comprises for each $x,y \in \cM^{LK}$ with $|x|>|y|$, and each interior point $u \in \mathring{\cM}_{xy}$, a choice of map $u:Z \to X$ in the $\bR$-translation class of $u$ and a real constant $L(u) > 0$ (defining a continuous bounded below function on $\cM_{xy}$) such that $u([L(u),\infty) \times [0,1])$ respectively $u((-\infty,-L(u)] \times [0,1])$ lie in fixed small neighbourhoods of $x$ respectively $y$ where the almost complex structure is standard. 

In light of our choice of gluing profile, we will write $T_i = 1/\delta_i$ for gluing parameters when determined by associated collar co-ordinates. Following \cite[Section 8]{Large}, we have:

\begin{lem} 
    One can construct compatible systems of lengths and collars, meaning 
      \begin{itemize}
\item if $u=\scrC((u_1,\ldots,u_n), T_1,\ldots, T_{n-1}) $ then 
\[
L(u) = L(u_1) + 2T_1 + L(u_2) + \cdots + 2T_{n-1} + L(u_n); 
\]
\item if $L_i^- = L(\scrC(u_1,\ldots, u_{i}, T_1,\ldots, T_{i-1}))$ and $L_i^+ = L(\scrC(u_{i+1},\ldots, u_n, T_{i+1},\ldots, T_{n-1}))$ then we have uniform convergence, exponentially fast as $T_i \to \infty$, 
\[
u|_{[-L(u), -L(u)+2L_i^-] \times [0,1]} \ \to \ \scrC(u_1,\ldots, u_{i}, T_1,\ldots, T_{i-1}) |_{[-L_i^-,L_i^-]\times[0,1]}
\]
and
\[
u|_{[L(u)-2L_i^+, L(u)]} \to \scrC(u_{i+1},\ldots, u_n, T_{i+1},\ldots, T_{n-1}) |_{[-L_i^+, L_i^+]\times [0,1]}.
\]
\end{itemize}
\end{lem}

More prosaically, this means that the domain of the curve $\scrC((u_1,\ldots,u_n), T_1,\ldots, T_{n-1})$ is obtained from the collection of domains $u_i$ by gluing with lengths $T_i$, where gluing co-ordinates are measured with respect to the prescribed strip-like ends on all domain components. 

\begin{proof}
     Pick a system of collars which extends the canonical  system of normals, in the sense of \cite[Section 3]{PS}. The result then follows from the gluing construction, cf Figure \ref{Fig:gluing}. More precisely we will choose the gluing parameters $T_i$ to be sufficiently large, inductively in the dimension of the stratum.
\end{proof}

    Fix a  gluing parameter $T > T_0$ and consider the embedded submanifold
    \begin{equation}\label{eq: coll comp}
        \cM_{x_0x_1} \times \cdots \times \cM_{x_{n-1}x_n} \times \{T\}^{l-1} \hookrightarrow \cM_{x_0x_n}
    \end{equation}
    which is the image of a collar map $\scrC(u_1,\ldots,u_n; t,t,\ldots t)$ with $t=1/T$.
    
    At a point $u = u_1 \# u_2 \# \cdots \# u_n$ in the image, we can construct a map
    \begin{equation} \label{L2 projection}
    \oplus_j \ker(D_{u_j}) \to \ker(D_u)
    \end{equation}
    by cutting off vector fields translated to be supported in the constituent parts of the glued  strip (cf. Figure \ref{Fig:gluing}; the cut-offs take place in the regions of width $O(1)$ that figure in pregluing of strips) and then composing with 
    $L^2$-orthogonal projection.

We now fix a small $\lambda > 0$ so that any linear map in the $\lambda$-ball around \eqref{L2 projection} is an isomorphism. Again by the argument of Lemma \ref{linear gluing} we have:

   \begin{lem} \label{lem:orthogonal projections with collars}
    If $T \gg T_0$ is sufficiently large, the map of \eqref{L2 projection} is close to an isometry; in particular there is $\lambda>0$ so that any map in the $\lambda$-ball around this map is itself a linear isomorphism.
\end{lem}

Concatenating two collaring maps
\begin{align}
(u_1,\ldots,u_n) \mapsto (u_1,\ldots, u_i, \scrC(u_{i+1},\ldots,u_{i+j},t,t,\ldots,t), u_{i+j+1},\ldots, u_n) \mapsto \\ 
\scrC(u_1,\ldots, u_i, \scrC(u_{i+1},\ldots,u_{i+j};t,t,\ldots,t), u_{i+j+1},\ldots, u_n; t,t,\ldots,t)
\end{align}
yields a pair of isomorphisms
\begin{equation} \label{collaring isomorphisms} 
\oplus_j \ker(D_{u_j}) \to \oplus_{k<i} \ker(D_{u_k}) \oplus \ker(D_{\scrC(u_{i+1},\ldots,u_{i+j})} \oplus_{k>{i+j+1}} \ker(D_{u_k}) \, \to \, \ker(D_u).
\end{equation}

\begin{lem} \label{lem:orthogonal projections are isotopic}
    For $T \gg 0$ the composite of the two maps  of \eqref{collaring isomorphisms} is isotopic to the isomorphism of \eqref{L2 projection} up to contractible choice.
\end{lem}

\begin{proof}
    Consider the map  (with $Z=\bR\times[0,1]$ as usual)
    \[
    \mathrm{splice}: \oplus_j \ker(D_{u_j}) \to L^{2,1}(Z, u^*TX, u^*TL \sqcup u^*TK)
    \]
    given by cutting off translated versions of the vector fields, so \eqref{L2 projection} is given by the composition of `splice' with orthogonal $L^2$-projection.  All of  the maps of the shape \eqref{collaring isomorphisms},  and  the original \eqref{L2 projection}, are isomorphisms $\oplus \ker(D_{u_j}) \to \ker(D_u)$ which, in the space of injections $\oplus \ker(D_{u_j}) \to L^{2,1}(Z,u^*TX, u^*TL \sqcup u^*TK)$, lie in a ball of operator norm radius $<\varepsilon$ about $\mathrm{splice}$, for some small $\varepsilon \ll 1$.  This implies that all these isomorphisms lie in a convex set.
\end{proof}

The exact sequences
\begin{equation} \label{eqn:extended sequence}
0 \to \bR^{l-1} \to \oplus_j\Ind_{x_{j-1}x_j} \to T(\cM_{x_0x_1}\times\cdots\times \cM_{x_{l-1}x_l}) \to 0
\end{equation}
which exist canonically over the interiors of moduli spaces extend as exact sequences of topological vector bundles, compatibly over the various strata (with diagonal inclusions on the first terms, cf. Remark \ref{rmk:R factor extends}).  By construction, the index bundle, and the map from the index bundle to the tangent bundle of the moduli space, is not canonical on the boundary, but defined up to contractible choice: we are using a diffeomorphism from the compactified moduli space to a subspace of the interior, arising from the collars, to \emph{define} the extension of the index bundle on the interior to a bundle on the boundary strata (which circumvents issues arising from the degeneration of the domain and the change in the automorphism group at the boundary). The resulting extensions \eqref{eqn:extended sequence} are constructed inductively over strata, where Lemma \ref{lem:orthogonal projections are isotopic} ensures that one can interpolate previously made choices. In a collar neighbourhood of a boundary stratum the map $\Ind_{xy} \to T\cM_{xy}$  will in general only be homotopic to the canonical analytically defined map on the interior.  

\begin{lem}\label{lem:index breaking}
    There are isomorphism $\Ind_{xy} \cong I_{xy}^{\cM}$ compatible with breaking.
\end{lem}

\begin{proof}
    Using the previously listed ingredients this follows by an induction over $|x|-|y|$ as in \cite[Lemma 8.12]{PS}.
\end{proof}

\subsection{$\Theta$-orientations}\label{Subsec:theta-orientations in Floer theory}

Assume we are given lifts of the stable Gauss map of the Lagrangians $L$ and $K$ to $\Theta$.

\begin{lem} \label{lem:abstract discs on complements of collars}
    There are homotopy equivalent subspaces $\cM_{xy}' \subset \cM_{xy}$ which are compact and contained in $\mathring{\cM}_{xy}$ with the following property:  there are maps $\cM'_{xy} \to \bU^{\Theta}_{xy}$ which are compatible with composition in the flow category, i.e. such that there are commuting diagrams
     \[
    \xymatrix{
    \cM'_{xy} \times \cM'_{yz} \ar[r] \ar[d] & \bU^{\Theta}_{xy} \times \bU^{\Theta}_{yz} \ar[d]_{\rho} \\
     \cM'_{xz} \ar[r]&  \bU^{\Theta}_{xz}
    }
    \]
    where $\rho$ is a gluing map of abstract discs from Lemma \ref{lem:upshot}, and the left vertical map is obtained from (\ref{eq: coll comp}).
    
\end{lem}

\begin{proof}[Sketch]
We fix a compatible system of lengths and collars on $\cM$ in the sense of Section 
\ref{Sec:collaring}, with collar maps 
\[
\scrC: \cM_{x_0x_1} \times \cM_{x_1x_2} \times \cdots \times \cM_{x_{k-1}x_k} \times [0,2\delta)^{k-1} \to \cM_{x_0x_k}.
\]
We will refer to the subcollar on $[0,\delta)^k$ as an `inner collar'.

Define a `restricted interior' $ \cM'_{x_0x_k}:= \mathring{\cM}^{\scrC}_{x_0x_k}$ to be given by removing all collar neighbourhoods of all smaller boundary strata with collar parameters $[0,\delta) \subset [0,2\delta)$. We get a compact subset 
\begin{equation} \label{inner collar boundary}
\cM_{x_0x_1} \times \cM_{x_1x_2} \times \cdots \times \cM_{x_{k-1}x_k} \times \{\delta\}^{k-1} \subset \mathring{\cM}_{x_0x_k}
\end{equation}
of the interior, which is the boundary of the complement of the inner collar.  This inherits the structure of a manifold with faces. There are retractions $\cM_{x_0x_k} \to \mathring{\cM}^{\scrC}_{x_0x_k}$, compatible with gluing so
\[
\xymatrix{
\cM_{x_0x_1} \times \cM_{x_1x_2} \ar[r] \ar[d] & \cM_{x_0x_2} \ar[d] \\ 
\mathring{\cM}^{\scrC}_{x_0x_1} \times \mathring{\cM}^{\scrC}_{x_1x_2} \ar[r]  & \mathring{\cM}^{\scrC}_{x_0x_2}
}
\]
commutes, where the bottom arrow is \eqref{inner collar boundary}.

We take the lengths of shrunken strip-like ends to be determined by a constant function $L = 1/\delta$ on the moduli spaces $\cM_{x_ix_j}$ for $\delta$ sufficiently small that the conclusions of Lemma \ref{lem linearised discs as abstract discs} hold. These then admit maps $a_{x_ix_j}$ to the space of abstract discs with $\Theta$-lifts and (trivial) perturbation data.

We then apply the gluing of abstract discs. To obtain compatibility with composition in the flow category, we work inductively over the length $|x|-|y|$ of a broken trajectory between $x$ and $y$.  For an unbroken face, this is the end step of the construction.

Choices for $\cM'_{xy}$ and $\cM'_{yz}$ induce a choice on $\cM'_{xy} \times \cM'_{yz}$. We view this as a subspace of $\cM'_{xy}$ via the $\delta$-level set of the collaring map.  The resulting map, when the length parameters are all large enough, is homotopic to the choice already induced on this subspace by the choice made on (the interior of) $\cM_{xy}$, by a contractible choice. We interpolate between these in the collars of parameters $[\delta,2\delta]$.
\end{proof}

\begin{lem}\label{lem:abstract ind on complements of collars}
    The maps constructed in Lemma \ref{lem:abstract discs on complements of collars} are covered by associative isomorphisms $a_{xy}^*\bV^{\Theta} \cong I_{xy}$.
\end{lem}

\begin{proof}
    This follows from the construction together with Lemma \ref{lem:orthogonal projections are isotopic} and resulting Lemma \ref{lem:index breaking}. 
\end{proof}

\begin{cor}\label{cor: Theta or Floer}
    $\cM$ admits a natural $\Theta$-orientation.
\end{cor}
\begin{proof}

    We proceed inductively in $|x|-|y|-1$. Starting when $|x|-|y|=0$, for each $x,y$, we choose deformation retractions $p_{xy}: \cM_{xy} \to \cM'_{xy}$ and isomorphisms $p_{xy}^* I^{\cM'}_{xy} \cong I^\cM_{xy}$, both compatible with gluing. Then pull back the $\Theta$-orientations from Lemmas \ref{lem:abstract discs on complements of collars} and \ref{lem:abstract ind on complements of collars}.
\end{proof}

The construction of $\Theta$-orientations for morphisms and bilinear maps follows the same template. The smooth structures on moduli spaces are now associated to charts in which local co-ordinates come from a mixture of `intersection co-ordinates with hypersurfaces' and `neck lengths'. The index bundles are defined using extended linearised operators incorporating the tangent space to the moduli space of domains where necessary (e.g. for constructing associators from moduli spaces of discs with 3 inputs and 1 output).  This does not represent a substantial difference, but means that spaces of abstract discs $\bU_{3,1}^{V,\Theta}$ with non-trivial stabilising vector space $(V,f)$ now appear, since some stabilisation is needed to make the Cauchy-Riemann operators simultaneously surjective at all points of the moduli space. 

\begin{rmk}[Stabilisation] \label{rmk: Floer N inf}
    In the above, we assumed we were working with a tangential structure of rank $N < \infty$. Given a tangential structure $\Theta$ of rank $N=\infty$, we may take the tangential structure $\Theta(N)$ of rank $N$ obtained by pullback along the inclusion $\widetilde{U/O}(N) \to \widetilde{U/O}$. For $N$ sufficiently large, any Lagrangian admitting a $\Theta$-structure also admits a $\Theta(N)$-structure, canonically up to homotopy. Given another such Lagrangian $K$, we may then use the procedure above to construct a $\Theta(N)$-orientation on $\cM^{LK}$, which induces a $\Theta$-orientation on $\cM^{LK}$ by stabilising.
\end{rmk}

\subsection{Adaptations for tangential pairs}
    In this section, we extend the construction of $\Theta$-orientations from Section \ref{Subsec:theta-orientations in Floer theory} for a tangential structure $\Theta$ to the setting of tangential pairs. Let $(\Theta, \Phi)$ be a tangential pair of rank $N<\infty$. Instead of assuming $X$ is stably framed, we choose a classifying map $X \to BU(N)$ for $TX \oplus \bC^{N-d}$ for some $N \gg 0$, and for each Lagrangian $L \subseteq X$, choose a classifying map $L \to BO(N)$ for $TL \oplus \bR^{N-n}$, compatible with that of $X$.

    Assume we are given lifts of the classifying map for $TX$ to $\Phi$, and compatible lifts of the classifying maps for $TL$ and $TK$ to $\Theta$, meaning the following diagram commutes:
    \begin{equation}
        \xymatrix{
            L
            \ar[r]
            \ar[d]
            &
            X
            \ar[d]
            \\
            \Theta 
            \ar[r]
            &
            \Phi
        }
    \end{equation}
    and similarly for $K$. Then in the same way we used the evaluation along the boundary of $Z:=\bR \times [0,1]$ to pull back a totally real subspace of $\bC^N$ to one along $\partial Z$ along with a lift to $\Theta$, we may use the evaluation along the whole of $Z$ to pull back the complex vector bundle $E$ to $Z$ along with the lift to $\Theta$. Combined with the same proof as that of Lemma \ref{lem:abstract discs on complements of collars} and Corollary \ref{cor: Theta or Floer}, we obtain a $(\Theta, \Phi)$-orientation on the flow category $\cM^{L,K}$. 

    We may incorporate tangential pairs of rank $N=\infty$ exactly as in Remark \ref{rmk: Floer N inf}.

\section{The $(\Theta,\Phi)$-oriented spectral Fukaya category}

Fix an oriented tangential pair $(\Theta,\Phi)$ of rank $N\leq\infty$.  When $X$ admits a $\Phi$-structure, our goal is to define
a spectral Donaldson-Fukaya category $\scrF(X;(\Theta,\Phi))$ whose objects are $\Theta$-oriented exact Lagrangian submanifolds which are either compact or cylindrical at infinity.  Throughout this section, we work with this tangential pair.

This section follows \cite[Section 7]{PS} with appropriate adaptations to reflect the preceding constructions of $\Theta$-orientations in place of framings. To minimise repetition, we will focus exposition on the places where the arguments are substantively  different.

\subsection{Bordism of proper maps} \label{Sec:proper bordism}

In this section we discuss some aspects of bordism of non-compact manifolds, which is required for dealing with non-compact manifolds.

For the purposes of this section, we are unspecific about the particular flavour of bordism we are considering, since the argument applies \emph{mutatis mutandis} quite generally.

The bordism groups $\Omega_*(X)$  of a topological space $X$ are usually defined by considering bordism classes of maps $M \to X$ where $M$ is a compact manifold. Relative bordism groups $\Omega_*(X,Y)$ consider maps of compact manifolds with boundary $(M, \partial M) \to (X,Y)$. One can also define groups $\Omega_*^{nc}(X)$ by considering proper maps to $X$ of not necessarily compact manifolds.

The one-point compactification $X^+$ of a space $X$ has a distinguished point $\infty \in X^+$ at infinity. It is typically poorly behaved unless the point $\infty \in X^+$ admits a cone neighbourhood, and hence a contractible neighbourhood base.  This is certainly true if $X$ is the interior of a manifold with boundary, which is the only case we shall require.

\begin{lem}
    Suppose $X$ is a non-compact manifold of finite type. There is an isomorphism $\Omega_*^{nc}(X) \simeq \Omega_*(X^+,\infty)$.
\end{lem}

\begin{proof}
    Let $f: M \to X$ be a proper map with domain a non-compact manifold $M$. 
    Let $h: X \to \bR_+$ be a proper map. Let $M' = f^{-1}(\{ h \leq C\})$ where $C$ is large and generic, so a regular value for $h$. Then $M'$ is a compact manifold with boundary. Its boundary is sent under $f$ to a neighbourhood of infinity in $X$, so it's sent to near $\infty$ in the one-point  compactification $X^+$. After homotoping so the boundary maps entirely to infinity, this gives an element of $\Omega_*(X^+,pt)$. In the other direction, if $f: M \to X^+$ is a map from a compact manifold whose boundary is sent to $\infty$, then setting $M' = M \setminus f^{-1}(\infty)$, the induced map $M' \to X$ is a proper map.
\end{proof}

Most of the constructions in this paper are the same for compact or for non-compact geometrically bounded Lagrangians. One delicate point in constructing open-closed maps is addressed in Lemma \ref{lem:proper bordism}.

One can similarly define proper left modules/capsules over a non-compact $X$, in the same way as in Section \ref{sec: rel flow mod} but now requiring the manifolds to be proper over $X$ instead of compact.

\subsection{Objects and morphisms}

Let $X$ be a Liouville manifold of real dimension $2d$.   Let $(\Theta,\Phi)$ be an oriented tangential pair of rank $N\leq\infty$, where $N > d+2$,  and fix a homotopy class of $\Phi$-structure on $X$, meaning a lift 
\[
f: X \to \Phi \to BS_\pm U(N) \to BU(N)
\]
to $\Phi$ of the classifying map of $\bC^{N-d} \oplus TX$ as a complex vector bundle. 
 Let $(\Theta,\Phi)$ be an oriented tangential pair of rank $N$. Recall that a $\Theta$-structure on a Lagrangian $L \subset X$, with respect to the given $\Phi$-structure on $X$, is a lift of the classifying map for  its tangent bundle in the diagram
 \begin{equation}
            \xymatrix{
                L 
                \ar[r]^{\eta}
                \ar[d]
                &
                \Theta
                \ar[r]
                \ar[d]
                &
                B\operatorname{Pin}(N)
                \ar[d]
                \\
                X
                \ar[r]
                &
                \Phi
                \ar[r]
                &
                BS_\pm U(N)
            }
        \end{equation}

The objects of the category $\scrF(X;(\Theta,\Phi))$ are pairs  $(L,\eta)$ where $L$ is a geometrically bounded exact Lagrangian submanifold, and $\eta$ is such a $\Theta$-lift of the Gauss map. As usual we will usually drop $\eta$ from the notation. Recall that for a given Lagrangian $L$, the set of $\Theta$-structures on $L$ was described in Proposition \ref{prop: class brane str}.

Given two objects $(L,\eta_L)$ and $(K,\eta_K)$, we construct a $(\Theta,\Phi)$-oriented flow category $\cM^{LK}$ as follows. First, we pick regular Floer data $(H_t,J_t)$, so in particular $\phi_{H_t}^1(L) \pitchfork K$. We always assume that the flow of $H_t$ is a small-time Reeb flow at infinity, so we are doing `infinitesimal wrapping'. In particular, even though  $L$ and $K$ may both be non-compact, the transverse intersection $\phi_{H_t}^1 L\pitchfork K$ comprises a finite set. Then
\begin{enumerate}
    \item The objects of $\cM^{LK}$ are the time-1 Hamiltonian chords $x(t)$ from $L$ to $K$ with respect to $H_t$.
    \item The Gauss lifts $\eta_L$ and $\eta_K$ define both gradings and Pin structures on $L,K$ via the lifts of $\eta_L$ and $\eta_K$ to the homotopy fibre $\widetilde{U/O}^{or}(N)$ of $B\operatorname{Pin}(N) \to BS_\pm U(N)$. With respect to these preferred lifts any object $x$ has a grading $|x|$. The further lift of the Gauss map to $\Theta$ then defines $(\Theta,\Phi)$-oriented puncture data $\bE_x$ at $x$ of grading $|x|$.
    \item Since the Floer data is regular, we can take the stabilising vector spaces $V_{xy} = 0$ for all $x,y$.
    \item The moduli spaces $\cM^{LK}_{xy}$ are the usual Floer moduli spaces, which come with maps $\rho_{xy}: \cM_{xy} \to \bU_{xy}^{(\Theta,\Phi)}$ as in Section \ref{Subsec:theta-orientations in Floer theory}.
\end{enumerate} 

\begin{defn}
   We define the morphism space 
\[
\scrF_i(L,K;(\Theta,\Phi)) := \Omega_i^{(\Theta,\Phi)}(\cM^{LK})
\]
to be the set of bordism classes of $(\Theta,\Phi)$-oriented right flow modules of degree $i$.   
\end{defn}

Recall from Remark \ref{rmk:not a morphism group} that these are not  morphism groups in $\mathrm{Flow}^{(\Theta,\Phi)}$.

Given three branes $L, K, N$ with $\Theta$-lifts, and objects $x\in \cM^{LK}$, $y \in \cM^{KN}$ and $z \in \cM^{LN}$, we consider the moduli space $\cM_{xy;z}^{LKN}$ of holomorphic triangles, as in \cite[Section 7.3]{PS}. Assuming genericity and hence regularity, so all operators are stabilised by the trivial vector space, this has a forgetful map $\cM_{xy;z} \to \bU_{2,1}^{0,(\Theta,\Phi)}$, and admits a $(\Theta,\Phi)$-orientation.  The usual description of boundary strata shows that $\cM^{LKN}$ defines a bilinear map 
\[
\cM^{LK} \times \cM^{KN} \to \cM^{LN}
\]
and the composition law in $\scrF(X;(\Theta,\Phi))$ then comes from the induced map
\begin{equation} \label{eqn:composition}
\Omega^{(\Theta,\Phi)}_i(\cM^{LK}) \otimes \Omega^{(\Theta,\Phi)}_j(\cM^{KN}) \to \Omega^{(\Theta,\Phi)}_{i+j}(\cM^{LN})
\end{equation}
described in \eqref{eqn:composition in general} (and its direct analogue for tangential pairs).

Exactly as in \cite{PS}, but using $(\Theta,\Phi)$-orientations rather than framings, the moduli spaces of four-marked discs (which will now parametrise Cauchy-Riemann operators stabilised by non-trivial vector spaces $V$) define associators which show that the composition law \eqref{eqn:composition} is associative.

\begin{rmk}\label{rmk:endomorphisms}
    Generalising arguments from \cite[Section 7.10]{PS}, we expect that endomorphisms $\scrF_*(L,L)$ of $L$ recovers $\pi_{*+d}(L \wedge \bT_{01}(\Theta,\Phi))$, equivalently the $F$-oriented bordism groups of $L$.

    In the case $\Theta$ and $\Phi$ are contractible (corresponding to the spectral Fukaya category considered in \cite{PS}), this has been proved by Blakey \cite[Propositions 4.9 \& 6.8]{Blakey}.
\end{rmk}

\subsection{Units and open-closed maps}\label{sec: unit OC}

For a single $L$ we have a flow-category $\cM^{LL}$. The moduli spaces of perturbed holomorphic half-planes $\cM_{*x}$ asymptotic to a chord $x$, together with the linearisation maps $\cM_{*x}\to \bU_{0,1}^{(\Theta,\Phi)}$ and associated $(\Theta,\Phi)$-orientations, define a distinguished right module
\[
e_L \in \Omega_0^{(\Theta,\Phi)}(\cM^{LL}) = \scrF_0(L,L;(\Theta,\Phi)).
\]

\begin{lem}
    $e_L$ is idempotent.
\end{lem}
\begin{proof}
    This follows from gluing, directly analogous to \cite[Lemma 7.10]{PS}.
\end{proof}

\begin{cor}
    $\scrF(X;(\Theta,\Phi))$ is unital, with $e_L$ being the unit for $L$.
\end{cor}
\begin{proof}
    This follows from Proposition \ref{prop: flow white head} as in the arguments around \cite[Corollary 7.12]{PS}. Namely, one shows that product with $e_L$ is surjective; that implies that right or left multiplication by $e_L$ is an isomorphism over any field; and one then lifts that quasi-isomorphism statement through the truncations to infer that multiplication by $e_L$ is an isomorphism on bordism of flow-modules.
\end{proof}

Analogously, the moduli spaces $\cM_{x*}$ of discs with a single input at a chord $x$ for $L$, and their $\Theta$-orientations, define a canonical (possibly noncompact) left module  $\Upsilon_L \in  \Omega^d_{(\Theta,\Phi)}(\cM^{LL})$, with $d = \dim_{\bR}(L)$. 

\begin{rmk}
Throughout the rest of this section, we work with compact left modules (i.e. ones for which the moduli spaces $\cM_{\ast x}$ are compact manifolds with corners) if $L$ is compact, and with proper left modules (i.e. ones whose underlying moduli spaces are proper over $X$) if $L$ is allowed to be noncompact and geometrically bounded. 
    \end{rmk}

In \cite[Section 7.6]{PS} and Definition  \ref{defn:over a space} above, we introduced the  notion of a (left or right)  module over a flow category `living over' a pair $(Z,E)$ comprising a target topological space $Z$ equipped with a complex vector bundle $E \to Z$ and a $\Phi$-lift $f: Z \to \Phi \to BU$ of the classifying map of $E$. 

\begin{lem} \label{lem:proper bordism}
    Fix an interior marked point in the disc with one incoming puncture. 
    
    The evaluation map $\cM_{x \ast} \to X$ is proper, and the module defining $\Upsilon_L$ lives over $(X,f)$.
\end{lem}

\begin{proof}
    If $L$ is compact, properness is immediate. If $L$ is cylindrical an additional argument is required since the moduli spaces are defined with respect to a Hamiltonian of small negative slope. Fix an auxiliary Riemannian metric on $L$, for instance that induced by the metric associated to $\omega$ and $J$. Stokes' theorem bounds the area of discs in $\cM_{\ast x}$ in terms of the action of the chord $x$. The reverse isoperimetric inequality \cite{Groman-Solomon} then shows that the boundary length of the disc is uniformly controlled. If we restrict to a compact subset of $X$ and hence a compact subset of $L$, this means that discs stay within a bounded region of the cylindrical end, which implies the properness statement. 
    
    That the module then lives over $(X,f)$ is then directly implied by the construction of the $\Theta$-orientation on the spaces $\cM_{\ast x}$.
\end{proof}

The pair $(X,f)$ define a capsule bordism theory
\[
\Omega^{\Theta,\Phi,\circ}_*(X,f).
\]
(even if $\Phi = \{pt\}$ as in the stably framed case, this theory depends on the particular homotopy class of stable framing of $TX$).

A $\Theta$-oriented Lagrangian $L\subset X$ comes with a canonical map $\eta: L \to \Theta$ which lifts the classifying map of $TL$. Composing with the inclusion of constant loops $i: \Theta \to \mathcal{L}\Theta$, and taking $V_{**} = \{0\}$, this then defines a capsule with data
\[
\cM_{**} = L,\,  i\circ \eta: L \to \mathcal{L}\Theta \simeq \bU_{**}, \, V_{**} = 0.
\]
In particular $L$ (equipped with a specific $\Theta$-orientation $\eta$) has an associated fundamental class 
\[
[L,\eta] \in \Omega^{(\Theta,\Phi),\circ}_d(X,f)
\]
which we may abbreviate to $[L]$ when the choice of $\eta$ is clear or implicit. 
\begin{lem}\label{lem:OC}
    The composition of the unit and counit for $L$ define a $\Theta$-oriented capsule $\cC_{**}$ living over $X$, which is bordant to $[L,\eta]$. 
\end{lem}

\begin{proof}
The usual gluing argument, incorporating $\Theta$-orientations, implies that the composition is bordant to a capsule living over $X$, associated to a moduli space of maps from an unpunctured disc with one interior marked point and a non-trivial Hamiltonian term $Y$.  One then considers a moduli space of pairs $(u,t)$ where $u$ is a mapping of a holomorphic disc satisfying the  perturbed Floer equation with Hamiltonian term $t\cdot Y$ for $t\in [0,1]$. This defines a bordism  from the given capsule to a space of constant discs (by exactness). The entire bordism lives over $X$, with properness of evaluation following from Lemma \ref{lem:proper bordism}.
\end{proof}

\begin{lem}
    If $L$ and $K$ are quasi-isomorphic in $\scrF(X;(\Theta,\Phi))$ then $[L,\eta_L] = [K,\eta_K] \in \Omega^{(\Theta,\Phi),\circ}_d(X,f)$.
\end{lem}

\begin{proof}
    Given Lemma \ref{lem:OC}, we construct a bordism living over $X$ from the choice of quasi-isomorphism, as in \cite[Corollary 7.17]{PS}.
\end{proof}

\subsection{Comparison to classical Floer theory}

Recall that since our tangential pair is both graded and oriented, we have: 

\begin{lem}
    A choice of $\Theta$-structure for $L$ induces a Pin structure on $L$, canonical up to homotopy.
\end{lem}

\begin{proof}
    This follows from the definition of $\Theta$ being oriented. 
\end{proof}

\begin{lem}
    For a pair $L,K$ of $\Theta$-oriented branes, the classical Floer complex $CF(L,K)$ is canonically $\bZ$-graded.
\end{lem}

\begin{proof}
    At any intersection point $x$ of $L,K$ we have a unique homotopy class of path between the elements of $\widetilde{U/O}(N)$ associated to the $\Theta$-orientations, and this path has a well-defined integer-valued Maslov index.  
\end{proof}

\begin{lem}
    The Morse complex $CM_*(\cM^{LK})$ is canonically identified with the classical Floer complex $CF(L,K)$ up to a global sign. Given $\Theta$-oriented $L,K,N$, under this identification holomorphic discs with 3 punctures define a bilinear map $\cR: \cM^{LK} \otimes \cM^{NL} \to \cM^{NK}$ which induces the usual Floer triangle product (up to global sign).
\end{lem}

\begin{proof}
    As in \cite[Section 7.8]{PS} this is largely tautological. A sign arises because the Morse complex of the flow category involves orientation lines of index bundles $I_{xy} = \bR \oplus T\cM_{xy}$ rather than orientation lines of the tangent bundles to moduli spaces of flow lines themselves.
\end{proof}

Since the Morse complex is built precisely from information held in the zero-th truncation, one finds:

\begin{cor}\label{cor: 0-trun fuk}
    There is a fully faithful functor $\tau_{\leq 0}\scrF(X;(\Theta,\Phi)) \to \scrF(X;\bZ)$ whose image is the subcategory of graded Lagrangians which admit a $\Theta$-orientation with respect to the given $\Phi$-structure on $X$.
\end{cor}

\begin{rmk}
In our conventions, $HF_*(L,K)$ is nonpositively graded, has a unit in degree zero, and has a degree-preserving product, so $HF_*(L,L) \cong H^{-*}(L)$; this is consistent 
with \cite[Proposition 7.27]{PS}. 
\end{rmk}

\subsection{Choices and naturality}

For any oriented tangential pair $(\Theta,\Phi)$ there is a subgroup of the symplectic mapping class group of $X$ which preserves the $\Phi$-structure on $X$ up to homotopy, and hence can potentially act on $\scrF(X;(\Theta,\Phi))$. In general, this subgroup will depend on the particular $\Phi$-structure on $X$.  Here we spell this out in the key case of a stable framing or polarisation, where there is a connection with classical $K$-theoretic invariants of $X$.

Suppose $X$ admits a (stable) polarisation, so $TX \oplus (E \otimes_\bR \bC) = \bC^N$ for some vector bundle $E \to X$ and some $N$. A choice of polarisation is the same as a choice of $\Phi$-structure on $X$, where $\Phi = BO$ and the map $\Phi \to BU$ is the natural complexification map.

\begin{lem}
    The set of choices of polarisation  on $X$ form a torsor for $[X,U/O] = KO^1(X)$. If $X$ admits a stable framing, the set of choices of stable framing form a torsor for $[X,U] = KU^1(X)$.
\end{lem}

\begin{proof}
    A choice of polarisation is equivalent to a choice of section of the stable Lagrangian Grassmannian over $X$, which is the pullback of the universal principal $U/O$ bundle over $B(U/O)$ under the natural map $X \to B(U/O)$.  Two polarisations therefore differ by a map $X \to U/O$; the final statement comes from real Bott periodicity. The case of framings is analogous, considering lifts in the sequence $U \to \ast \to BU$.
\end{proof}

If $\phi: X \to X$ is a symplectic diffeomorphism, then one can compare a polarisation (or framing) with its pullback by $\phi$.  There is a space $\Symp^{\Phi}(X)$ of symplectic diffeomorphisms equipped with a homotopy between these two stable polarisations / framings.  Analogous to \cite{Seidel:graded} we obtain:
\begin{lem}
If $X$ is stably polarised by $\Psi$, there is an exact sequence
\[
1 \to KO^0(X)  \to \pi_0\Symp^{\Psi}(X) \to \pi_0\Symp(X) \to KO^1(X)
\]
where the final map is not a homomorphism (but its kernel is a subgroup). If $X$ is stably framed via $\Psi$, there is an analogous exact sequence
\[
1 \to KU^0(X)  \to \pi_0\Symp^{\Psi}(X) \to \pi_0\Symp(X) \to KU^1(X)
\]
(with the same caveat on the final map).
\end{lem}

\begin{cor}
    Suppose $H^1(X;\bR) = 0$. For $\Phi \to BU$ the complexification $BO \to BU$, the spectral Fukaya category $\scrF(X;(\Theta,\Phi))$ is natural under $\pi_0\Symp^{\Psi}(X)$.
    \end{cor}

    (The hypothesis $H^1(X;\bR) = 0$ circumvents the distinction between symplectic and Hamiltonian isotopy.) 
In principal there may be symplectic mapping classes which do not lift to this subgroup for any choice of framing or polarisation data $\Psi$.

Given a stable polarisation, for a Lagrangian $L\subset X$ there is a stable Gauss map classifying the bundle $TL+E$ which now takes values in $U/O(N)$. Given a graded tangential structure $\Theta \to \widetilde{U/O}(N)$,  the set of Lagrangians which admit a $\Theta$-orientation  will thus depend on both the polarisation $E$ and the choice of graded tangential structure.

\begin{rmk}
    There are analogous results describing the choice of $\Theta$-structure on a given connected Lagrangian submanifold $L$; for instance, taking $\Theta=\{\ast\}$, two different nullhomotopies of $L \to U/O$ differ by a map $[\Sigma L, U/O] = KO^0(L)$ (for $N$ sufficiently large). 

    Similarly for a tangential pair $(\Theta, \Phi)$, if both spaces are infinite loop spaces and all four maps in (\ref{eq: strong intro}) are maps of infinite loop spaces, a similar argument shows that the set of $\Theta$-structures on a given $L$ is a torsor over $KO^0(L)$.
\end{rmk}

\subsection{Proofs of applications}

The theory developed in Section \ref{sec: ob lift} applies in this setting to give conditions on which quasi-isomorphism over $\Theta$ is detected by quasi-isomorphism over $\bZ$.

Suppose $L$ and $K$ are $\Theta$-branes whose underlying $\bZ$-branes are quasi-isomorphic. 
We begin by choosing an isomorphism $\alpha^0 \in HF_0(L,K;\bZ)$ in $H_*(\scrF(X;\bZ))$. 
By Corollary \ref{cor: 0-trun fuk}, we can view this as an isomorphism in $\tau_{\leq 0}\scrF_0(X; (\Theta,\Phi))$. 
\begin{lem}\label{lem: inv det tau0}
    For $n \geq 0$, an element $\beta^n \in \scrF_n(L,L; (\Theta,\Phi))$ is invertible if and only if $\tau_{\leq 0} \beta^n$ is.
\end{lem}
\begin{proof}
    Same as \cite[Corollary 7.13]{PS}, using Proposition \ref{prop: flow white head}.
\end{proof}
\begin{lem}
    Let $n \geq 0$, and let $\alpha^{n+1} \in \tau_{\leq n+1}\scrF_0(L,K; (\Theta,\Phi))$. If $\alpha^0 := \tau_{\leq 0}\alpha^{n+1}$ is an isomorphism, then $\alpha^{n+1}$ is an isomorphism. 
    
\end{lem}
\begin{proof}
    By Proposition \ref{prop: flow white head}, composition with $\alpha^{n+1}$ is surjective. Therefore there is some $\beta^{n+1} \in \tau_{\leq n+1}\scrF_0(K,L;(\Theta,\Phi))$ such that $\alpha^{n+1}\beta^{n+1}=Id$. This implies that $\tau_{\leq 0} \beta^{n+1}$ is a 1-sided inverse to $\tau_{\leq 0}\alpha^{n+1}$; since the latter is an isomorphism it follows that $\tau_{\leq 0}\beta^{n+1}$ is also an isomorphism. Therefore by Lemma \ref{lem: inv det tau0}, both $\alpha^{n+1}\beta^{n+1}$ and $\beta^{n+1}\alpha^{n+1}$ are units.\par 
    Let $\tilde\beta^{n+1} = \beta^{n+1} \left(\alpha^{n+1} \beta^{n+1}\right)^{-1}$. Then $\alpha^{n+1} \tilde \beta^{n+1} = Id$ by construction, and $\gamma^{n+1} := \tilde\beta^{n+1}\alpha^{n+1}$ is idemponent, but is also a unit. Idempotent units are always the identity. So $ \alpha^{n+1}$ and $\tilde \beta^{n+1}$ are inverse isomorphisms.
\end{proof}

This brings us into the situation discussed in the Introduction. The following completes the proof of Theorem \ref{thm: fuk ob}.

\begin{lem}
    Suppose $L$ and $K$ are branes in $\scrF(X;(\Theta,\Phi))$ and $L \simeq K$ in $\scrF(X;\bZ)$. If $H^i(L;\Omega^{(\Theta,\Phi)}_{i-1}) = 0$ for every $i$, then $L \simeq K$ in $\scrF(X;(\Theta,\Phi))$.
\end{lem}

\begin{proof}
Assume we are given an isomorphism $\alpha^n \in \tau_{\leq n}\scrF_0(L,K; (\Theta,\Phi))$ for some $n \geq 0$.
Recall that by Theorem \ref{thm: ob lift}, the obstruction to finding an extension $\alpha^{n+1}$ of $\alpha$ such that $\tau_{\leq n} \alpha^{n+1} = \alpha^n$ is a class
\begin{equation*}
    [\cO]=[\cO(\alpha^n)] \in HM_{-n-2}\left(\cM^{LK};\Omega_0^{(\Theta,\Phi)}(\cE_{-n-1})\right) = H_{d-n-2}(L;\Omega_{n+1}^{(\Theta,\Phi)}) = H^{n+2}(L,\Omega_{n+1}^{(\Theta,\Phi)}).
\end{equation*}
where the grading structure on $L$ gives it a canonical orientation which is used in the final duality isomorphism. 
Under the given hypotheses, the obstructions all vanish for degree reasons.
  \end{proof}

\begin{proof}[Proof of Corollary \ref{cor:1}]
    This follows from the previous result, the fact that $MU_*$ is concentrated in even degrees, and the hypothesis that $H^*(L;\bZ)$ is also concentrated in even degrees. In this case, the classes $[\cO]$ all vanish for degree reasons. To obtain a conclusion for fundamental classes in $MU$-homology rather than in capsule bordism, one appeals to the `splitting result' from Lemma \ref{lem:MU splitting}.
\end{proof}

\begin{proof}[Proof of Corollary \ref{cor:immersed general}]
The Gromov-Lees $h$-principle implies that the classification of exact Lagrangian submanifolds up to \emph{immersed} Lagrangian cobordism is a purely homotopy-theoretic problem, which was studied in detail by Eliashberg \cite{Eliashberg},  Audin \cite{Audin} and Rathel-Fournier \cite{R-F}. Since we assume $X$ is stably framed, the stable Lagrangian Grassmannian over $X$ is exactly $X \times U/O$. There is a natural map $\eta: U/O \to BO \times \bZ$, and immersed Lagrangian cobordism is governed by the bordism theory associated to $\mathrm{Thom}(\eta)$; similarly for immersed oriented Lagrangian cobordism, starting with $U/SO$, and for graded immersed Lagrangian cobordism starting with $\widetilde{U/SO}$. The corresponding bordism theory $\Omega^{\eta}_*$ corresponds in our setting to taking $\Phi = \{pt\}$ and $\Theta = Sp/U \to \widetilde{U/O}$.  The conclusion of the Corollary then relies on the splitting $\Omega^{Sp/U,\circ}_* \to \Omega^{\eta}_*$ from Lemma \ref{lem:Sp/U splitting}, along with \cite[Theorem 1.1]{R-F}.
\end{proof}

\begin{ex} \label{ex:U/O branes}
    Lagrangian homotopy spheres admit $\Theta$-brane structures for $\Theta = Sp/U=B({U/O})$ in dimensions $8k+i$ with $i\not\in\{1,3,5\}$. (This follows from combining vanishing results for  $\pi_i(U/O)$ with lifting results in the fibration $B(U/O) \to U/O \to U$.)
\end{ex}

\begin{proof}[Proof of Corollary \ref{cor:2}]
Any such dimension $d$ satisfies:
\begin{itemize}
    \item $d$ is not in $\{1,3,5\}$ mod 8;
    \item $d$ is not a sum of odd integers not of the form $2^j-1$
\end{itemize}
In fact, an application of the Chicken McNugget theorem shows that the list in Corollary \ref{cor:2} consists of all such $d$ \footnote{We are grateful to Danil Koževnikov for pointing this out to us.}.

We now take the tangential structure $Sp/U \to U/O$. Example \ref{ex:U/O branes} shows that the spheres $L$ and $K$ admit $Sp/U = B(U/O)$-orientations. The homotopy groups of the Thom spectrum $\mathrm{Thom}(\eta)$ for $\eta: B(U/O)\to BO\times \bZ$ are purely 2-torsion, with a generator in each odd degree $\neq 2^i-1$, cf. \cite{Audin, Smith-Stong}.  In particular, the dimension hypotheses ensure that the obstructions encountered in degrees $d-1$ vanish automatically. This enables us to construct an immersed not-necessarily-orientable Lagrangian cobordism; since we are already in the non-orientable setting, we can then replace this by an embedded non-orientable cobordism using Lagrange surgery.
\end{proof}

\bibliographystyle{amsalpha}
\bibliography{Refs.bib}{} 

\end{document}